\documentclass[11pt,reqno]{amsart}
\usepackage{amsmath,amsthm,amssymb,amscd}
\usepackage{mathrsfs,mathtools,color}

\usepackage{hyperref}

\usepackage{color}
\usepackage{graphicx}
\usepackage[margin=1in]{geometry}
\usepackage{mathabx}
\usepackage{scalerel}
\usepackage{mathrsfs}
\usepackage{tikz}

\definecolor{darkgreen}{rgb}{0.0, 0.6, 0.13}
\definecolor{airforceblue}{rgb}{0.0, 0.30, 0.69}

\newtheorem{theorem}{Theorem}%[section]
\newtheorem{definition}{Definition}[section]
\newtheorem{lemma}{Lemma}[section]
\newtheorem{corollary}{Corollary}[section]
\newtheorem{remark}{Remark}[section]
\newtheorem{proposition}{Proposition}

\newcommand{\norm}[1]{\left\lVert#1\right\rVert}

 \theoremstyle{remark}
 
 \numberwithin{equation}{section}

\def\sqw{\hbox{\rlap{\leavevmode\raise.3ex\hbox{$\sqcap$}}$%
\sqcup$}}
\def\findem{\ifmmode\sqw\else{\ifhmode\unskip\fi\nobreak\hfil
\penalty50\hskip1em\null\nobreak\hfil\sqw
\parfillskip=0pt\finalhyphendemerits=0\endgraf}\fi}

\begin{document}

\title[On the transition of the Rayleigh-Taylor instability]{On the transition of the Rayleigh-Taylor instability  in 2d water waves}

\author[Qingtang Su]{Qingtang Su}
\address{Department of Mathematics, University of Southern California, Los Angeless, CA,  90089, USA}
\email{qingtang@usc.edu}

\subjclass[2010]{}
\thanks{}

\begin{abstract}

    In this paper we prove the existence of water waves with sign-changing Taylor sign coefficients, that is, the strong Taylor sign holds initially, while breaks down at a later time, and vice versa. Such a phenomenon can be regarded as the transition between the stable and unstable regime in the sense of Rayleigh-Taylor of water waves. As a byproduct, we prove the sharp wellposedness of 2d water waves in Gevrey-2 spaces. 

\end{abstract}

\maketitle
\setcounter{tocdepth}{2}

\tableofcontents

\section{Introduction}
\subsection{The background and the main result}
We consider the motion of an inviscid and incompressible ideal fluid with a free surface in two space dimensions (that is, the interface separating the fluid and the vacuum is one dimensional), such as the surface waves in the ocean. We refer such fluid as \emph{water waves}. We denote the fluid region by $\Omega(t)$, with a free interface $\Sigma(t)$. The equations of motion are Euler's equations, coupled to the motion of the boundary, and with vanishing boundary condition for the pressure. It is assumed that the fluid region is below the air region. Assume that the density of the fluid is
1, the gravitational field is $-\vec{k}$, where $\vec{k}$ is the unit vector pointing in the upward vertical direction. In two dimensions, if the surface tension is zero, then the motion of the fluid is described by
\begin{equation}\label{vortex_model}
\begin{cases}
\begin{rcases}
v_t+v\cdot \nabla v=-\nabla P-(0,1)\\
div~v=0,\\
\end{rcases}
\quad \quad \quad \Omega(t)\\
P\equiv 0\quad \quad  \quad\quad \quad \quad\quad \quad \quad\quad \quad \quad\quad~~\quad\Sigma(t)\\
(1,v) \text{ is tangent to the free surface } (t, \Sigma(t)).
\end{cases}
\end{equation}
where $v$ is the fluid velocity, $P$ is the fluid pressure.

This system, and many variants and generalizations, has been extensively studied in the literature. The so-called \emph{Taylor-sign condition}(also referred as \emph{Rayleigh-Taylor sign condition} in many literature) $-\frac{\partial P}{\partial\Vec{n}}>0$ on the pressure is an important stability condition for the water waves problem. We shall call $-\frac{\partial P}{\partial \Vec{n}}$ the \emph{Taylor-sign coefficient}. If the \emph{Taylor-sign condition} fails, the system is, in general, unstable, see, for example, \cite{beale1993growth}, \cite{birkhoff1962helmholtz}, \cite{taylor1950instability}, \cite{ebin1987equations}. We refer such instability as the \emph{Rayleigh-Taylor instability}. In the irrotational case and without a bottom the validity of the \emph{Taylor-sign condition} was shown by Wu \cite{Wu1997, Wu1999}, and was the key to obtaining the first local-in-time existence results for large data in Sobolev spaces. In the case of non-trivial vorticity or with a bottom it is widely believed that the \emph{Taylor-sign condition} can fail and the sign condition has to be part of the assumptions for the initial data.  In the irrotational case, Nalimov \cite{Nalimov}, Yosihara \cite{Yosihara} and Craig \cite{Craig} proved local well-posedness for 2d water waves equation for small initial data. In S. Wu's breakthrough works \cite{Wu1997, Wu1999} she proved local-in-time well-posedness without smallness assumption. Since then, a lot of interesting local well-posedness results were obtaind, see for example \cite{alazard2014cauchy}, \cite{ambrose2005zero}, \cite{christodoulou2000motion}, \cite{coutand2007well}, \cite{iguchi2001well}, \cite{lannes2005well}, \cite{lindblad2005well}, \cite{ogawa2002free}, \cite{shatah2006geometry}, \cite{zhang2008free}, \cite{ai2019low}, \cite{ai2020low}, \cite{miao2020well}, and the references therein. See also \cite{Wu1, Wu2, Wu3, Wu4} for water waves with non-smooth interface. For the formation of splash singularities, see for example \cite{castro2012finite}, \cite{castro2013finite}, \cite{coutand2014finite}, \cite{coutand2016impossibility}.
Regarding the local-in-time wellposedness with regular vorticity, see \cite{iguchi1999free},\cite{ogawa2002free}, \cite{ogawa2003incompressible}, \cite{christodoulou2000motion}, ,\cite{christodoulou2000motion}, \cite{lindblad2005well},\cite{zhang2008free}, and \cite{su2018long} for water waves with point vortices.  In the irrotational case, almost global and global well-posedness for water waves were proved in \cite{Wu2009,Wu2011}, \cite{germain2012global}, \cite{Ionescu2015}, \cite{AlazardDelort},  and see also \cite{HunterTataruIfrim1,HunterTataruIfrim2}, \cite{wang2018global}, \cite{zheng2019long}. In the rotational case, see \cite{ifrim2015two}, \cite{bieri2017motion}, \cite{ginsberg2018lifespan}, and \cite{su2018long}.

All the aforementioned results require the \emph{Taylor-sign condition} to hold. Shinbrot \cite{shinbrot1976initial}, Kano, and Nishida \cite{kano1979ondes} justify the Friedrichs expansion for the
water-waves equations in terms of the shallowness parameter using the Cauchy-Kowalevski theorem. They take analytic initial data, and the \emph{Taylor-sign condition} is not assumed. Very recently, Alazard, Burq, and Zuily \cite{alazardAnalytic2020} revisit the analysis of the
water-problem with analytic data, using tools and methods that they developed previously to
study the Cauchy problem with rough initial data. Their result allows a bottom with Sobolev regularity. Yet none of these results proved the existence of solutions with the \emph{Taylor sign} hold initially while failing at a later time. As we know, the \emph{Taylor sign condition} is in some sense a criterion for the stability of the water waves.
It is natural to ask the following question:

\vspace*{1ex}

\noindent \textbf{\emph{Question 1:}} Is there a solution to the system (\ref{vortex_model}) such that the strong \emph{Taylor-sign condition} holds at $t=0$, while fails at some $t_0>0$?

\vspace*{1ex}

The rigorous mathematical analysis of the transition of the \emph{Rayleigh-Taylor instability} is an interesting yet less-understood subject. From the technical point of view, the break down of the \emph{Taylor sign condition} corresponds to the loss of derivatives of the solution. In \cite{castro2012rayleigh}, Castro, Cordoba, Fefferman, and Gancedo showed that the Rayleigh-Taylor condition for the Muskat problem may hold initially but break down in finite time. In their case, the solution has instant analyticity, that is, even the initial data has finite smoothness, the solution will simultaneously become real analytic after the initial time. Such instant analyticity can compensate for the loss of derivative caused by the break down of the Rayleigh-Taylor condition. For the water waves, there is no such instant analyticity, which makes the problem more difficult. To the best of our knowledge, such a transition of the Rayleigh-Taylor instability/stability was not known for the water waves. Of course, for infinite depth and irrotational water waves, the \emph{Taylor sign condition} always holds, as long as the free interface is smooth and non-self-intersecting. In order that (\ref{vortex_model}) admits a solution with the \emph{Taylor sign} breaks down, we must assume that the initial data has nontrivial vorticity. 

The interaction of the free surface and the vorticity is in general very complicated. In order to control the motion of the vorticity and therefore make the interaction between the free surface and the vorticity predictable, we assume that the initial vorticity is highly concentrated in a small region. In idealized cases, we can assume that the vorticity is given by the \emph{point vortices}, that is, the vorticity distribution $\omega:=curl~v$ is a linear combination of dirac masses, i.e.,
\begin{equation}\label{vorticitydistribution}
    \omega(\cdot,t)=\sum_{j=1}^N\lambda_j \delta_{z_j(t)}(\cdot),
\end{equation}
where $\{z_j(t)\}\subset \Omega(t)$, and $\lambda_j\in \mathbb{R}$. Since the vorticity of the 2d Euler is transported following the fluid, if the initial vorticity of (\ref{vortex_model}) is a Dirac delta mass, then $\omega$ might remain as such a point vortex for all t. The
question is, as the singular vorticity $\lambda_j \delta_{z_j(t)}$ generates a singularly rotational part $\frac{\lambda_j}{\pi}\nabla^{\perp}\log|z-z_j(t)|$ of $v$ (Here, $z=(x,y)$ and $\nabla^{\perp}=(-\partial_y, \partial_x)$), following what vector field should $z_j(t)$ move?  Since the above singular vector
field is purely rotational about $z_j(t)$ and does not move that particle $z_j(t)$, it is reasonable to expect
that the dynamics of $z_j(t)$ is governed only by the remaining smooth part of $v$, that is,
\begin{equation}\label{pointvorticesmotion}
    \dot{z}_j(t)=\Big(v(z,t)-\sum_{j=1}^N\frac{\lambda_j}{\pi}\nabla^{\perp}\log|z-z_j(t)|\Big)\Big|_{z=z_j(t)}.
\end{equation}
This well-known result is rigorously established by considering a family of vortex patch solutions whose
initial vorticity limiting (weakly) to a Dirac delta mass (see \cite{marchioro2012mathematical} Theorem 4.1, 4.2 for more details). 

The system (\ref{vortex_model})-(\ref{vorticitydistribution})-(\ref{pointvorticesmotion}) is a model for the motion of submerged bodies (see e.g. \cite{chang2001vortex},\cite{dalrymple2006numerical}) and it is believed to give some  insight into the problem of turbulence (\cite{marchioro2012mathematical}, chap 4, \S 4.6). There have been many numerical studies on this system (see for example, \cite{curtis2017vortex}, \cite{hill1975numerical}, \cite{fish1991vortex}, \cite{marcus1989interaction}, \cite{telste1989potential}, \cite{willmarth1989vortex}) and studies on the local in time well-posedness of the Cauchy problem in the irrotational case (that is, no vorticity), and for regular vortex distributions.  In \cite{su2018long}, the author proved local wellposedness of (\ref{vortex_model}) with point vortices under the strong Taylor sign assumption, and obtain extended lifespan with small initial data when there is a pair of counter-rotating point vortices travelling downward.

In the case of point vortices analyzed in this paper, we give an affirmative answer to \emph{Question 1}. The main result of this paper is 
\begin{theorem}\label{main}
For any given constants $0<\eta_0<1$ and $\eta_1>0$, there exist $T_0>0$ and a nonempty open set $\mathcal{O}$ in a category of infinite smoothness such that for each $\mathfrak{u}_0\in \mathcal{O}$, the system (\ref{vortex_model})-(\ref{vorticitydistribution})-(\ref{pointvorticesmotion}) with initial value $\mathfrak{u}_0$ admits a unique smooth solution on $[0, T_0]$. Moreover, 
the \emph{Taylor sign coefficient} $-\frac{\partial P}{\partial\Vec{n}}$ of the corresponding solution satisfies:
\begin{itemize}
\item [(1)] At $t=0$, $\inf_{\alpha\in\mathbb{R}}(-\frac{\partial P}{\partial\Vec{n}}(\alpha,0))\geq 1-\eta_0$; 
\item [(2)] At $t=T_0$, $\inf_{\alpha\in \mathbb{R}}(-\frac{\partial P}{\partial \Vec{n}}(\alpha,T_0)\leq -\eta_1$.
\end{itemize}
\end{theorem}
In another words, the infimum of the Taylor sign coefficient can transit from almost the constant one to an arbitrary large yet negative number.
\begin{remark}
The category of infinite smoothness in \emph{Theorem} \ref{main} is indeed the Gevrey-2 space, see \emph{Definition} \ref{Gevrey} for the definition of such spaces, and \emph{Theorem} \ref{taylorsignfailtheorem2} for the precise quantitative version of \emph{Theorem} \ref{main}.
\end{remark}

Such a result is our first step to investigate systematically the transition of the Rayleigh-Taylor instability (from stable regime to unstable regime, and vice versa) in water waves. Since (\ref{vortex_model})-(\ref{vorticitydistribution})-(\ref{pointvorticesmotion}) is reversible in time, \emph{Theorem} \ref{main} implies that there exists a nonempty open set of initial data in a category of infinite smoothness with a pair of point vortices such that $\inf_{\alpha\in \mathbb{R}}(-\frac{\partial P}{\partial \Vec{n}}(\alpha,0))<-\eta_1$, while $\inf_{\alpha\in \mathbb{R}}(-\frac{\partial P}{\partial \Vec{n}}(\alpha, T_0))>1-\eta_0$ at the finite time $0<T_0<\infty$.

\begin{remark}
Our method actually works if the pair of point vortices are replaced by a pair of concentrated vortex patches.
\end{remark}

\subsection{The difficulty and the strategy}
In Riemann mapping variables (See \S \ref{Riemannvariable} for the precise derivations), the water waves system can be written in the form
\begin{equation}\label{temporaly}
    \begin{cases}
     D_tU=-\Re\{D_tQ\}+A\Lambda W\\
     D_tW=-U+\Re\{Q\}-\Re\{[D_tZ,\mathbb{H}](\frac{1}{Z_{\alpha}}-1)\}-2\Re \{(I-\mathbb{H})Q\},\\
     Z-\alpha=(I+\mathbb{H})W,\\
     F=(I+\mathbb{H})U.
    \end{cases}
\end{equation}
where $U, W$ are real-valued functions, $D_t=\partial_t+b\partial_{\alpha}$, $A$ and $b$ are real-valued functions depending on $U, W$ and $Q$, and $Q=-\sum_{j=1}^N\frac{\lambda_j i}{2\pi}\frac{1}{Z(\alpha,t)-z_j(t)}$. $\mathbb{H}$ is the standard Hilbert transform. $A$ plays the role of the \emph{Taylor-sign condition}. If $\inf_{\alpha\in \mathbb{R}}A\geq 0$, then the \emph{Taylor-sign condition} holds; otherwise, the \emph{Taylor-sign condition} fails.

\subsubsection{A mechanism for solution with $A$ changes sign: the interaction of vorticity and the free interface}
Let $A_1:=A|Z_{\alpha}|^2$.  It suffices to construct solutions to the water waves with a sign-changing $A_1$. In \cite{su2018long}, the author derived a formula for $A_1$, which we record as follows. 
\begin{equation}\label{temp_riemann_formula_A1}
         A_1=1+\frac{1}{2\pi}\int \frac{|D_tZ(\alpha,t)-D_tZ(\beta,t)|^2}{(\alpha-\beta)^2}d\beta-\sum_{j=1}^N \frac{\lambda_j}{2\pi} Re\Big\{\Big((I-\mathbb{H})\frac{Z_{\alpha}}{(Z(\alpha,t)-z_j(t))^2}\Big)(D_tZ-\dot{z}_j(t))\Big\}.
 \end{equation}
Here, $D_tZ$ is the trace of the velocity field along the free boundary. We split $A_1$ as
$$A_1=A_1^{pv}+A_1^{int},$$
where $A_1^{pv}$ represents the unperturbed case, i.e., corresponding to $Z(\alpha,t)=\alpha$, $F=0$; $A_1^{int}$ represents the contribution due to the interaction of the wave $F$ and the point vortices.  Our idea is to find a solution to the water wave system such that, for given $\eta_0\in (0,1)$ and $\eta_1>0$,
\begin{itemize}
    \item [(R1)] $\inf_{\alpha\in\mathbb{R}} A_1^{pv}(\alpha,0)>1-\eta_0$, 
    
    \item [(R2)] $\inf_{\alpha\in\mathbb{R}} A_1^{pv}(\alpha, T_0)<-\eta_1$, for some $T_0>0$,
    
    \item [(R3)] $\sup_{\alpha\in \mathbb{R}}|A_1^{int}(\alpha,t)|\ll  \Big|\inf_{\alpha\in\mathbb{R}}A_1^{pv}(\alpha,t)\Big|$, \quad for $t=0, ~T_0$.
\end{itemize}
(R1) and (R2) can be guaranteed by analyzing the motion of a pair of counter-rotating point vortices. Indeed, we have  the following observation.
\vspace*{1ex}

\noindent \emph{\textbf{The Taylor sign of the 2d water waves with a vortex pair} }

\vspace*{1ex}

\noindent \textbf{Idealized case: the motion of the water waves is completely given by the motion of the point vortices.} 
Suppose there exists a smooth solution to the water waves such that at time $t$, the free interface $\Sigma_t=\mathbb{R}$, and the velocity field is generated by a pair of counter-rotating point vortices only, that is, the initial vorticity $\omega:=\lambda\delta_{z_1(t)}-\lambda\delta_{z_2(t)}$, with $z_1(t)=-x(t)+iy(t)$ and $z_2(t)=x(t)+iy(t)$ symmetric about the vertical axis. In this case, $A^{pv}=A_1$. Then we have 
\begin{itemize}
    \item [(C1)] Fix $\lambda$ and $x(t)$. If $|y(t)|$ is sufficiently large, then $A_1=1+O(\frac{1}{|y(t)})$.
    
    \item [(C2)] Fix $x(t)$. Let $\frac{\lambda^2}{|y(t)|^3}\rightarrow \infty$ and require $|y(t)|\gg 1$, then $A_1\rightarrow -\infty$.
\end{itemize}
 (C1) and (C2) will be justified from the calculation in \S \ref{sectionfail} and the appendix.

\vspace*{1ex}

\noindent \emph{\textbf{Another observation:} If $\lambda>0$, then the vortex pair $z_1(t), z_2(t)$ travel toward the free interface (so $y(t)$ becomes smaller). }

\vspace*{1ex}

\noindent \emph{\textbf{A natural idea:} } Construct water waves with a symmetric pair of point vortices far away from the free interface initially and traveling toward the free interface rapidly. Then we expect $A_1$ to satisfy (R1)-(R2)-(R3).

\vspace*{2ex}

To guarantee (R3), the main difficulty is to justify that the water waves really remain a small perturbation of the motion of the vortex pair (for a period of time that allows the vortex pair to travel sufficiently close to the free interface such that the Taylor-sign fails), which can be guaranteed if we can control the growth of $U$ and $W$. This forces us to prove the local existence in some reasonable spaces of such water waves. Nevertheless, even in the irrotational case, all the existing local wellposedness results for the water wave system with finite smoothness require the strong \emph{Taylor-sign condition} to hold. As we mentioned earlier, if the \emph{Taylor-sign condition} fails, then the water waves can be subject to the \emph{Rayleigh-Taylor instability} and a local wellposedness is not expected, at least in spaces with finite smoothness. 

Roughly speaking, if the \emph{Taylor-sign condition} fails, the system (\ref{temporaly}) losses  $\frac{1}{2}$ derivatives.  In Sobolev spaces $H^s$, such a loss of derivative can be overcome if the strong \emph{Taylor-sign condition} holds. If the \emph{Taylor-sign condition} fails, then the Fourier mode can grow exponentially, which prevents a local wellposedness in Sobolev spaces.

Instead of seeking for water waves with finite regularity (e.g., Sobolev spaces), we perturb the vortex pair in a category with infinite smoothness. Such an idea is certainly not new. Even in the context of water waves, 
by expanding the solution as a power series in time,  Shinbrot \cite{shinbrot1976initial}, Kano and Nishida's \cite{kano1979ondes} obtained local existence in time for the general
initial value problem with real analytic data. Their tool is the  Cauchy-Kowalevski theorem. In \cite{alazardAnalytic2020}, Alazard, Burq, and Zuily proved the local wellposedness of water waves in real analytic spaces via the energy method. Their main analytical tool is the paradifferential calculus. One should be able to adapt their method to prove the local in time wellposedness of water waves with point vortices. In this paper, working in Riemann mapping variables, we prove the local wellposedness in \emph{Gevrey-2 space} (The \emph{Gevrey-2 space} is a space of infinite smoothness that extends the real-analytic space. See \emph{Definition} \ref{Gevrey} for the precise definition) by using the energy method directly. The key to prove \emph{Theorem} \ref{main} is the following.
\begin{theorem}\label{main2}
The water wave system (\ref{temporaly}) is locally wellposed in Gevrey-2 spaces.
\end{theorem}
Such a local wellposedness is expected to be sharp, in the sense that if the initial data is in a space rougher than the Gevrey-2 space, then the system (\ref{temporaly}) is illposed. See \emph{Remark} \ref{remarklwp} for more quantitative discussions.

\vspace*{2ex}

To illustrate our strategy, we 
consider the following toy model:
\begin{equation}\label{toy}
\begin{cases}
\partial_t U=A\Lambda W,\\
\partial_t W=-U,\\
(U,W)=(U_0,W_0),
\end{cases}
\end{equation}
for some real-valued function $A$ that might depend on $u$ and $W$.

\subsubsection{The Gevrey framework}\label{thegevreyframework}

Assume that $A(\alpha_0,0)<0$ for some given $\alpha_0\in \mathbb{R}$. Heuristically, by considering solutions $U$ and $W$ concentrating near $\alpha_0$, we can then without loss of generality (although it is not trivial to rigorously justify this) assume that $A(\alpha,t)\equiv A(\alpha_0,0)<0$. In Fourier variables, we have
\begin{equation}
  \partial_t\begin{bmatrix}\hat{U}\\ \hat{W}\end{bmatrix} =\begin{bmatrix} 0 & A|\xi|\\
-1 & 0\end{bmatrix}\begin{bmatrix}\hat{U}\\\hat{W}\end{bmatrix}.
\end{equation}
The matrix $\begin{bmatrix} 0 & A|\xi|\\
-1 & 0\end{bmatrix}$ has real eigenvalues $\lambda_1=\sqrt{-A|\xi|}$ and $\lambda_2=-\sqrt{-A|\xi|}$, with eigenvectors $V_1=\begin{bmatrix} -\sqrt{-A|\xi|}\\ 1 \end{bmatrix}$ and $V_2=\begin{bmatrix} \sqrt{-A|\xi|}\\ 1 \end{bmatrix}$, respectively. Let $S=[V_1, V_2]$, and let $\mathcal{Y}=S^{-1}\begin{bmatrix}\hat{U}\\ \hat{W}\end{bmatrix}$, then 
\begin{equation}
    \partial_t\mathcal{Y}=\begin{bmatrix} \sqrt{-A|\xi|} & 0\\
    0 & -\sqrt{-A|\xi|}\end{bmatrix}\mathcal{Y}.
\end{equation}
So 
$$\mathcal{Y}(\xi,t)=\begin{bmatrix} e^{t\sqrt{-A|\xi|}} & 0\\
    0 & e^{-t\sqrt{-A|\xi|}}\end{bmatrix}\mathcal{Y}(\xi,0),$$
which implies that for any $t>0$, we have \footnote{The reason that we use the infinite series to define the Gevrey norm is as follows: we need to estimate the Gevrey norm of functions of the form $Q_j:=\frac{1}{Z(\alpha,t)-z_j(t)}$. It is easier to estimate $\sum \frac{\sigma^{2n}}{(n!)^4}\|\partial_{\alpha}^n Q_j\|_{L^2}^2$ than $\int_{\mathbb{R}}e^{2\sigma |\partial_{\alpha}|^{1/2}}|Q(\alpha,t)|^2 d\alpha$. }
\begin{equation}\label{sharpgevrey} \int_{\mathbb{R}} e^{\delta |\xi|} |\mathcal{Y}(\xi,t)|^2 d\xi=\infty \quad \quad \Leftrightarrow \quad \sum_{n=0}^{\infty}\frac{\delta^{2n}}{(n!)^4}\int_{\mathbb{R}}|\xi|^{2n}|\mathcal{Y}(\xi,t)|^2d\xi=\infty,\end{equation}
provided that $\mathcal{Y}(\xi,0)\geq Ce^{\delta |\xi|^{1/2}}$ for some $\delta>-\sqrt{-A}$. Therefore, it is natural to impose the assumption $|\mathcal{Y}(\xi,0)|\leq Ce^{-\sqrt{-A|\xi|}}$. Such $\mathcal{Y}(\cdot,0)$ lies in the Gevrey-2 spaces. See \emph{Definition} \ref{Gevrey} for the more precise definition. We denote 
$$\|f\|_{X_{L_0}}^2:= \sum_{n=0}^{\infty}\frac{L_0^{2n}}{(n!)^4}\int_{\mathbb{R}}|\xi|^{2n}|\hat{f}(\xi)|^2d\xi\quad(= \sum_{n=0}^{\infty}\frac{L_0^{2n}}{(n!)^4}\|\partial_{\alpha}^n f\|_{L^2}^2).$$
$L_0$ is called the \emph{radius of analyticity}. We seek for solution $(U, W)$ such that the initial data $(U_0, W_0)$ satisfies 
$$\| \Lambda^{1/2} U_0\|_{X_{L_0}}<\infty, \quad \quad \| \partial_{\alpha}W_0\|_{X_{L_0}}<\infty.$$

\vspace*{2ex}

Of course, in general, $A$ is not a constant. In particular, in the context of water waves, $A$ depends on $U$ and $W$. Moreover, in the first equation of (\ref{temporaly}), $W$ losses one half derivative, due to the term $A\Lambda W$. To compensate for this loss, we allow the \emph{radius of analyticity} of the solution to decay linearly in time, that is, we seek for solution $(U, W)$ to (\ref{toy}) with initial data $(U_0, W_0)$ such that 
\begin{equation}
    \| \Lambda^{1/2}U(\cdot,t)\|_{X_{L_0-\delta t}}<\infty, \quad \quad \|\partial_{\alpha}W\|_{X_{L_0-\delta t}}<\infty.
\end{equation}
Such a linear decay in the radius of analyticity has the same effect of introducing a viscosity term $\approx \Lambda^{1/2}\begin{bmatrix} U\\ W\end{bmatrix}$, that is, in the $X_{L_0-\delta t}$ norm, the toy model behaves like
\begin{equation}\label{toy2}
    \begin{cases}
    \partial_t U+c \Lambda^{1/2}U=A \Lambda W,\\
    \partial_t W+c \Lambda^{1/2}W=-U,
    \end{cases}
\end{equation}
for some positive constant $c>0$. Clearly, (\ref{toy2}) does not loss any derivative.
A more apparent way to explain the compensation of the loss of derivative is by energy estimates.  We obtain
\begin{align*}
   & \frac{d}{dt}(\|\Lambda^{1/2}U\|_{X_{L_0-\delta t}}^2+\|\partial_{\alpha}W\|_{X_{L_0-\delta t}}^2)\\
   \leq &  C(\|\Lambda^{1/2}U\|_{X_{L_0-\delta t}}^2+\|\partial_{\alpha}W\|_{X_{L_0-\delta t}}^2)\\
    &-(\frac{2\delta}{L_0-\delta t}\sum_{n=1}^{\infty}\frac{n(L_0-\delta t)^{2n}}{(n!)^4}\|\partial_{\alpha}^n\Lambda^{1/2}U(\cdot,t)\|_{L^2}^2+\frac{2\delta}{L_0-\delta t}\sum_{n=1}^{\infty}\frac{n(L_0-\delta t)^{2n}}{(n!)^4}\|\partial_{\alpha}^n\Lambda^{1/2}U(\cdot,t)\|_{L^2}^2)\\
    &+\sum_{n=0}^{\infty}\frac{(L_0-\delta t)^{2n}}{(n!)^4}\Big|\int_{\mathbb{R}} \partial_{\alpha}^n \partial_{\alpha}U(\alpha,t) \partial_{\alpha}^n W_{\alpha}(\alpha,t)d\alpha\Big|
\end{align*}
By choosing $\delta$ sufficiently large, one can bound the last summation by $$(\frac{2\delta}{L_0-\delta t}\sum_{n=1}^{\infty}\frac{n(L_0-\delta t)^{2n}}{(n!)^4}\|\partial_{\alpha}^n\Lambda^{1/2}U(\cdot,t)\|_{L^2}^2+\frac{2\delta}{L_0-\delta t}\sum_{n=1}^{\infty}\frac{n(L_0-\delta t)^{2n}}{(n!)^4}\|\partial_{\alpha}^n\Lambda^{1/2}U(\cdot,t)\|_{L^2}^2).$$ So we can obtain closed energy estimates.

\vspace*{2ex}

Using this idea, we are able to prove the local wellposedness of (\ref{temporaly}) in Gevrey-2 spaces, regardless of the sign of $A$.

\vspace*{2ex}

\noindent \textbf{The lifespan of the perturbation.} As we mentioned earlier, it is critical that the water waves remain a small perturbation of the vortex pair for sufficiently long time\footnote{By \emph{sufficiently long time}, we do not mean a long time existence for water waves. It merely means that the lifespan is sufficient for the \emph{Taylor sign coefficient} to transit from a positive number to an arbitrary large negative number.}. 

\begin{itemize}
    \item [(I1)] By (C1), in order that $A_1$ is sufficiently close to 1 at $t=0$, we need $|\lambda|^2/|y(0)|^3\ll 1$. So $|y(0)|\gg |\lambda|^{2/3}$. 
    
    \item [(I2)] By (C2), in order that $A_1(T_0)$ takes a large negative number, we need $\frac{\lambda^2}{|y(t)|^3}$ to be sufficiently large, which suggests that $|y(T_0)|\ll |y(0)|$. For simplicity, we take $y(T_0)\approx -|y(0)|^{9/10}$.
    
    \item [(I3)] In a perturbative regime,  the velocity of the point vortex pair is $\frac{\lambda}{4\pi x(0)}$, plus a negligible error term (the details will be provided in \S \ref{sectionfail}). So we have 
\begin{equation}
    T_0=\frac{|y(0)-y(T_0)|}{\frac{\lambda}{4\pi x(0)}}.
\end{equation}
In particular, $T_0\gg \frac{4\pi x(0)}{|y_0|^{1/2}}$.

\end{itemize}

It is not difficult to prove the local well-posedness of (\ref{temporaly}), yet to obtain a lifespan with lower bound at least $\frac{4\pi x(0)}{|y_0|^{1/2}}$ is another story, because if $\frac{\lambda^2}{|y(t)|^3}$  and $y(0)$ are large, then the quantity $-\Re\{D_tQ\}$ on the right hand side of the first equation of (\ref{temporaly}) is also large, of size roughly speaking $\frac{\lambda^2}{|y(t)|^2}$. Thus we need to analyze the interaction of the free interface and the point vortices to exclude the possibility of rapid growth of the water waves, from which we can conclude the proof of our main theorem.

\subsection{Outline of the paper}
In \S \ref{prelim}, we formulate the water wave system in Riemann variables, define the function spaces, and derive some estimates in Gevrey spaces. In \S \ref{maintheorems}, we state a quantitative version of the main results. In \S \ref{sectionlinear}, using the energy method, we prove the local wellposedness in Gevrey spaces of a specific quasilinear system.  In \S \ref{sectionestimates}, we derive a priori estimates. In \S \ref{sectionapproximate}, we use the picard iteration method to prove the existence and uniqueness of solutions of the system (\ref{quasi2}) and therefore conclude the proof of Theorem \ref{theorem1}. In \S \ref{sectionfail}, using Theorem \ref{theorem1}, we prove Theorem \ref{taylorsignfailtheorem2}. In the appendix, we provide the details of some calculations for the Taylor sign coefficient.

\section{Preliminaries}\label{prelim}
Throughout this paper, we denote the fluid region at time $t$ by $\Omega(t)$, with a nonself-intersect free surface $\Sigma(t)$. We identify a point $(x,y)\in \mathbb{R}$ as a point $x+iy\in \mathbb{C}$.  For a point $x+iy\in \Sigma(t)$, we assume $|y|\rightarrow 0$ as $|x|\rightarrow \infty$. That is, $\Sigma(t)$ approaches the real axis at $\pm\infty$. Define $\langle u, v\rangle:=\int_{\mathbb{R}}u(\alpha)\bar{v}(\alpha)$. Let $z=x+iy\in \mathbb{C}$, $\Re\{z\}:=x$, $\Im\{z\}:=y$.

\subsection{Governing equation for the free boundary}

The system (\ref{vortex_model})-(\ref{vorticitydistribution})-(\ref{pointvorticesmotion}) is completely determined by the free surface $\Sigma(t)$, the trace of the velocity $v$ along the free surface, and the position of the point vortices. We shall first formulate the system (\ref{vortex_model})-(\ref{vorticitydistribution})-(\ref{pointvorticesmotion}) in Lagrangian coordinates. Because of the moving boundary, it is convenient to use the Riemann mapping variables to study the Taylor sign and construct special solutions.

\subsubsection{Lagrangian formulation} We parametrize the free surface by Lagrangian coordinates, i.e., let $\alpha$ be such that 
\begin{equation}
z_t(\alpha,t)=v(z(\alpha,t),t).
\end{equation}
$P\Big|_{\Sigma(t)}\equiv 0$ implies that $\nabla P\Big|_{\Sigma(t)}$ is along the normal direction, so we can write $\nabla P$ as $-iaz_{\alpha}$, where $a=-\frac{\partial P}{\partial\vec{n}}\frac{1}{|z_{\alpha}|}$ is real valued. Here, $\vec{n}=i\frac{z_{\alpha}}{|z_{\alpha}|}$ is the unit outward normal to $\Omega(t)$. So the trace of the momentum equation $v_t+v\cdot\nabla v=-\nabla P-(0,1)$ can be written as
\begin{equation}
z_{tt}-iaz_{\alpha}=-i.
\end{equation}

\begin{definition}[Hilbert transform]
Assume that $z(\alpha)$ satisfies
\begin{equation}\label{chordchordarcarc}
    \beta_0|\alpha-\beta|\leq |z(\alpha)-z(\beta)|\leq \beta_1|\alpha-\beta|, \quad \quad \forall~~ \alpha,\beta\in \mathbb{R},
\end{equation}
where $0<\beta_0<\beta_1<\infty$ are constants.
The Hilbert transform associates to the curve $z(\alpha)$ is defined as 
\begin{equation}
\mathfrak{H}f(\alpha):=\frac{1}{\pi i}p.v.\int_{-\infty}^{\infty}\frac{z_{\beta}(\beta)}{z(\alpha)-z(\beta)}f(\beta)d\beta.
\end{equation}
The standard Hilbert transform is the Hilbert transform associated to $z(\alpha)=\alpha$, that is,
\begin{equation}
\mathbb{H}f(\alpha):=\frac{1}{\pi i}p.v.\int_{-\infty}^{\infty}\frac{1}{\alpha-\beta}f(\beta)d\beta.
\end{equation}
\end{definition}
\vspace*{1ex}

We can use the Hilbert transform to characterize the boundary value of holomorphic functions. Such characterization is classical, the reader can see for example Proposition 2.1 in \cite{wu2016wellposedness}.
\begin{lemma}\label{holomorphic}
Let $\Omega\subset \mathbb{C}$ be a domain with $C^1$ boundary $\Sigma: z=z(\alpha), \alpha\in \mathbb{R}$, oriented clockwise. Let $\mathfrak{H}$ be the Hilbert transform associated to $\Omega$.
\begin{itemize}
    \item [(a.)] Let $g\in L^p$ for some $1<p<\infty$. Then $g$ is the boundary value of a holomorphic function $G$ on $\Omega$ with $G(z)\rightarrow 0$ at infinity if and only if 
    \begin{equation}
        (I-\mathfrak{H})g=0.
    \end{equation}
    
    \item [(b.)] Let $f\in L^p$ for some $1<p<\infty$. Then $\frac{1}{2}(I+\mathfrak{H})f$ is the boundary value of a holomorphic function $\mathfrak{G}$ on $\Omega$, with $\mathfrak{G}(z)\rightarrow 0$ as $|z|\rightarrow \infty$.
    
    \item [(c.)] $\mathfrak{H}1=0$.
\end{itemize}
\end{lemma}

We decompose $\bar{z}_t$ as $\bar{z}_t=f+p$, where $p=-\sum_{j=1}^n\frac{\lambda_j i}{2\pi}\dfrac{1}{z(\alpha,t)-z_j(t)}$.
Note that $f$ is holomorphic in $\Omega(t)$ with the value at the boundary $\Sigma(t)~$ given by $\bar{z}_t+\sum_{j=1}^N \frac{\lambda_j i}{2\pi (z(\alpha,t)-z_j(t))}$ . Let $f\in L^2(\mathbb{R})$, then $f$ is the boundary value of a holomorphic function in $\Omega(t)$ if and only if 
\begin{equation}
(I-\mathfrak{H})f=0,
\end{equation}
where $\mathfrak{H}$ is the Hilbert transform associated with the curve $z(\alpha,t)$, i.e.,
\begin{equation}
\mathfrak{H}f(\alpha):=\frac{1}{\pi i}p.v.\int_{-\infty}^{\infty}\frac{z_{\beta}}{z(\alpha,t)-z(\beta,t)}f(\beta)d\beta.
\end{equation}
Because of the singularity of the velocity at the point vortices, we don't have $(I-\mathfrak{H})\bar{z}_t=0$. However, the following lemma asserts that $\bar{z}_t$ is almost holomorphic, in the sense that $(I-\mathfrak{H})\bar{z}_t$ consists of lower order terms.
\begin{lemma}[Almost holomorphicity]\label{almost}
Assume that $z(\cdot,t)\in L^2(\mathbb{R})$ satisfies (\ref{chordchordarcarc}) and $\bar{z}_t$ is the boundary value of a velocity field $\bar{v}$ in $\Omega(t)$ such that $\bar{v}+\sum_{j=1}^N \frac{\lambda_j i}{2\pi(z-z_j(t)})$ is holomorphic in $\Omega(t)$. Then we have
\begin{equation}
\begin{split}
(I-\mathfrak{H})\bar{z}_t=&-\frac{i}{\pi}\sum_{j=1}^N \frac{\lambda_j}{z(\alpha,t)-z_j(t)}.
\end{split}
\end{equation}
\end{lemma}
\begin{proof}
Since $\bar{z}_t+\sum_{j=1}^N \frac{\lambda_j i}{2\pi(z(\alpha,t)-z_j(t))}$ is the boundary value of a holomorphic function in $\Omega(t)$, by lemma \ref{holomorphic},
$$(I-\mathfrak{H})(\bar{z}_t+\sum_{j=1}^N \frac{\lambda_j i}{2\pi(z(\alpha,t)-z_j(t))})=0,$$
we have 
\begin{equation}\label{holo1}
(I-\mathfrak{H})\bar{z}_t=-\sum_{j=1}^N (I-\mathfrak{H}) \frac{\lambda_j i}{2\pi(z(\alpha,t)-z_j(t))}.
\end{equation}
Since $\frac{1}{z(\alpha,t)-z_j(t)}$ is boundary value of the holomorphic function $\frac{1}{z-z_j(t)}$ in $\Omega(t)^c$, by lemma \ref{holomorphic} again,  we have 
\begin{equation}\label{holo2}
(I-\mathfrak{H})\frac{1}{z(\alpha,t)-z_j(t)}=\frac{2}{z(\alpha,t)-z_j(t)}.
\end{equation}
(\ref{holo1}) together with (\ref{holo2}) complete the proof of the lemma.
\end{proof}

So the system (\ref{vortex_model}) is reduced to a system of equations for the free boundary coupled with the dynamic equation for the motion of the point vortices:
\begin{equation}\label{vortex_boundary}
\begin{cases}
z_{tt}-iaz_{\alpha}=-i\\
\frac{d}{dt}z_j(t)=(v-\frac{\lambda_j i}{2\pi(\overline{z-z_j})})\Big |_{z=z_j}\\
(I-\mathfrak{H})f=0.
\end{cases}
\end{equation}
Note that $v$ can be recovered from (\ref{vortex_boundary}). Indeed, we have 
\begin{equation}\label{velocity}
\bar{v}(z,t)+\sum_{j=1}^N\frac{\lambda_j i}{2\pi(z-z_j(t))}=\frac{1}{2\pi i}\int \frac{z_{\beta}}{z-z(\beta)}\Big(\bar{z}_t(\beta,t)+\sum_{j=1}^N \frac{\lambda_j i}{2\pi(z(\beta,t)-z_j(t))}\Big)d\beta.
\end{equation}
So the system (\ref{vortex_model}) and the system (\ref{vortex_boundary}) are equivalent. 

\vspace*{2ex}

\noindent The quantity $a|z_{\alpha}|$ plays an important role in the study of water waves\footnote{Indeed, $a|z_{\alpha}|=-\frac{\partial P}{\partial \vec{n}}\Big|_{\Sigma(t)}$.}.
\begin{definition}
(The \emph{Taylor-sign condition} and the \emph{strong Taylor sign condition})
\begin{itemize}
\item [(1)] If $a|z_{\alpha}|\geq 0$ pointwisely, then we say the \emph{Taylor-sign condition} holds.

\item [(2)] If there is some positive constant $c_0$ such that $a|z_{\alpha}|\geq c_0>0$ pointwisely, then we say the \emph{strong Taylor sign condition} holds.
\end{itemize}
\end{definition}

In order to derive a useful formula for the \emph{Taylor sign coefficient} and use the iteration method to construct solutions to (\ref{vortex_boundary}),  we use the Riemann mapping formulation.

\subsubsection{The Riemann mapping formulation.}\label{Riemannvariable}  Let $\mathbb{P}_-=\{z\in \mathbb{C}: \Im\{z\}<0\}$.
Let $\Phi(\cdot,t):\Omega(t)\rightarrow \mathbb{P}_-$ be the Riemann mapping such that $\Phi_z\rightarrow 1$ as $z\rightarrow \infty$. 
Denote
\begin{equation}
\begin{cases}
h(\alpha,t):=\Phi(z(\alpha,t),t),\\
Z(\alpha,t):=z\circ h^{-1}(\alpha,t),\\
b=h_t\circ h^{-1},\\
D_t:=\partial_t+b\partial_{\alpha},\\
A:=(ah_{\alpha})\circ h^{-1},\\
F:=f\circ h^{-1}.
\end{cases}
\end{equation}

In Riemann mapping variables, the system (\ref{vortex_boundary}) becomes
\begin{equation}\label{vortex_model_Riemann}
\begin{cases}
(D_t^2-iA\partial_{\alpha})Z=-i\\
\frac{d}{dt}z_j(t)=(v-\dfrac{\lambda_j i}{2\pi(\overline{z-z_j})})\Big |_{z=z_j}\\
(I-\mathbb{H})F=0.
\end{cases}
\end{equation}
Denote 
\begin{equation}
A_1:= A|Z_{\alpha}|^2.
\end{equation}
Since $(a|z_{\alpha}|)\circ h^{-1}=A|Z_{\alpha}|=\frac{A_1}{|Z_{\alpha}|}$, it's clear that the \emph{Taylor-sign condition} holds if and only if
\begin{equation}\label{conditionA1}
\inf_{\alpha\in \mathbb{R}}\frac{A_1}{|Z_{\alpha}|}\geq 0,
\end{equation}
and the \emph{strong Taylor sign condition} holds if and only if
\begin{equation}\label{conditionA1}
\inf_{\alpha\in \mathbb{R}}\frac{A_1}{|Z_{\alpha}|}> 0.
\end{equation}
\begin{definition}
We call $A_1$ the \emph{Taylor sign coefficient} corresponding to the solution $(Z, F, \{z_j\})$ of the water waves. 
\end{definition}

\begin{remark}
Note that $A_1=-|Z_{\alpha}|\frac{\partial P}{\partial \Vec{n}}$. 
\end{remark}
\vspace*{2ex}

\noindent \textbf{Formula for $b$:}  Recall that $h(\alpha,t)=\Phi(z(\alpha,t),t)$, where $\Phi$ is the Riemann mapping. So we have 
\begin{equation}
h_t=\Phi_t+\Phi_z z_t,\quad \quad \Phi_z=\frac{h_{\alpha}}{z_{\alpha}}.
\end{equation}
Precomposite with $h^{-1}$ on both sides of the above,
\begin{align*}
b=& h_t\circ h^{-1}=\Phi_t\circ Z+\frac{D_tZ}{Z_{\alpha}}\\
=& \Phi_t\circ Z+D_tZ(\frac{1}{Z_{\alpha}}-1)+D_tZ\\
=& \Phi_t\circ Z+D_tZ(\frac{1}{Z_{\alpha}}-1)+\bar{Q}+\bar{F}.
\end{align*}
 Applying $I-\mathbb{H}$, using $(I-\mathbb{H})\Phi_t\circ Z=0$ and $(I-\mathbb{H})(\frac{1}{Z_{\alpha}}-1)=0$, $(I-\mathbb{H})\bar{F}=2F$, then taking real part, we obtain
 \begin{equation}\label{riemanformulab}
 b=\Re\{[D_tZ,\mathbb{H}](\frac{1}{Z_{\alpha}}-1)\}+2\Re \{D_tZ\}=\Re\{[D_tZ,\mathbb{H}](\frac{1}{Z_{\alpha}}-1)\}+2\Re \{(I-\mathbb{H})\bar{Q}\}+2\Re\{F\}.
 \end{equation}
Decomposing $b$ as $b=b_0+b_1$, where 
\begin{equation}\label{formulab0}
    b_0=2\Re\{F\}+\Re\{[\bar{F},\mathbb{H}](\frac{1}{Z_{\alpha}}-1)\},
\end{equation}
and
\begin{equation}\label{formulab1}
    b_1=\Re\{[\bar{Q},\mathbb{H}](\frac{1}{Z_{\alpha}}-1)\}+2\Re \{(I-\mathbb{H})\bar{Q}\}.
\end{equation}
Note that $b_1$ is more regular than $b_0$. 
\vspace*{2ex}

\noindent \textbf{Formula for $A_1$:} In \S 4 of \cite{su2018long}, the author derived a formula for $A_1$, which we record as follows:
\begin{equation}\label{riemann_formula_A1}
         A_1=1+\frac{1}{2\pi}\int \frac{|D_tZ(\alpha,t)-D_tZ(\beta,t)|^2}{(\alpha-\beta)^2}d\beta-\sum_{j=1}^N \frac{\lambda_j}{2\pi} Re\Big\{\Big((I-\mathbb{H})\frac{Z_{\alpha}}{(Z(\alpha,t)-z_j(t))^2}\Big)(D_tZ-\dot{z}_j(t))\Big\}.
 \end{equation}

\vspace*{1ex}

Let $u=D_t\bar{Z}=F+Q$, recall that $F(\alpha,t)=\mathcal{U}(Z(\alpha,t),t)$ for some holomorphic function $\mathcal{U}$ in $\mathbb{P}_-$, and
$Q$ is given by
\begin{equation}\label{formulaforQ}
    Q=-\sum_{j=1}^N\frac{\lambda_j i}{2\pi}\frac{1}{Z(\alpha,t)-z_j(t)},
\end{equation}

We have  
\begin{equation}\label{formulaforU}
    \mathcal{U}(z,t)=\frac{1}{2\pi i}\int_{-\infty}^{\infty}\frac{Z_{\beta}}{z-Z(\beta,t)}F(\beta,t)d\beta.
\end{equation}
So we obtain a system in Riemann variables:
\begin{equation}\label{quasione}
\begin{cases}
D_tF=-D_tQ-iA\bar{Z}_{\alpha}+i,\\
D_t(\bar{Z}-\alpha)=F+Q-b,\\
F, Z-\alpha\text{ holomorphic}.
\end{cases}
\end{equation}
(\ref{quasione}) is equivalent to (\ref{vortex_boundary}). Let $U=\Re\{F\}$, $W=\Re\{Z-\alpha\}$. Since $F$ and $Z-\alpha$ are holomorphic, $U$ and $W$ determine $F$ and $Z-\alpha$, respectively. Moreover, we have 
\begin{equation}
    -iA\bar{Z}_{\alpha}+i=A\Lambda (\bar{Z}-\alpha)-i(A-1),
\end{equation}
where $\Lambda=|\partial_{\alpha}|$.
Taking real parts on both sides of the first and the second equation of (\ref{quasione}), we obtain
\begin{equation}
    \begin{cases}
     D_tU=A\Lambda W-\Re\{D_tQ\}\\
     D_tW=U+\Re\{Q\}-b,\\
     Z-\alpha=(I-\mathbb{H})W,\\
     F=(I-\mathbb{H})U.
    \end{cases}
\end{equation}
Using (\ref{riemanformulab}), and decompose $b=b_0+b_1$ for $b_0$, $b_1$ given by (\ref{formulab0}) and (\ref{formulab1}), respectively, we obtain\footnote{We need to estimate $\|W_{\alpha}\|_{X_{\phi(t)}}$ (see \emph{Definition} \ref{Gevrey} for the definition of the Gevrey-2 norm $\|\cdot\|_{X_{\sigma}}$), so we need to estimate $\|b_{\alpha}\|_{X_{\phi(t)}}$. We can prove that $\|\partial_{\alpha}b_1\|_{X_{\phi(t)}}<\infty$. However, $\|\partial_{\alpha}b_0\|_{X_{\phi(t)}}$ is not necessarily finite. So we need to treat $\partial_{\alpha}b_0$ carefully. }
\begin{equation}\label{quasi2}
    \begin{cases}
     D_tU=-\Re\{D_tQ\}+A\Lambda W\\
     D_tW=-U+\Re\{Q\}-\Re\{[D_t\bar{F},\mathbb{H}](\frac{1}{Z_{\alpha}}-1)\}-b_1,\\
     Z-\alpha=(I+\mathbb{H})W,\\
     F=(I+\mathbb{H})U.
    \end{cases}
\end{equation}

The motion of $z_j(t)$ is given by
\begin{equation}\label{pointvortexevolution}
    \frac{d}{dt}z_j(t)=\overline{\mathcal{U}(z_j(t),t)}+\sum_{\substack{1\leq k\leq N\\ k\neq j}}\frac{\lambda_k i}{2\pi}\frac{1}{\overline{z_j(t)-z_k(t)}},
\end{equation}
where $\mathcal{U}$ is given by (\ref{formulaforU}). Now we obtain a system (\ref{quasi2})-(\ref{formulaforQ})-(\ref{pointvortexevolution}), which is equivalent to (\ref{vortex_model}). Let's denote
\begin{equation}
    G:=-\Re\{D_tQ\},
\end{equation}
\begin{equation}
    R:=\Re\{Q\}-b_1=\Re\{Q\}-2\Re \{(I-\mathbb{H})\bar{Q}\}-\Re\{[\bar{Q},\mathbb{H}](\frac{1}{Z_{\alpha}}-1)\}.
\end{equation}
(\ref{quasi2}) then becomes 
\begin{equation}\label{quasi3}
    \begin{cases}
     D_tU=A\Lambda W+G\\
     D_tW=-U-\Re\{[\bar{F},\mathbb{H}](\frac{1}{Z_{\alpha}}-1)\}+R,\\
     Z-\alpha=(I-\mathbb{H})W,\\
     F=(I-\mathbb{H})U.
    \end{cases}
\end{equation}
 Note that $G$ is determined by $Z, F$ and $\{z_j\}$ (and therefore determined by $W$, $U$, and $\{z_j\}$), so we can write $G$ as $G(W, U, \{z_j(t)\})$. Here, $\{z_j\}$ means $\{z_1(t),\cdots, z_N(t)\}$.  For the same reason, $D_tF$, $b_1$, $A$, $R$ are determined by $W, U$ and $\{z_j\}$. Therefore, we write
 \begin{align*}
     &G=G(W, U, \{z_j(t)\})\\
     &R=R(W, U, \{z_j(t)\})\\
     &b_1=b_1(W, U, \{z_j(t)\})\\
     &A=A(W, U, \{z_j(t)\})
 \end{align*}

\subsubsection{Some notations} For the reader's convenience, we list some notations as follows

\begin{align*}
\Sigma(t)\quad \quad  & \text{The free surface at time } t,\\
\Omega(t)\quad \quad  & \text{The fluid region at time } t, \Omega(t) \text{ is bounded above by } \Sigma(t),\\
\Phi(\cdot,t) \quad \quad & \text{The Riemann mapping from } \Omega(t) \text{ to } \mathbb{P}_-,\\
    Z(\alpha,t) \quad \quad  &\text{The free interface in Riemann variables},\\
    h(\alpha,t)\quad \quad & \text{The trace of } \Phi(\cdot,t), \text{ that is }, h(\alpha,t)=\Phi(Z(\alpha,t),t),\\
        b(\alpha,t)\quad \quad & b(\cdot,t)=h_t\circ h^{-1},\\
    D_t\quad \quad & D_t=\partial_t+b\partial_{\alpha},\\
    D_tZ\quad \quad & \text{The trace of the velocity field},\\
    A|Z_{\alpha}| \quad \quad &\text{The Taylor sign coefficient in Riemann mapping variable},\\
        z_{j}(t) \quad\quad  & \text{The coordinates of the $j$-th point vortex}, \\ &z_j(t)=x_j(t)+iy_j(t),\\
    \dot{z}_{j} \quad \quad & \text{The time derivative of the point vortex},\\
    d_I(t) \quad \quad & \text{The distance between the point vortices and the free interface at time } t, \text{that is, }\\ &d_I(t):=\inf_{\alpha\in\mathbb{R}}\min_{1\leq j\leq N}|Z(\alpha,t)-z_j(t)|,\\
    Q(\alpha,t)\quad\quad & \text{The conjugate of the velocity field generated by the point vortices},\\ &Q(\alpha,t)=-\sum_{j=1}^N\frac{\lambda_j i}{2\pi} \frac{1}{Z(\alpha,t)-z_j(t)}\\
        F(\alpha, t) \quad\quad   &\text{The wave part of the conjugate of the velocity field, that is}, F=D_t\bar{Z}-Q,\\
        \mathcal{U}(\cdot,t)\quad \quad &\text{The holomorphic extension of } F(\cdot,t), \text{that is }, F(\alpha,t)=\mathcal{U}(Z(\alpha,t),t),\\
    A_{1}(\alpha, t)\quad\quad & A_1(\alpha,t)=A|Z_{\alpha}|^2,\\
\    U \quad \quad & \text{The real part of the velocity field, that is }, U=\Re\{D_tZ\}, \\
    W(\alpha,t)\quad \quad & \text{The real part of } Z(\alpha,t)-\alpha:  W=\Re\{Z-\alpha\}.
\end{align*}
 We parametrize $\Sigma_0$ by Riemann mapping, that is, $h_0(\alpha)=\alpha$, so $z_0(\alpha)=Z_0(\alpha)$. Let $U_0$,  $W_0$, $z_{j,0}$, $Q_0$, $Q_1$ denote the initial value of $U$, $W$, $z_j$, $Q$, $D_tQ$, respectively.

 Using (\ref{formulaforQ}), we obtain
$$D_tQ=\sum_{j=1}^N\frac{\lambda_j i}{2\pi}\frac{D_tZ-\dot{z}_j(t)}{(Z(\alpha,t)-z_j(t))^2}.$$
Using (\ref{pointvortexevolution}), we have
\begin{equation}
    \dot{z}_{j,0}=\sum_{\substack{1\leq k\leq N\\ k\neq j}}\frac{\lambda_j i}{2\pi}\dfrac{1}{\overline{z_{k,0}-z_{j,0}}}+\frac{1}{2\pi i}\int_{-\infty}^{\infty}\frac{\partial_{\beta}Z_0(\beta)}{z_{j,0}-Z_0(\beta)}F_0(\beta)d\beta.
\end{equation}
\begin{equation}
    Q_1=\sum_{j=1}^N\frac{\lambda_j i}{2\pi}\frac{\bar{u}_0-\dot{z}_{j,0}}{(Z_0(\alpha)-z_{j,0})^2}.
\end{equation}
\begin{equation}\label{initialA1}
    A_{1,0}=1+\frac{1}{2\pi}\int \frac{|\bar{u}_0(\alpha)-\bar{u}_0(\beta)|^2}{(\alpha-\beta)^2}d\beta-\sum_{j=1}^N\frac{\lambda_j }{2\pi}\Re\Big\{\Big((I-\mathbb{H})\frac{\partial_{\alpha}Z_0(\alpha)}{(Z_0(\alpha)-z_j(t))^2}\Big)(\bar{u_0}(\alpha)-\dot{z}_{j,0})\Big\}.
\end{equation}
The initial interface $Z_0$ satisfies
\begin{equation}
    |Z_0(\alpha)-Z_0(\beta)|\geq C_0|\alpha-\beta|, \quad \quad \forall ~~\alpha,\beta\in \mathbb{R}.
\end{equation}

\subsection{Function spaces}
We define the Gevrey-2 spaces as follows.
\begin{definition}\label{Gevrey}
Let $\sigma>0$, we define
\begin{equation}
    \dot{X}_{\sigma}=\{f\in C^{\infty}(\mathbb{R}) ~\Big |\quad  \|f\|_{\dot{X}_{\sigma}}^2:=\sum_{j=1}^{\infty} \frac{\sigma^{2j}}{(j!)^{4}}\| \partial_{\alpha}^jf\|_{L^2}^2<\infty\}.
\end{equation}
\begin{equation}
    X_{\sigma}=\{f\in C^{\infty}(\mathbb{R}) ~\Big |\quad  \|f\|_{X_{\sigma}}^2:=\sum_{j=0}^{\infty} \frac{\sigma^{2j}}{(j!)^{4}}\| \partial_{\alpha}^jf\|_{L^2}^2<\infty\}.
\end{equation}
\begin{equation}
    \dot{Y}_{\sigma}=\{f\in C^{\infty}(\mathbb{R}) ~\Big |\quad  \|f\|_{\dot{Y}_{\sigma}}^2:=\sum_{j=1}^{\infty} \frac{j^2\sigma^{2j}}{(j!)^4}\| \partial_{\alpha}^jf\|_{L^2}^2<\infty\}.
\end{equation}

\begin{equation}
    Y_{\sigma}=\{f\in C^{\infty}(\mathbb{R}) ~\Big |\quad  \|f\|_{Y_{\sigma}}^2:=\|f\|_{L^2}^2+\sum_{j=0}^{\infty} \frac{j^2\sigma^{2j}}{(j!)^4}\| \partial_{\alpha}^jf\|_{L^2}^2<\infty\}.
\end{equation}
\end{definition}
It is not difficult to verify that $\dot{X}_{\sigma}, X_{\sigma}, \dot{Y}_{\sigma}$ and $Y_{\sigma}$ are Banach spaces (actually, Hilbert spaces.), we shall call these spaces the \emph{Gevrey spaces}, or \emph{Gevrey-2 spaces}, and call $\sigma$ the \emph{radius of convergence}.  Moreover, we have the Sobolev type embedding.
\begin{lemma}\label{sobolev}
\begin{itemize}
\item [1.] Let $f\in X_{\sigma}$ and $n\geq 0$ be an integer, then there is an absolute constant $C>0$ such that 
\begin{equation}\label{sobolev111}
    \|\partial_{\alpha}^nf\|_{L^{\infty}}\leq C\Big(\frac{((n+1)!)^2}{\sigma^{n+1}}+\frac{(n!)^2}{\sigma^n}\Big)\|f\|_{X_{\sigma}}.
\end{equation}\label{sobolev222}
\item [2.] Let $f\in Y_{\sigma}$ and $n\geq 1$ be an integer, then
\begin{equation}
    \|\partial_{\alpha}^nf\|_{L^{\infty}}\leq C\Big(\frac{((n+1)!)^2}{(n+1)\sigma^{n+1}}+\frac{(n!)^2}{n\sigma^n}\Big)\|f\|_{\dot{Y}_{\sigma}}.
\end{equation}
\end{itemize}
\end{lemma}
\begin{proof}
The proof follows from the standard Sobolev embedding $\|f\|_{L^{\infty}}\leq C\|f\|_{H^1}$.
\end{proof}

The Hilbert transform is a unitary operator on these spaces.
\begin{equation}
    \mathbb{H}f(\alpha):=\frac{1}{\pi i}\int_{\mathbb{R}}\frac{1}{\alpha-\beta}f(\beta)d\beta.
\end{equation}
\begin{lemma}\label{hilbert}
Let $\sigma>0$ and  $f\in X_{\sigma}$, we have
\begin{equation}\label{xxx}
    \|\mathbb{H}f\|_{X_{\sigma}}=\|f\|_{X_{\sigma}}, 
\end{equation}
\begin{equation}\label{xxx1111}
    \|\mathbb{H}f\|_{\dot{X}_{\sigma}}=\|f\|_{\dot{X}_{\sigma}}, 
\end{equation}

\begin{equation}\label{xxx12}
    \|\mathbb{H}f\|_{Y_{\sigma}}=\|f\|_{Y_{\sigma}}.
\end{equation}
and
\begin{equation}\label{xxx13}
    \|\mathbb{H}f\|_{\dot{Y}_{\sigma}}=\|f\|_{\dot{Y}_{\sigma}}.
\end{equation}

\end{lemma}
\begin{proof}
We prove (\ref{xxx}) only. (\ref{xxx12}) and (\ref{xxx13}) are proved similarly.
Using the fact that $\|\partial_{\alpha}^n\mathbb{H}f\|_{L^2}=\|\mathbb{H}\partial_{\alpha}^nf\|_{L^2}=\|\partial_{\alpha}^nf\|_{L^2}$, we have 
\begin{align*}
    \|\mathbb{H}f\|_{X_{\sigma}}^2=&\sum_{n\geq 0}\frac{\sigma^{2n}}{(n!)^4}\|\partial_{\alpha}^n\mathbb{H}f\|_{L^2}^2=\sum_{n\geq 0}\frac{\sigma^{2n}}{(n!)^4}\|\partial_{\alpha}^nf\|_{L^2}^2=\|f\|_{X_{\sigma}}^2.
\end{align*}
\end{proof}

\begin{lemma}[Product estimates]\label{lemmaproduct}
Let $\sigma\geq \sigma_0> 0$. We have 
\begin{equation}\label{product1}
    \|fg\|_{X_{\sigma}}\leq C\|f\|_{X_{\sigma}}\|g\|_{X_{\sigma}},
\end{equation}
\begin{equation}\label{product2}
    \|fg\|_{Y_{\sigma}}\leq C\|f\|_{Y_{\sigma}}\|g\|_{Y_{\sigma}},
\end{equation}
\begin{equation}\label{producthomo1}
    \|fg\|_{\dot{X}_{\sigma}}\leq C\|f\|_{\dot{X}_{\sigma}}\|g\|_{X_{\sigma}}+C\|f\|_{X_{\sigma}}\|g\|_{\dot{X}_{\sigma}},
\end{equation}
\begin{equation}\label{producthomo2}
    \|fg\|_{\dot{Y}_{\sigma}}\leq C\|f\|_{\dot{Y}_{\sigma}}\|g\|_{Y_{\sigma}}+C\|f\|_{Y_{\sigma}}\|g\|_{\dot{Y}_{\sigma}},
\end{equation}
for some constant $C>0$ depends linearly on $1+\frac{1}{\sigma_0}$.
\end{lemma}
\begin{proof}
We prove (\ref{product1}) only. 

\vspace*{1ex}

\noindent First,  note  that 
\begin{equation}
    \Big(\sum_{n=0}^{10} \frac{\sigma^{2n}}{(n!)^4}\|\partial_{\alpha}^n(fg)\|_{L^2}^2\Big)^{1/2}\leq C\|f\|_{X_{\sigma}}\|g\|_{X_{\sigma}}.
\end{equation}

\vspace*{1ex}

\noindent Second, for $n>10$, we estimate $\|\partial_{\alpha}^n(fg)\|_{L^2}$ by
\begin{align}
    \|\partial_{\alpha}^n(fg)\|_{L^2}\leq & \sum_{k=0}^n \frac{n!}{k!(n-k)!}\|f^{(k)}g^{(n-k)}\|_{L^2}\\
    \leq &\sum_{k=0}^{[n/2]}\frac{n!}{k!(n-k)!}\|f^{(k)}g^{(n-k)}\|_{L^2}+\sum_{k=[n/2]+1}^{n}\frac{n!}{k!(n-k)!}\|f^{(k)}g^{(n-k)}\|_{L^2}\\
    :=& I+II.
\end{align}
Here, $f^{(k)},  g^{(n-k)}$ represent the $k$-th and $(n-k)$-th derivative of $f$ and $g$, respectively. For $I$, we estimate $f^{(k)}$ in $L^{\infty}$ and estimate $g^{(n-k)}$ in $L^2$, and we have 
\begin{equation}\label{goodfaccc}
\begin{split}
    \|f^{(k)}g^{(n-k)}\|_{L^2}\leq &\|f^{(k)}\|_{L^{\infty}}\|g^{(k)}\|_{L^2}\leq C(\frac{(k!)^2}{\sigma^k}+\frac{((k+1)!)^2}{\sigma^{k+1}})\|f\|_{X_{\sigma}}\frac{((n-k)!)^2}{\sigma^{n-k}}\|g\|_{X_{\sigma}}\\
    =& C(1+\frac{(k+1)^2}{\sigma})\frac{(k!)^2((n-k)!)^2}{\sigma^n} \|f\|_{X_{\sigma}}\|g\|_{X_{\sigma}}.
\end{split}
\end{equation}
Moreover, we further decompose $I$ as 
\begin{align}
    I=\sum_{k=0}^{4}\frac{n!}{k!(n-k)!}\|f^{(k)}g^{(n-k)}\|_{L^2}+\sum_{k=5}^{[n/2]}\frac{n!}{k!(n-k)!}\|f^{(k)}g^{(n-k)}\|_{L^2}:=I_1+I_2.
\end{align}
For $n>10$ and $0\leq k\leq 4$, we have
\begin{equation}
    \frac{\sigma^n}{(n!)^2}I_1\leq Cn^{-2}\|f\|_{X_{\sigma}}\|g\|_{X_{\sigma}}.
\end{equation}
For $k\geq 5$ and $n> 10$, we have 
\begin{equation}\label{factorials}
    \frac{n!}{k!(n-k)!}=\frac{(n-k)!}{(n-k)!}\frac{\prod_{j=0}^{k-1}(n-j)}{k!}\geq Cn^4.
\end{equation}
Using (\ref{goodfaccc}) and (\ref{factorials}), we obtain
\begin{align*}
    \frac{\sigma^n}{(n!)^2}I_2\leq & C\sum_{k=5}^{[n/2]}\frac{k!(n-k)!}{n!}\frac{1}{(k!)^2((n-k)!)^2}(1+\frac{k^2}{\sigma})\frac{(k!)^2((n-k)!)^2}{\sigma^n} \|f\|_{X_{\sigma}}\|g\|_{X_{\sigma}}\\
    \leq & C\sum_{k=5}^{[n/2]}\frac{1}{n^4}(1+\frac{k^2}{\sigma})\|f\|_{X_{\sigma}}\|g\|_{X_{\sigma}}\\
    \leq & Cn^{-1}\|f\|_{X_{\sigma}}\|g\|_{X_{\sigma}}.
\end{align*}
So we obtain
\begin{align*}
    \frac{\sigma^n}{(n!)^2}I\leq Cn^{-1}\|f\|_{X_{\sigma}}\|g\|_{X_{\sigma}}.
\end{align*}
Similarly,
\begin{align*}
    \frac{\sigma^n}{(n!)^2}II\leq Cn^{-1}\|f\|_{X_{\sigma}}\|g\|_{X_{\sigma}}.
\end{align*}
Then we obtain
\begin{align*}
    \|fg\|_{X_{\sigma}}^2=&\sum_{n=0}^{\infty}\frac{\sigma^{2n}}{(n!)^4}\|\partial_{\alpha}^n (fg)\|_{L^2}^2\\
    \leq &\sum_{n=0}^{\infty}(1+n^2)^{-1}\|f\|_{X_{\sigma}}^2\|g\|_{X_{\sigma}}^2\\
    \leq & C\|f\|_{X_{\sigma}}^2\|g\|_{X_{\sigma}}^2.
\end{align*}

\end{proof}

We have also the following estimates.
\begin{lemma}\label{producttrilinear}
Let $\sigma\geq \sigma_0>0$ and $f, g\in X_{\sigma}$. If in addition $$\sum_{n=1}^{\infty}\frac{n^3\sigma^{2n}}{(n!)^4}\|\partial_{\alpha}^nf\|_{L^2}^2+\sum_{n=1}^{\infty}\frac{n\sigma^{2n}}{(n!)^4}\|\partial_{\alpha}^n g\|_{L^2}^2<\infty,$$
then 
\begin{equation}
\begin{split}
   \Big| \sum_{n=1}^{\infty}\frac{\sigma^{2n}}{(n!)^4}\langle \partial_{\alpha}^{n+1}(fg), \partial_{\alpha}^n g\rangle \Big|\leq & d_1\|f\|_{\dot{X}_{\sigma}}\|g\|_{\dot{X}_{\sigma}}^2+d_1\|g\|_{X_{\sigma}}\Big(\sum_{n=1}^{\infty}\frac{n^3\sigma^{2n}}{(n!)^4}\|\partial_{\alpha}^n f\|_{L^2}^2\Big)^{1/2}\Big(\sum_{n=1}^{\infty}\frac{n\sigma^{2n}}{(n!)^4}\|\partial_{\alpha}^n g\|_{L^2}^2\Big)^{1/2},
   \end{split}
\end{equation}
where $d_1=d(\sigma_0^{-1}+1)$ for some absolute constant $d>0$.  In particular, if $\sigma_0\geq 1$, then $d_1$ is an absolute constant.
\end{lemma}
\begin{proof}
We have 
\begin{align*}
    \partial_{\alpha}^{n+1}(fg)=g\partial_{\alpha}^{n+1}f+f\partial_{\alpha}^{n+1}g+\sum_{k=1}^{n+1}\frac{(n+1)!}{k!(n+1-k)!}\partial_{\alpha}^kf \partial_{\alpha}^{n+1-k}g.
\end{align*}
Using the same proof as in Lemma \ref{lemmaproduct}, we obtain
\begin{align*}
    \Big| \sum_{n=1}^{\infty}\frac{\sigma^{2n}}{(n!)^4}\langle \Big(\sum_{k=1}^{n+1}\frac{(n+1)!}{k!(n+1-k)!}\partial_{\alpha}^kf \partial_{\alpha}^{n+1-k}g\Big), \partial_{\alpha}^n g\rangle \Big|\leq d_1\|f\|_{\dot{X}_{\sigma}}\|g\|_{\dot{X}_{\sigma}}^2.
\end{align*}
By Cauchy-Schwarz inequality, we have  
\begin{align*}
    \Big| \sum_{n=1}^{\infty}\frac{\sigma^{2n}}{(n!)^4}\langle g\partial_{\alpha}^{n+1}f, \partial_{\alpha}^n g\rangle \Big|\leq & \Big(\sum_{n=1}^{\infty}\frac{\sigma^{2n}}{n(n!)^4}\|g\partial_{\alpha}^{n+1}f\|_{L^2}^2\Big)^{1/2}\Big(\sum_{n=1}^{\infty}\frac{n\sigma^{2n}}{(n!)^4}\|\partial_{\alpha}^n g\|_{L^2}^2\Big)^{1/2}\\
    \leq &\|g\|_{L^{\infty}}\Big(\sum_{n=1}^{\infty}\frac{\sigma^{2n}}{n(n!)^4}\|\partial_{\alpha}^{n+1}f\|_{L^2}^2\Big)^{1/2}\Big(\sum_{n=1}^{\infty}\frac{n\sigma^{2n}}{(n!)^4}\|\partial_{\alpha}^n g\|_{L^2}^2\Big)^{1/2}\\
    =&\|g\|_{L^{\infty}}\Big(\sum_{n=1}^{\infty}\frac{\sigma^{2(n+1)}}{((n+1)!)^4}\frac{((n+1)!)^4}{\sigma^2n(n!)^4}\|\partial_{\alpha}^{n+1}f\|_{L^2}^2\Big)^{1/2}\Big(\sum_{n=1}^{\infty}\frac{n\sigma^{2n}}{(n!)^4}\|\partial_{\alpha}^n g\|_{L^2}^2\Big)^{1/2}\\
    \leq & d_1\|g\|_{L^{\infty}}\Big(\sum_{n=1}^{\infty}\frac{n^3\sigma^{2n}}{(n!)^4}\|\partial_{\alpha}^n f\|_{L^2}^2\Big)^{1/2}\Big(\sum_{n=1}^{\infty}\frac{n\sigma^{2n}}{(n!)^4}\|\partial_{\alpha}^n g\|_{L^2}^2\Big)^{1/2},
\end{align*}
where we can take $d_1=2\sigma^{-1}+K$, for some absolute constant $K>0$.
Integration by parts,  we obtain
\begin{align*}
    \Big| \sum_{n=1}^{\infty}\frac{\sigma^{2n}}{(n!)^4}\langle f\partial_{\alpha}^{n+1}g, \partial_{\alpha}^n g\rangle \Big|=\Big| \sum_{n=1}^{\infty}\frac{\sigma^{2n}}{(n!)^4}\langle f_{\alpha}\partial_{\alpha}^{n}g, \partial_{\alpha}^n g\rangle \Big|\leq \|f_{\alpha}\|_{L^{\infty}}\|g\|_{\dot{X}_{\sigma}}^2\leq d_1\|f\|_{\dot{X}_{\sigma}}\|g\|_{\dot{X}_{\sigma}}^2.
\end{align*}
Using $\|g\|_{L^{\infty}}\leq 2(1+\sigma^{-1})\|g\|_{X_{\sigma}}$, we conclude the proof of the lemma.
\end{proof}

\subsection{Commutator estimates in Gevrey spaces}
Let $f,g\in X_{\sigma}$. Define
\begin{equation}
    S_1(g,f)(\alpha,t):=\frac{1}{\pi}p.v.\int_{\mathbb{R}}\frac{g(\alpha,t)-g(\beta,t)}{(\alpha-\beta)^2}f(\beta,t)d\beta.
\end{equation}
\begin{equation}
    S_2(g,f)h(\alpha,t):=\frac{1}{\pi}p.v.\int_{\mathbb{R}}\frac{(g(\alpha,t)-g(\beta,t))(f(\alpha,t)-f(\beta,t))}{(\alpha-\beta)^2}\frac{h_{\beta}}{Z_{\beta}}d\beta.
\end{equation}
The following commutator estimates are the Gevrey version of the Sobolev counterpart.
\begin{lemma}\label{commutator1}
\begin{itemize}
\item [1.] Let $f\in Y_{\sigma}$, $g\in X_{\sigma}$, then
\begin{equation}\label{commutator111}
    \|[f,\mathbb{H}]g_{\alpha}\|_{X_{\sigma}}\leq C\|f\|_{\dot{X}_{\sigma}}\|g\|_{X_{\sigma}}.
\end{equation}
\begin{equation}\label{commutator112}
    \|[f,\mathbb{H}]g_{\alpha}\|_{Y_{\sigma}}\leq C\|f\|_{\dot{Y}_{\sigma}}\|g\|_{X_{\sigma}}.
\end{equation}

\item [2.] Let $f, g, h\in X_{\sigma}$, $Z-\alpha\in \dot{Y}_{\sigma}$. There holds
\begin{equation}\label{commutator113}
    \|S_2(g,f)h\|_{X_{\sigma}}\leq C\|f\|_{X_{\sigma}}\|g\|_{X_{\sigma}}\|h\|_{X_{\sigma}}(1+\|Z-\alpha\|_{\dot{Y}_{\sigma}}).
\end{equation}

\item [(3.)]  If in addition $\sum_{n=1}^{\infty}\frac{n^3\sigma^{2n}}{(n!)^4}\|\partial_{\alpha}^n  f\|_{L^2}^2<\infty$, and $\sum_{n=1}^{\infty}\frac{n\sigma^{2n}}{(n!)^4}\|\partial_{\alpha}^n g\|_{L^2}^2<\infty$, then
\begin{equation}\label{commutator114}
    \|[f_{\alpha}, \mathbb{H}]g\|_{X_{\sigma}}\leq C\Big(\sum_{n=1}^{\infty}\frac{n^3\sigma^{2n}}{(n!)^4}\|\partial_{\alpha}^n  f\|_{L^2}^2+\sum_{n=1}^{\infty}\frac{n\sigma^{2n}}{(n!)^4}\|\partial_{\alpha}^n g\|_{L^2}^2\Big)+C\|g\|_{L^2}^2.
\end{equation}
\end{itemize}
\end{lemma}
\begin{proof}
We prove (\ref{commutator111}) only. It suffices to notice that 
\begin{equation}\label{easyhelp}
    \|[f,\mathbb{H}]g_{\alpha}\|_{L^2}\leq C\min\{ \|f\|_{L^{2}}\|g\|_{H^1},  \|f'\|_{L^{\infty}}\|g\|_{L^2}\}.
    \end{equation}
(\ref{easyhelp}) follows easily from Fourier analysis, we omit the proof. Using (\ref{easyhelp}), we obtain 
\begin{itemize}
\item [(1)] For $n=0$:
\begin{equation}
    \|[f,\mathbb{H}]g_{\alpha}\|_{L^2}\leq C\|f'\|_{L^{\infty}}\|g\|_{L^2}\leq C\|f\|_{\dot{X}_{\sigma}}\|g\|_{X_{\sigma}}.
\end{equation}

\item [(2)] For $n\geq 1$, we have 
\begin{equation}\label{help2}
    \|\partial_{\alpha}^n[f,\mathbb{H}]g_{\alpha}\|_{L^2}\leq C(\|\partial_{\alpha}^n\|_{L^2}\|g\|_{X_{\sigma}}.
\end{equation}
\end{itemize}
Indeed, let $n\geq 1$, using $|\xi|^n |sgn(\xi)-sgn(\eta)|\leq |\xi-\eta|^n$, we have 
\begin{align*}
    \Big|\widehat{\partial_{\alpha}^n [f,\mathbb{H}]g_{\alpha}}(\xi)\Big|\leq & |\xi|^n \int |sgn(\eta)-sgn(\xi)||\eta| \hat{f}(\xi-\eta)\hat{g}(\eta)|d\eta\\
    \leq & \int |\widehat{\partial_{\alpha}^n f}(\xi-\eta) \widehat{\partial_{\alpha}g}(\eta)| d\eta\\
    =& |\widehat{\partial_{\alpha}^n f}|\ast |\widehat{\partial_{\alpha}g}|(\xi)
\end{align*}
So 
\begin{align*}
    \|\partial_{\alpha}^n [f, \mathbb{H}]g_{\alpha}\|_{L^2}\leq & \||\widehat{\partial_{\alpha}^n f}|\ast |\widehat{\partial_{\alpha}g}|\|_{L^2}\leq \|\widehat{\partial_{\alpha}^n f}\|_{L^2}\|\widehat{\partial_{\alpha}g}\|_{L^1}.
\end{align*}
By Plancherel Theorem,  $\||\widehat{\partial_{\alpha}^n f}|\ast |\widehat{\partial_{\alpha}g}|\|_{L^2}=\|\partial_{\alpha}^n f\|_{L^2}$. Using
\begin{align*}
\widehat{g_{\alpha}}(\xi)=\frac{1}{|\xi|}|\xi|\widehat{g_{\alpha}}(\xi)=\frac{1}{|\xi|}\widehat{\partial_{\alpha}^2 g},
\end{align*}
we have
\begin{align*}
    \|\widehat{\partial_{\alpha}g}\|_{L^1}\leq \int_{|\xi|\leq 1}|\widehat{\partial_{\alpha}g}(\xi)|d\xi+\int_{|\xi|\geq 1}\frac{1}{|\xi|}|\widehat{\partial_{\alpha}^2 g}(\xi)|d\xi\leq \|g_{\alpha}\|_{H^1}\leq C\|g\|_{X_{\sigma}}.
\end{align*}
So we obtain (\ref{help2}), and therefore conclude the proof of the lemma.
\end{proof}

\begin{lemma}\label{reciprocal}
Let $H\in X_{\sigma}$ be such that $1+H\geq c_0$ for some constant $c_0>0$. Then

\begin{itemize}
\item [(1)] $\frac{1}{H+1}-1\in X_{\sigma}$, and 
\begin{equation}
    \norm{\frac{1}{H+1}-1}_{X_{\sigma}}\leq C(c_0)\|H\|_{X_{\sigma}}.
\end{equation}

\item [(2)] If in addition $\sum_{n=1}^{\infty}\frac{n\sigma^{2n}}{(n!)^4}\|\partial_{\alpha}^n H\|_{L^2}^2<\infty$, then 
\begin{equation}
    \sum_{n=1}^{\infty}\frac{n\sigma^{2n}}{(n!)^4}\norm{\partial_{\alpha}^n \Big(\frac{1}{H+1}-1\Big)}_{L^2}^2\leq C(c_0)\sum_{n=1}^{\infty}\frac{n\sigma^{2n}}{(n!)^4}\|\partial_{\alpha}^n H\|_{L^2}^2.
\end{equation}
\end{itemize}
Here $C(c_0)$ is a constant depending on $c_0$.
\end{lemma}
The proof of Lemma \ref{reciprocal} is similar to that of Lemma \ref{lemmaproduct}, so we omit the proof.

\begin{lemma}\label{nice}
Let $w\in \mathbb{P}_-$ and define $h(\alpha)=\frac{1}{(\alpha-\omega)^2}$. Then $h\in X_{\sigma}$ for any $\sigma\in \mathbb{R}$, and 
\begin{equation}
    \|h\|_{X_{\sigma}}\leq C\frac{1}{\sqrt{|Im\{w\}|}}e^{\frac{|\sigma|}{4\pi |Im\{w\}|}}
\end{equation}
\end{lemma}
\begin{proof}
The Fourier transform of $h$ is 
\begin{equation}
    \hat{h}(\xi)=-2\pi i (i\xi) e^{-w|\xi|}.
\end{equation}
For $\sigma\in \mathbb{R}$,
\begin{align*}
    \|h\|_{X_{\sigma}}^2=&\sum_{n=0}^{\infty}\frac{\sigma^{2n}}{(n!)^4}\|\partial_{\alpha}^n h\|_{L^2}^2\\
=&\sum_{n=0}^{\infty}\frac{\sigma^{2n}}{(n!)^4}\||\xi|^n \hat{h}(\xi)\|_{L^2}^2\\
=& \sum_{n=0}^{\infty}\frac{\sigma^{2n}}{(n!)^4} \int |\xi|^{2n+2} e^{2\pi Im\{w\}|\xi|}d\xi\\
=& 2\sum_{n=0}^{\infty}\frac{\sigma^{2n}}{(n!)^4}\frac{1}{|2\pi Im\{w\}|^{2n+1}}\int_0^{\infty} |\xi|^{2n} e^{-|\xi|}d\xi\\
=& 2\sum_{n=0}^{\infty}\frac{\sigma^{2n}}{(n!)^4}\frac{1}{|2\pi Im\{w\}|^{2n+1}} (2n+1)!\\
\leq & \frac{1}{2\pi |Im\{w\}|}\sum_{n=0}^{\infty} \Big(\frac{\sigma}{2\pi |Im\{w\}|}\Big)^{2n}\frac{2^n}{(n!)^2}\\
\leq & C \frac{1}{|Im\{w\}|}e^{\frac{|\sigma|}{2\pi |Im\{w\}|}},
\end{align*}
for some absolute constant $C>0$.
\end{proof}

\begin{lemma}\label{shifthalfderivative}
Let $\sigma>0$ and let $f\in X_{\sigma}$, $g, h\in Y_{\sigma}$. Then 
\begin{itemize}
\item [(1)]
\begin{equation}\label{trilinear1}
    \sum_{n=0}^{\infty}\frac{\sigma^{2n}}{(n!)^4}\Big|\langle\partial_{\alpha}^n (f\Lambda g), \partial_{\alpha}^n h\rangle \Big|\leq d_0\|f\|_{X_{\sigma}}\|g\|_{Y_{\sigma}}\|h\|_{Y_{\sigma}},
\end{equation}

\begin{equation}\label{trilinear2}
    \sum_{n=0}^{\infty}\frac{\sigma^{2n}}{(n!)^4}\Big|\langle\partial_{\alpha}^n (f\partial_{\alpha} g), \partial_{\alpha}^n h\rangle \Big|\leq d_0\|f\|_{X_{\sigma}}\|g\|_{Y_{\sigma}}\|h\|_{Y_{\sigma}},
\end{equation}

\begin{equation}\label{trilinear3}
    \sum_{n=0}^{\infty}\frac{\sigma^{2n}}{(n!)^4}\Big|\langle\partial_{\alpha}^n (f\partial_{\alpha} g), \partial_{\alpha}^n g\rangle \Big|\leq d_0\|f\|_{\dot{X}_{\sigma}}\|g\|_{X_{\sigma}}^2,
\end{equation}
where $d_0=K(1+\sigma^{-1})$ for some absolute constant $K>0$. In particular, if $\sigma\geq 1$, then $d_0$ is an absolute constant.
 
 \item [(2)] If in addition $\sum_{n=1}^{\infty}\frac{n^3\sigma^{2n}}{(n!)^4}\|\partial_{\alpha}^ng\|_{L^2}^2<\infty$, and $\sum_{n=1}^{\infty}\frac{n\sigma^{2n}}{(n!)^4}\|\partial_{\alpha}^n h\|_{L^2}^2<\infty$, then
 \begin{equation}\label{trilinear4}
    \sum_{n=1}^{\infty}\frac{\sigma^{2n}}{(n!)^4}\Big|\langle\partial_{\alpha}^n (f\Lambda g), \partial_{\alpha}^n h\rangle \Big|\leq d_0\|f\|_{X_{\sigma}}\Big(\sum_{n=1}^{\infty}\frac{n^3\sigma^{2n}}{(n!)^4}\|\partial_{\alpha}^ng\|_{L^2}^2\Big)^{1/2}\Big(\|h\|_{L^2}^2+\sum_{n=1}^{\infty}\frac{n\sigma^{2n}}{(n!)^4}\|\partial_{\alpha}^nh\|_{L^2}^2\Big)^{1/2}
\end{equation}
\end{itemize}
\end{lemma}
\begin{proof}
Let $n\geq 1$, we have 
\begin{align*}
    \Big|\langle\partial_{\alpha}^n (f\Lambda g), \partial_{\alpha}^n h\rangle \Big|\leq &\Big|\langle f\partial_{\alpha}^n\Lambda g, \partial_{\alpha}^n h\rangle \Big|+ \Big|\langle \partial_{\alpha}^n(f\Lambda g)-f\partial_{\alpha}^n\Lambda g), \partial_{\alpha}^n h\rangle \Big|\\
    :=& I_1+I_2.
\end{align*}
For $I_1$, we have 
\begin{align*}
    |I_1|=& \Big|\int f(\partial_{\alpha}^n\Lambda g)\overline{\partial_{\alpha}^n h} \Big|\\
    \leq & \int \Big|f(\partial_{\alpha}^n\Lambda^{1/2} g)\overline{\Lambda^{1/2}\partial_{\alpha}^n h}\Big|+\int \Big|g[f, \Lambda^{1/2}] \overline{\partial_{\alpha}^n h}\Big|\\
    \leq & d_0\|f\|_{X_{\sigma}} \|\Lambda^{1/2}\partial_{\alpha}^n g\|_{L^2}\|\Lambda^{1/2}\partial_{\alpha}^n h\|_{L^2}.
\end{align*}
So we obtain
\begin{align*}
    \sum_{n=0}^{\infty}\frac{\sigma^{2n}}{(n!)^4} I_1\leq &C\|f\|_{X_{\sigma}} \sum_{n=0}^{\infty}\frac{\sigma^{2n}}{(n!)^4}\|\Lambda^{1/2}\partial_{\alpha}^n g\|_{L^2} \|\Lambda^{1/2}\partial_{\alpha}^n h\|_{L^2}\leq d_0\|f\|_{X_{\sigma}}\|g\|_{Y_{\sigma}}\|h\|_{Y_{\sigma}}.
\end{align*}
The term $I_2$ can be handled in the same way as in the proof of Lemma \ref{lemmaproduct}, we omit the details. We have 
\begin{align*}
    \sum_{n=0}^{\infty}\frac{\sigma^{2n}}{(n!)^4} I_1\leq d_0\|f\|_{X_{\sigma}}\|g\|_{Y_{\sigma}}\|h\|_{Y_{\sigma}}.
\end{align*}
So we complete the proof of (\ref{trilinear1}). The proof of (\ref{trilinear2}) is the same as that of (\ref{trilinear1}). 

The proof of (\ref{trilinear3}) is similar:
\begin{align*}
    \Big|\langle\partial_{\alpha}^n (f\Lambda g), \partial_{\alpha}^n g\rangle \Big|\leq &\Big|\langle f\partial_{\alpha}^n\Lambda g, \partial_{\alpha}^n g\rangle \Big|+ \Big|\langle \partial_{\alpha}^n(f\Lambda g)-f\partial_{\alpha}^n\Lambda g), \partial_{\alpha}^n g\rangle \Big|\\
    :=& \it{II}_1+\it{II}_2.
\end{align*}
The treatment of $\it{II}_2$ is the same as that for $I_2$. For $\it{II}_1$, we have 
\begin{align*}
    \it{II}_1=& \Big|Re \int  f\partial_{\alpha}^n\Lambda g  \overline{ \partial_{\alpha}^n g} \Big|
    = \Big| Re \int |\partial_{\alpha}^n g|^2\Lambda f \Big|\\
    \leq & d_0\|\Lambda f\|_{L^{\infty}}\|\partial_{\alpha}^n g\|_{L^2}^2\\
    \leq & d_0\|f\|_{\dot{X}_{\sigma}}\|\partial_{\alpha}^n g\|_{L^2}^2.
\end{align*}
So $\sum_{n=0}^{\infty}\frac{\sigma^{2n}}{(n!)^4} I_1\leq C\|f\|_{\dot{X}_{\sigma}}\|g\|_{X_{\sigma}}^2 $. So we obtain (\ref{trilinear3}). 

For the proof of (\ref{trilinear4}), it suffices to notice that 
\begin{align*}
    \Big|\int f(\partial_{\alpha}^n\Lambda g)\overline{\partial_{\alpha}^n h} \Big|   \leq & \int \Big|f(\partial_{\alpha}^n\Lambda^{3/4} g)\overline{\Lambda^{1/4}\partial_{\alpha}^n h}\Big|+\int \Big|g[f, \Lambda^{1/4}] \overline{\partial_{\alpha}^n h}\Big|\\
    \leq & d_0\|f\|_{X_{\sigma}}\|\Lambda \partial_{\alpha}^n g\|_{L^2}\|(1+\Lambda^{1/4})\partial_{\alpha}^n h\|_{L^2}.
\end{align*}
\end{proof}

\section{The main results: quantitative statements}\label{maintheorems}
Throughout the rest of this paper, for brevity, we consider a pair of symmetric and counter-rotating point vortices embedded in the water waves, the general situation can actually be treated similarly. We assume 
\begin{equation}
    \omega(\cdot,t)=\lambda \delta_{z_1(t)}-\lambda \delta_{z_2(t)}, 
\end{equation}
where $\lambda\in \mathbb{R}$, and $z_1(t)=-x(t)+iy(t)$, $z_2(t)=x(t)+iy(t)$, with $x(t)>0$ and $z_1(t), z_2(t)\in \Omega(t)$. We assume also that $\Omega(t)$ is symmetric about the vertical axis \footnote{Such symmetry assumption can be removed, we assume it for convenience.}.  Without loss of generality, we assume $x(0)= 1$. Moreover, we assume $|y(0)|\gg 1$. It should be clear from the proof that regarding the local wellposedness in Gevrey spaces, these assumptions are unnecessary.

\subsection{The initial data}
Let the initial fluid region be given by $\Omega_0$, with a nonself-intersect smooth free surface $\Sigma_0$. We parametrize $\Sigma_0$ by Riemann variable, that is, we choose $\alpha$ such that $h_0(\alpha)=\alpha$.

Let $Z_0(\alpha):=Z(\alpha,0)$. $U_0=\Re\{F_0\}$, and $W_0=\Re\{Z_0-\alpha\}$. Recall that 
$$d_I(t)=\inf_{\alpha\in\mathbb{R}}\min_{1\leq j\leq 2}|Z(\alpha,t)-z_j(t)|,$$
and $d_{I,0}:=d_{I}(0)$,   $x_0:=x(0)=1$. Without loss of generality, we assume $$d_{I,0}=\inf_{\alpha\in \mathbb{R}}\min_{j=1,2}|\Im\{Z(\alpha,0)-z_j(0)\}|.$$ 
Let $x_0, y_0$, $\lambda$, and $L_0\geq 4$ be given constants. Denote $\phi(t)=L_0-\delta_0 t$. We assume the following:
\begin{itemize}
    \item [(H1)] $(W_0, U_0)$ is given such that $(\partial_{\alpha}W_0, U_0)\in X_{L_0}\times Y_{L_0}$. Without loss of generality, we assume for simplicity that\footnote{In application, we will take $ \|(U_0, \partial_{\alpha}W_0)\|_{Y_{L_0}\times X_{L_0}}$ to be small.
}
\begin{equation}
    \|(U_0, \partial_{\alpha}W_0)\|_{Y_{L_0}\times X_{L_0}}\leq 1/2.
\end{equation}

\vspace*{1ex}
    
    \item [(H2)]  $\delta_0$ is chosen such that 
\begin{equation}\label{delta0equation}
 \frac{\delta_0}{L_0}-4-2d_0(\|A(\cdot,0)\|_{L^{\infty}}+\|A(\cdot,0)\|_{\dot{Y}_{\phi(t)}})-4d_1\|(U_0, \partial_{\alpha}W_0\|_{Y_{\phi(0)}\times X_{\phi(0)}}\geq 0.
\end{equation}
Here, $A(\cdot,0)=\frac{A_{1,0}}{|\partial_{\alpha}Z_0|^2}$, with $A_{1,0}$ given by (\ref{initialA1}), and $d_0$ and $d_1$ are the constants given in Lemma \ref{shifthalfderivative} and Lemma \ref{producttrilinear}, respectively.

\vspace*{1ex}

\item [(H3)] $z_{1,0}=-x_0+iy_0$, $z_{2,0}=x_0+iy_0$, where $x_0$ and $y_0$ are constants, satisfying $x_0=1$, $y_0<0$, and $|y_0|\gg 1$. 

\vspace*{1ex}

\item [(H4)] $|Z_0(\alpha)-Z_0(\beta)|\geq C_0|\alpha-\beta|$, for some absolute constant $C_0>0$.

\vspace*{1ex}

\item [(H5)] $Z_0(\alpha):=\alpha+(I+\mathbb{H})W_0$, $F_0=(I+\mathbb{H})U_0$.

\item [(H6)] $\Omega_0$ is symmetric about the $y$-axis, $W_0$ is an odd function. $U_0$ is also an odd function.
\end{itemize}

\vspace*{1ex}

\noindent For $0\leq t\leq \frac{L_0}{2\delta_0}$, we have $\phi(t)\geq \frac{1}{2}L_0$. 
Denote
\begin{equation}\label{equationtau0}
    \tau_0:= \frac{1}{4+2d_0(\|A(\cdot,0)\|_{L^{\infty}}+\|A(\cdot,0)\|_{\dot{Y}_{\phi(t)}})+4d_1\|(U_0, \partial_{\alpha}W_0\|_{Y_{\phi(0)}\times X_{\phi(0)}}}.
\end{equation}
In particular, $\delta_0\geq 4L_0\geq 16$.

\begin{remark}
The assumption that $|y_0|\gg 1$ is not essential in terms of the local wellposedness in Gevrey-2 spaces for the system (\ref{quasi3}). We can certainly remove it. Nevertheless, for the purpose of proving Theorem \ref{main}, we do need to assume $|y_0|\gg 1$. Therefore, for convenience, we make such an assumption. 
\end{remark}

\subsection{The main theorems}

\begin{theorem}\label{theorem1}
Let $W_0$, $U_0$, $\{z_{j,0}\}$, $L_0$, $\delta_0$, $Z_0$, and $F_0$ be given such that (H1)-(H6) hold. Then there exists a constant $T>0$, such that the system (\ref{quasi3}) with initial data $(W_0, U_0, \{z_{j,0}\})$ admits a unique solution $(W, U,  \{z_j\})$ satisfying
\begin{itemize}
\item [(a)] $(W_{\alpha}, U)\in C([0,T]; X_{\phi(t)}\times Y_{\phi(t)}))$, and
$z_j\in C^1([0, T]; \Omega(t)), \quad j=1, 2$. Moreover,
    $$\sup_{0\leq t\leq T}\Big(\|U(\cdot,t)\|_{\dot{Y}_{\phi(t)}}^2+\|W_{\alpha}\|_{X_{\phi(t)}}^2\Big)\leq  \begin{cases} \|(U_0, \partial_{\alpha}W_0)\|_{\dot{Y}_{\phi(t)}\times X_{\phi(t)}}^2+C|\lambda|(d_{I,0}-\frac{|\lambda|}{4\pi x(0)}T)^{-5/2}, \quad &\lambda>0, \\
    \|(U_0, \partial_{\alpha}W_0)\|_{\dot{Y}_{\phi(t)}\times X_{\phi(t)}}^2+C|\lambda|d_{I,0}^{-5/2}, \quad \quad & \lambda<0
    \end{cases}
$$
and
$$\sup_{0\leq t\leq T}\|U(\cdot,t)\|_{L^2}^2\leq 1, \quad \quad \sup_{0\leq t\leq T}\|U(\cdot,t)\|_{L^{\infty}}^2\leq Cd_{I,0}^{-1/3}.$$
Here, the constant $C$ depends on $C_0$  and $\frac{1}{L_0}$.

\item [(b)] 
$|Z(\alpha,t)-Z(\beta,t)|\geq \frac{1}{2}C_0|\alpha-\beta|,\quad \alpha, \beta\in \mathbb{R},\quad t\in [0,T].$

\item [(c)]
$ d_{I}(t)\geq \frac{1}{2}|y_0|^{9/10}, \quad \quad \frac{1}{2}x(0)\leq x(t)\leq 2x(0), \quad \quad t\in [0,T].$

\item [(d)] $\Omega(t)$ is symmetric about the $y$-axis. For each fixed $t\in [0,T]$, $W(\cdot,t)$ and $U(\cdot,t)$ are odd functions. $\Re\{z_1(t)\}=-\Re\{z_2(t)\}$, $\Im\{z_1(t)\}=\Im\{z_2(t)\}$.
\end{itemize}
Here \footnote{The power $\frac{9}{10}$ for $|y_0|^{9/10}$ is not optimal. We can use any $|y_0|^{1-\epsilon}$, where $\epsilon\in (0,1)$.}, $$T=\begin{cases}\min\{T_1(C_0, \|(U_0, \partial_{\alpha}W_0)\|_{\dot{Y}_{L_0}\times X_{L_0}}), \frac{L_0}{2\delta_0}, \frac{4\pi x(0)(|y_0|-|y_0|^{9/10})}{|\lambda|}\}, \quad \quad &\lambda>0,\\
\min\{T_2(C_0, \|(U_0, \partial_{\alpha}W_0)\|_{\dot{Y}_{L_0}\times X_{L_0}}), \frac{L_0}{2\delta_0}\}\quad \quad \quad &\lambda<0,\end{cases}$$
with $T_1, T_2$ depend continuously on its parameters.
In particular, if $|y_0|$ is sufficiently large, then if $\lambda>0$, we can take $T=\frac{4\pi x(0)(|y_0|-|y_0|^{9/10})}{|\lambda|}$; if $\lambda<0$, we can take $T=O(1)$.

\end{theorem}
\begin{remark}\label{remarklwp}
(1) Using the same proof as in \cite{su2018long}, we can prove that the solutions to the water waves preserve the symmetries in (H6). We shall omit the proof of (d) and refer the readers to Theorem 5 in \cite{su2018long}.

(2) From the discussion in \S \ref{thegevreyframework}, in particular, the equation (\ref{sharpgevrey}), our wellposedness result is sharp in the Gevrey spaces. That is, if we define 
$$X_{\sigma, k}:=\{f\in C^{\infty}:\quad \|f\|_{X_{\sigma, k}}^2=\sum_{n=0}^{\infty}\frac{\sigma^{2n}}{(n!)^{2k}}\|\partial_{\alpha}^n f\|_{L^2}^2\},$$
and $\dot{X}_{\sigma, k}, Y_{\sigma, k}, \dot{Y}_{\sigma,k}$ similarly. Then (\ref{quasi3}) is illposedness in space $(U, W_{\alpha})\in Y_{\phi(t), k}\times X_{\phi(t), k}$ for $k>2$. We shall prove this in a separate paper.

\end{remark}

Using \emph{Theorem} \ref{theorem1}, we are able to prove the following.
\begin{theorem}\label{taylorsignfailtheorem2}
Let $(W, U, \{z_1(t), z_2(t)\})$ be the unique solution to (\ref{quasi3}) on $[0,T]$ constructed in Theorem \ref{theorem1} with initial data $(W_0, U_0, \{z_1(0), z_2(0)\})$. For any given constants $0<\eta_0<1$, $\eta_1>0$, $\gamma>0$ and $0<\epsilon_0\ll 1$,
there exist constants $N_0\gg 1$, $\delta_0, \delta_1, \delta_2\ll 1$ such that for all 
\begin{align*}
   & y(0)\leq  -N_0\quad \quad  \lambda=\gamma |y_0|^{3/2-\epsilon_0},\quad \quad |x(0)-1|\leq \epsilon_1,\quad \quad \|(U_0, \partial_{\alpha}W_0)\|_{Y_{L_0}\times X_{L_0}}\leq  \epsilon_2,
\end{align*}
there holds,
 $$\inf_{\alpha\in \mathbb{R}}A_1(\alpha,0)\geq 1-\eta_0, \quad\quad  \inf_{\alpha\in \mathbb{R}}A_1(\alpha,T)<-\eta_1,$$
For example, we can take $\epsilon_1=\frac{1}{N_0}$, $\epsilon_2=\frac{1}{N_0}$, $\epsilon_0=\frac{1}{10}$, $L_0=10$, $\delta_0=1000$.
 
\end{theorem}

\begin{remark}
The choice of the parameters $\epsilon_0$, $\epsilon_1$, $\epsilon_2$ , etc are certainly not optimal. It is not our primary goal in this paper to obtain those sharp bounds.
\end{remark}

\section{A quasilinear system}\label{sectionlinear}
Given $b_1$, $G$, $R$, $A$, $U_0$ and $W_0$, we consider the following quasilinear system (with $G$ and $R$ as external forces):
\begin{equation}\label{linearsystem}
    \begin{cases}
D_tU=A\Lambda W+G,\\
D_tW=-U-\Re\{[\bar{F},\mathbb{H}](\frac{1}{Z_{\alpha}}-1)\}+R,\\
(U, W)(\cdot,0)=(U_0, W_0).
    \end{cases}
\end{equation}
Here, $Z_{\alpha}=(I+\mathbb{H})W_{\alpha}+1$. $D_t=\partial_t+b\partial_{\alpha}$, with 
$$b=b_0+b_1,  \quad \quad b_0=2U+\Re\{[(I-\mathbb{H})U,\mathbb{H}](\frac{1}{Z_{\alpha}}-1)\},$$
where $b_1$ is an a priori given real valued function. Using $\partial_{\alpha}[f, \mathbb{H}]g=[f_{\alpha}, \mathbb{H}]g+[f, \mathbb{H}]g_{\alpha}$, and by (\ref{commutator111}) and (\ref{commutator114}), we have 
\begin{equation}\label{b0estimate}
\begin{split}
   & \|\partial_{\alpha}\Re\{[(I-\mathbb{H})U,\mathbb{H}](\frac{1}{Z_{\alpha}}-1)\}\|_{X_{\phi(t)}}\\
   \leq & C(\|U\|_{\dot{X}_{\phi(t)}}^2+\|\frac{1}{Z_{\alpha}}-1\|_{X_{\phi(t)}}^2)+C\Big(\sum_{n=1}^{\infty}\frac{n^3\phi(
   t)^{2n}}{(n!)^4}\|\partial_{\alpha}^n  U\|_{L^2}^2+\sum_{n=1}^{\infty}\frac{n\phi(t)^{2n}}{(n!)^4}\|\partial_{\alpha}^n (\frac{1}{Z_{\alpha}}-1)\|_{L^2}^2\Big),
\end{split}
\end{equation}
for some constant $C>0$ depending on $\frac{1}{\phi(t)}$.

Denote
$$H:=(I+\mathbb{H})W_{\alpha}.$$
Let $\phi(t)=\phi_{\delta_0, L_0}(t):=L_0-\delta_0t$. Let 
\begin{equation}
    E_{L_0, \delta_0}(t):=\frac{1}{2}\|(U(\cdot,t), \partial_{\alpha}W(\cdot,t))\|_{\dot{Y}_{\phi(t)}\times X_{\phi(t)}}^2.
\end{equation}
By Lemma \ref{reciprocal}, if $\inf_{\alpha\in \mathbb{R}}(H+1)\geq c_0/2$, then 
\begin{equation}\label{reciprocalone}
    \norm{\frac{1}{H+1}-1}_{X_{\phi(t)}}\leq C(c_0)\|H\|_{X_{\phi(t)}}\leq C(c_0)\|W_{\alpha}\|_{X_{\phi(t)}},
\end{equation}
and
\begin{equation}\label{reciprocaltwo}
    \sum_{n=1}^{\infty}\frac{n\sigma^{2n}}{(n!)^4}\norm{\partial_{\alpha}^n \Big(\frac{1}{H+1}-1\Big)}_{L^2}^2\leq C(c_0)\sum_{n=1}^{\infty}\frac{n\sigma^{2n}}{(n!)^4}\|\partial_{\alpha}^n H\|_{L^2}^2\leq C(c_0)\sum_{n=1}^{\infty}\frac{n\sigma^{2n}}{(n!)^4}\|\partial_{\alpha}^n W_{\alpha}\|_{L^2}^2.
\end{equation}

\begin{theorem}\label{theoremlinear}
Let $c_0>0$ and $L_0\geq 4$ be given constants. Assume that $b_1\in C([0,T]; Y_{\phi(t)})$, $A-1, R_{\alpha}\in C([0,T]; X_{\phi(t)})$, $G\in C([0,T]; Y_{\phi(t)})$ and $(U_0, \partial_{\alpha}W_0)\in Y_{\phi(0)}\times X_{\phi(0)}$. Assume further that $\inf_{\alpha\in\mathbb{R}}(1+H(\alpha,0))\geq c_0$. Let $d_0$ and $d_1$ be the constant given in Lemma \ref{shifthalfderivative} and Lemma \ref{producttrilinear}, respectively. Assume that 
\begin{equation}\label{controldamping}
    \sup_{0\leq t\leq T}\Big(\frac{\delta_0}{L_0}-4-2d_0(\|A\|_{L^{\infty}}+\|A\|_{\dot{Y}_{\phi(t)}})-4d_1\|(U_0, \partial_{\alpha}W_0\|_{Y_{\phi(0)}\times X_{\phi(0)}}\Big)\geq 0.
\end{equation}

Then there exists $0<T_0\leq T$ such that (\ref{linearsystem}) admits a unique solution $(U, \partial_{\alpha}W)\in C^0([0,T]; Y_{\phi(t)}\times X_{\phi(t)})$. Moreover,
\begin{equation}
    E_{L_0, \delta_0}(t)\leq E_{L_0, \delta_0}(0)e^{B_{T_0}t}+\int_0^t e^{B_{T_0}\tau}\mathcal{N}(\tau)d\tau,
\end{equation}
\begin{equation}
    \|U(\cdot,t)\|_{L^2}^2\leq U(\cdot,0)\|_{L^2}^2 e^{\gamma(t)}+\int_0^t  e^{\gamma(t)-\gamma(s)}(\|G(s)\|_{L^2}^2+\|A(s)\|_{L^{\infty}}E_{L_0,\delta_0}(s))ds,
\end{equation}
where 
\begin{equation}
    B_{T_0}:=C(1+\|b_1\|_{C(;0,T_0];Y_{\phi(t))}}),
\end{equation}
and
\begin{equation}\label{mathcalN_0}
    \mathcal{N}(t):=\|R_{\alpha}\|_{X_{\phi(t)}}^2+\|G\|_{\dot{Y}_{\phi(t)}}^2+C\|b_1\|_{\dot{Y}_{\phi(t)}}^2,
\end{equation}
$$\gamma(t):= \int_0^t C(E_{L_0,\delta_0}^{1/2}(\tau)+\|b_1(\cdot,\tau)\|_{\dot{Y}_{\phi(\tau)}}+\|A(\cdot,t)\|_{L^{\infty}})d\tau.$$

Here, $T_0=\min\{T, \frac{L_0}{2\delta_0}\}$, and $C$ is a constant depending on $c_0$ and $\frac{1}{L_0}$.

\end{theorem}
We prove \emph{Theorem} \ref{theoremlinear} by the energy method. We provide closed a priori energy estimates only. Note that $T_0\leq \frac{L_0}{2\delta_0}$ implies that $\phi(t)\geq \frac{L_0}{2}\geq 2$ for $t\in [0,T_0]$, so $d_0$ and $d_1$ are absolute constants. By Lemma \ref{commutator1} and (\ref{reciprocaltwo}), under the assumption of \emph{Theorem} \ref{theoremlinear}, we have the a priori estimate
\begin{equation}\label{estimateb0b0}
\begin{split}
    \|b_0\|_{\dot{Y}_{\phi(t)}}^2=&\sum_{n=1}^{\infty}\frac{n^2\phi(t)^{2n}}{(n!)^4}\norm{\partial_{\alpha}^n \Big(2U+\Re\{[(I-\mathbb{H})U,\mathbb{H}](\frac{1}{Z_{\alpha}}-1)\}\Big)}_{L^2}^2\\
    \leq & C\|U\|_{\dot{Y}_{\phi(t)}}^2+C\|U\|_{\dot{Y}_{\phi(t)}}\|\partial_{\alpha}W\|_{X_{\phi(t)}}\\
    \leq & CE_{L_0,\delta_0},
\end{split}
\end{equation}
for some constant $C>0$ depending on $c_0$ and $\frac{1}{L_0}$.

\begin{proof}
Applying $\partial_{\alpha}$ on both sides of $D_tW=-U-\Re\{[\bar{F},\mathbb{H}](\frac{1}{Z_{\alpha}}-1)\}+R$, we obtain
\begin{equation}\label{Dtwalpha}
    D_t W_{\alpha}=-U_{\alpha}+R_{\alpha}-b_{\alpha}W_{\alpha}-\partial_{\alpha}\Re\{\bar{F},\mathbb{H}](\frac{1}{Z_{\alpha}}-1)\}.
\end{equation}

Using (\ref{Dtwalpha}) and $U_t=A\Lambda W-bU_{\alpha}+G$, we have 
\begin{align*}
    \frac{d}{dt}E_{L_0, \delta_0}(t)=& \frac{1}{2}\frac{d}{dt}\Big(\sum_{n=1}^{\infty}\frac{n^2\phi(t)^{2n}}{(n!)^4}\|\partial_{\alpha}^n U(\cdot,t)\|_{L^2}^2\Big)+\frac{1}{2}\frac{d}{dt}\sum_{n=0}^{\infty}\frac{\phi(t)^{2n}}{(n!)^4}\|\partial_{\alpha}^n W_{\alpha}(\cdot,t)\|_{L^2}^2\\
    =& -\delta_0\sum_{n=1}^{\infty}\frac{n^3 \phi(t)^{2n-1}}{(n!)^4}\|\partial_{\alpha}^n U(\cdot,t)\|_{L^2}^2+\Re\sum_{n=1}^{\infty}\frac{n^2\phi(t)^{2n}}{(n!)^4}\langle \partial_{\alpha}^n U_t, \partial_{\alpha}^n U\rangle\\
    &-\delta_0\sum_{n=1}^{\infty}\frac{n \phi(t)^{2n-1}}{(n!)^4}\|\partial_{\alpha}^n W_{\alpha}(\cdot,t)\|_{L^2}^2+\Re\sum_{n=0}^{\infty}\frac{\phi(t)^{2n}}{(n!)^4}\langle \partial_{\alpha}^n \partial_{\alpha}W_t, \partial_{\alpha}^n W_{\alpha}\rangle\\
    =& -\delta_0\sum_{n=1}^{\infty}\frac{n^3 \phi(t)^{2n-1}}{(n!)^4}\|\partial_{\alpha}^n U(\cdot,t)\|_{L^2}^2\\
    &+\Re\sum_{n=1}^{\infty}\frac{n^2\phi(t)^{2n}}{(n!)^4}\langle \partial_{\alpha}^n (A\Lambda W-bU_{\alpha}+G), \partial_{\alpha}^n U\rangle\\
    &-\delta_0\sum_{n=1}^{\infty}\frac{n\phi(t)^{2n-1}}{(n!)^4}\|\partial_{\alpha}^n W_{\alpha}(\cdot,t)\|_{L^2}^2\\
    &+\Re\sum_{n=0}^{\infty}\frac{\phi(t)^{2n}}{(n!)^4}\langle \partial_{\alpha}^n (-U_{\alpha}+R_{\alpha}-b_{\alpha}W_{\alpha}-\partial_{\alpha}\Re\{[\bar{F},\mathbb{H}](\frac{1}{Z_{\alpha}}-1)\}-b\partial_{\alpha}W_{\alpha}), \partial_{\alpha}^n W_{\alpha}\rangle
\end{align*}
Using Lemma \ref{shifthalfderivative}, we have \footnote{It turns out that if we choose $|y_0|$ large and $|\lambda|\sim |y_0|^{3/2}$, then $\|A-1\|_{L^{\infty}}\sim 1$, while $\|A-1\|_{L^{2}}\sim |y_0|^{1/2}$, which is large. So we avoid to bound $A-1$ in $L^2$.}
\begin{align*}
    &\Big|\sum_{n=1}^{\infty}\frac{n^2\phi(t)^{2n}}{(n!)^4}\langle \partial_{\alpha}^n(A\Lambda W), \partial_{\alpha}^n U\rangle \Big|\\
    \leq & d_0(\|A\|_{L^{\infty}}+\|A\|_{\dot{X}_{\phi(t)}})(\sum_{n=1}^{\infty}\frac{n\phi(t)^{2n}}{(n!)^4}\|\partial_{\alpha}^n \Lambda W\|_{L^2}^2+\sum_{n=1}^{\infty}\frac{n^3\phi(t)^{2n}}{(n!)^4}\|\partial_{\alpha}^n U\|_{L^2}^2),
\end{align*}
for some absolute constant $d_0>0$.
Using Lemma \ref{shifthalfderivative} again, we obtain 
\begin{align*}
    \Big| \Re\sum_{n=1}^{\infty} & \frac{n^2\phi(t)^{2n}}{(n!)^4}\langle \partial_{\alpha}^n (A\Lambda W-bU_{\alpha}+G), \partial_{\alpha}^n U\rangle\Big|\leq  C((\|b\|_{L^{\infty}}+\|b\|_{\dot{Y}_{\phi(t)}})\|U\|_{\dot{Y}_{\phi(t)}}^2+\|U\|_{\dot{Y}_{\phi(t)}}^2+\|G\|_{\dot{Y}_{\phi(t)}}^2\\
    &+d_0(\|A\|_{L^{\infty}}+\|A\|_{\dot{Y}_{\phi(t)}})(\sum_{n=1}^{\infty}\frac{n\phi(t)^{2n}}{(n!)^4}\|\partial_{\alpha}^n \Lambda W\|_{L^2}^2+\sum_{n=1}^{\infty}\frac{n^3\phi(t)^{2n}}{(n!)^4}\|\partial_{\alpha}^n U\|_{L^2}^2)\\
    \leq & C(1+\|b\|_{L^{\infty}}+\|b\|_{\dot{Y}_{\phi(t)}})E_{L_0, \delta_0}(t)+\|G\|_{\dot{Y}_{\phi(t)}}^2\\
    &+d_0(\|A\|_{L^{\infty}}+\|A\|_{\dot{X}_{\phi(t)}})(\sum_{n=1}^{\infty}\frac{n\phi(t)^{2n}}{(n!)^4}\|\partial_{\alpha}^n \Lambda W\|_{L^2}^2+\sum_{n=1}^{\infty}\frac{n^3\phi(t)^{2n}}{(n!)^4}\|\partial_{\alpha}^n U\|_{L^2}^2).
\end{align*}
Cauchy-Schwarz implies
\begin{align*}
    &\Big| \Re\sum_{n=1}^{\infty}\frac{n^2\phi(t)^{2n}}{(n!)^4}\langle \partial_{\alpha}^n U_{\alpha}, \partial_{\alpha}^n W_{\alpha}\rangle\Big|\leq \sum_{n=1}^{\infty}\frac{n^3\phi(t)^{2n}}{(n!)^4}\|\partial_{\alpha}^n U\|_{L^2}^2+\sum_{n=1}^{\infty}\frac{n\phi(t)^{2n}}{(n!)^4}\|\partial_{\alpha}^n W_{\alpha}\|_{L^2}^2.
\end{align*}
Decomposing $b=b_0+b_1$, we obtain 
\begin{align*}
     &\Big| \Re\sum_{n=1}^{\infty}\frac{\phi(t)^{2n}}{(n!)^4}\langle \partial_{\alpha}^n (b\partial_{\alpha}W_{\alpha}+b_{\alpha}W_{\alpha}), \partial_{\alpha}^n W_{\alpha}\rangle\Big|=\Big|\sum_{n=1}^{\infty}\frac{\phi(t)^{2n}}{(n!)^4}\langle \partial_{\alpha}^{n+1} (bW_{\alpha}), \partial_{\alpha}^n W_{\alpha}\rangle\Big|\\
     \leq & \Big|\sum_{n=1}^{\infty}\frac{\phi(t)^{2n}}{(n!)^4}\langle \partial_{\alpha}^{n+1} (b_0W_{\alpha}), \partial_{\alpha}^n W_{\alpha}\rangle\Big|+\Big|\sum_{n=1}^{\infty}\frac{\phi(t)^{2n}}{(n!)^4}\langle \partial_{\alpha}^{n+1} (b_1W_{\alpha}), \partial_{\alpha}^n W_{\alpha}\rangle\Big|.
     \end{align*}
Since we can estimate $\partial_{\alpha}b_1$ in $X_{\phi(t)}$, expressing $\partial_{\alpha}^{n+1}(b_1W_{\alpha})=\partial_{\alpha}^n (\partial_{\alpha}b_1W_{\alpha}+b_1\partial_{\alpha}W_{\alpha})$ and using (\ref{trilinear3}), we obtain
\begin{align*}
    &\Big|\sum_{n=1}^{\infty}\frac{\phi(t)^{2n}}{(n!)^4}\langle \partial_{\alpha}^{n+1} (b_1W_{\alpha}), \partial_{\alpha}^n W_{\alpha}\rangle\Big|\leq  \|\partial_{\alpha}b_1\|_{X_{\phi(t)}}\|\partial_{\alpha}W\|_{X_{\phi(t)}}^2.
\end{align*}
The formula for $b_0$ implies
\begin{align*}
    &\Big|\sum_{n=1}^{\infty}\frac{\phi(t)^{2n}}{(n!)^4}\langle \partial_{\alpha}^{n+1} (b_0W_{\alpha}), \partial_{\alpha}^n W_{\alpha}\rangle\Big|\\
    \leq & \Big|\sum_{n=1}^{\infty}\frac{\phi(t)^{2n}}{(n!)^4}\langle \partial_{\alpha}^{n+1} (2U W_{\alpha}), \partial_{\alpha}^n W_{\alpha}\rangle\Big|+\Big|\sum_{n=1}^{\infty}\frac{\phi(t)^{2n}}{(n!)^4}\langle \partial_{\alpha}^{n+1} (W_{\alpha}\Re[(I-\mathbb{H})U,\mathbb{H}](\frac{1}{Z_{\alpha}}-1)), \partial_{\alpha}^n W_{\alpha}\rangle\Big|\\
:=& I+\it{II}.
\end{align*}
Using Lemma \ref{producttrilinear}, we have 
\begin{align*}
    I\leq &  d_1\|U\|_{\dot{X}_{\sigma}}\|W_{\alpha}\|_{\dot{X}_{\sigma}}^2+d_1\|W_{\alpha}\|_{X_{\sigma}}\Big(\sum_{n=1}^{\infty}\frac{n^3\sigma^{2n}}{(n!)^4}\|\partial_{\alpha}^n U\|_{L^2}^2\Big)^{1/2}\Big(\sum_{n=1}^{\infty}\frac{n\sigma^{2n}}{(n!)^4}\|\partial_{\alpha}^n W_{\alpha}\|_{L^2}^2\Big)^{1/2}\\
    \leq & d_1 E_{L_0, \delta_0}^{3/2}+d_1E_{L_0, \delta_0}^{1/2}\Big(\sum_{n=1}^{\infty}\frac{n^3\sigma^{2n}}{(n!)^4}\|\partial_{\alpha}^n U\|_{L^2}^2+\sum_{n=1}^{\infty}\frac{n\sigma^{2n}}{(n!)^4}\|\partial_{\alpha}^n W_{\alpha}\|_{L^2}^2\Big),
\end{align*}
where $d_1$ is the constant in Lemma \ref{producttrilinear}.
Similarly, using Lemma \ref{producttrilinear}, (\ref{reciprocalone}) and (\ref{reciprocaltwo}), and $\partial_{\alpha}[f, \mathbb{H}]g=[f_{\alpha}, \mathbb{H}]g+[f, \mathbb{H}]g_{\alpha}$, we obtain
\begin{align*}
    \it{II}\leq d_1E_{L_0, \delta_0}^{3/2}+d_1E_{L_0, \delta_0}^{1/2}\Big(\sum_{n=1}^{\infty}\frac{n^3\sigma^{2n}}{(n!)^4}\|\partial_{\alpha}^n U\|_{L^2}^2+\sum_{n=1}^{\infty}\frac{n\sigma^{2n}}{(n!)^4}\|\partial_{\alpha}^n W_{\alpha}\|_{L^2}^2\Big).
\end{align*}
By Cauchy-Schwarz,
\begin{align*}
    \Big| \Re\sum_{n=0}^{\infty}\frac{\phi(t)^{2n}}{(n!)^4}\langle \partial_{\alpha}^n R_{\alpha}, \partial_{\alpha}^n W_{\alpha}\rangle\Big|\leq \|R_{\alpha}\|_{X_{\phi(t)}}^2+\|W_{\alpha}\|_{X_{\phi(t)}}^2.
\end{align*}
 Combining the above estimates together with (\ref{estimateb0b0}), bouding $\|b_1\|_{L^{\infty}}$ by $C\|b_1\|_{Y_{\phi(t)}}$, we obtain
\begin{align*}
   & \frac{d}{dt}E_{L_0, \delta_0}(t)\\
   &+\Big(\frac{\delta_0}{\phi(t)}-2-d_0(\|A\|_{L^{\infty}}+\|A\|_{\dot{X}_{\phi(t)}})-d_1 E_{L_0, \delta_0}^{1/2})\Big)\Big(\sum_{n=1}^{\infty}\frac{n^3\phi(t)^{2n}}{(n!)^4}\|\partial_{\alpha}^n U\|_{L^2}^2+\sum_{n=1}^{\infty}\frac{n\phi(t)^{2n}}{(n!)^4}\|\partial_{\alpha}^n W_{\alpha}\|_{L^2}^2\Big)\\
   \leq &  C(1+\|b\|_{L^{\infty}}+\|b\|_{\dot{Y}_{\phi(t)}})E_{L_0, \delta_0}(t)+\|R_{\alpha}\|_{X_{\phi(t)}}^2+d_1E_{L_0, \delta_0}^{3/2}+\|G\|_{\dot{Y}_{\phi(t)}}^2+C\|b\|_{\dot{Y}_{\phi(t)}}^2\\
   \leq & C(1+\|b_1\|_{Y_{\phi(t)}}+E_{L_0,\delta_0}(t)^{1/2})E_{L_0, \delta_0}(t)+\|R_{\alpha}\|_{X_{\phi(t)}}^2+d_1E_{L_0, \delta_0}^{3/2}+\|G\|_{\dot{Y}_{\phi(t)}}^2+C\|b_1\|_{\dot{Y}_{\phi(t)}}^2+CE_{L_0, \delta_0}(t).
\end{align*}
We choose $T_0=\min\{T, \frac{L_0}{2\delta_0}, \tau_0\}$. With this choice, $\phi(t)\geq \frac{L_0}{2}\geq 2$. Using (\ref{controldamping}), by a bootstrap argument, 
$$\frac{\delta_0}{\phi(t)}-2-d_0(\|A\|_{L^{\infty}}+\|A\|_{\dot{X}_{\phi(t)}})-d_1 E_{L_0, \delta_0}^{1/2})\geq 0, \quad \quad \forall~ t\in [0, T_0].$$
By choosing $d_1$ larger if necessary, we obtain
\begin{equation}
    \frac{d}{dt}E_{L_0, \delta_0}(t)\leq C(1+\|b_1\|_{Y_{\phi(t)}})E_{L_0, \delta_0}(t)+d_1E_{L_0, \delta_0}^{3/2}+\|R_{\alpha}\|_{X_{\phi(t)}}^2+\|G\|_{\dot{Y}_{\phi(t)}}^2+C\|b_1\|_{\dot{Y}_{\phi(t)}}^2.
\end{equation}
By the method of continuity, we obtain
\begin{equation}
    E_{L_0, \delta_0}(t)\leq E_{L_0, \delta_0}(0)e^{B_{T_0}t}+\int_0^t e^{B_{T_0}\tau}\mathcal{N}(\tau)d\tau,
\end{equation}
where $\mathcal{N}$ is defined in (\ref{mathcalN_0}).

\vspace*{2ex}
To estimate $\|U\|_{L^2}$, using energy estimates, we have 
\begin{align*}
    \frac{d}{dt}\|U(\cdot,t)\|_{L^2}^2=& 2\langle U_t, U\rangle \\
    =& 2\langle -b\partial_{\alpha}U+G+A\Lambda W, U\rangle\\
    \leq & 2\|b_{\alpha}\|_{L^{\infty}}\|U\|_{L^2}^2+\|G\|_{L^2}^2+\|U\|_{L^2}^2+\|A\|_{L^{\infty}}\|\Lambda W\|_{L^2}^2+\|A\|_{L^{\infty}}\|U\|_{L^2}^2\\
    \leq & C(E_{L_0,\delta_0}^{1/2}+\|b_1\|_{\dot{Y}_{\phi(t)}})\|U\|_{L^2}^2+\|G\|_{L^2}^2+\|U\|_{L^2}^2+\|A\|_{L^{\infty}}^2E_{L_0, \delta_0}(t)+\|A\|_{L^{\infty}}\|U\|_{L^2}^2.
\end{align*}
So we obtain
\begin{equation}
    \|U(\cdot,t)\|_{L^2}^2\leq U(\cdot,0)\|_{L^2}^2 e^{\gamma(t)}+\int_0^t  e^{\gamma(t)-\gamma(s)}(\|G(s)\|_{L^2}^2+\|A(s)\|_{L^{\infty}}E_{L_0,\delta_0}(s))ds,
\end{equation}
where
$$\gamma(t):= \int_0^t C(E_{L_0,\delta_0}^{1/2}(\tau)+\|b_1(\cdot,\tau)\|_{\dot{Y}_{\phi(\tau)}}+\|A(\cdot,t)\|_{L^{\infty}})d\tau.$$
\end{proof}

\section{Estimates}\label{sectionestimates}
In \S \ref{sectionapproximate}, we use the Picard iteration to prove \emph{Theorem} \ref{theorem1}. For each iteration, we need to solve the quasilinear system (\ref{linearsystem}) with $G, R, b_1, A$ constructed in the previous iteration. Hence we need to derive estimates for these quantities.
Most of the estimates in this section hold for the general case: without symmetry, and applies to an arbitrary number of point vortices. The only place that the symmetry plays a role is the estimates for the velocity of the point vortices, which is in Lemma \ref{velocitypointvortices}.
\subsection{ A priori assumptions}\label{apriorilocal}
\noindent We'll derive estimates under the following a priori assumptions. These assumptions are verified during the process of the Picard iteration in the next section. 

\vspace*{2ex}

\noindent Without loss of generality, we assume $0<C_0<C_1$. Denote
\begin{equation}\label{theconstantM0}
   \frac{1}{4}m_0^2:=\|U_0\|_{Y_{\phi(0)}}^2+\|\partial_{\alpha}W_0\|_{X_{\phi(0)}}^2. 
\end{equation}
Let's denote \footnote{In the definition of $M_{\lambda, h}$, we will choose $d_{I,0}$ large. We do not claim that $|d_{I,0}|^{3/2}$ is the optimal choice. The letter \emph{h} in $M_{\lambda, h}$ refers to \emph{homogeneous}. }
\begin{equation}
M_{\lambda, h}^2:=
\begin{cases}
  m_0^2+C|\lambda| (|y_0|-\frac{|\lambda|}{4\pi x(0)}T)^{-5/2}, \quad\quad &\lambda>0;\\
    m_0^2+|\lambda| d_{I,0}^{-5/2}, \quad\quad &\lambda<0.
 \end{cases}
\end{equation}
\begin{equation}
    M_{\infty}=(d_{I,0}+C|\lambda| (|y_0|-\frac{|\lambda|}{4\pi x(0)}T)^{-5/2})^{-1/2}.
\end{equation}
Here, $C>0$ is a constant depending on $C_0$ and $\frac{1}{L_0}$ only\footnote{Since we choose $L_0\geq 4$, the constant can be chosen such that it depends on $C_0$ only. At the moment let's keep track of its dependence on $\frac{1}{L_0}$.}. We will choose $T=\frac{4\pi x(0) (|y_0|-|y_0|^{9/10})}{|\lambda|}$ if $\lambda>0$ and $T=O(1)$ if $\lambda<0$. So we can take 
\begin{equation}\label{boundMlambdah}
    M_{\lambda, h}^2=\begin{cases} m_0^2+ C|y_0|^{-3/4}, \quad \quad &\lambda>0 \\ m_0^2+C|y_0|^{-1} \quad \quad \quad &\lambda<0. \end{cases}
\end{equation}

Without loss of generality, we assume 
\begin{equation}\label{boundsmallm}
\begin{cases}
  m_0^2\leq |y_0|^{-3/4}, \quad\quad &\lambda>0;\\
    m_0^2\leq |y_0|^{-1}, \quad\quad\quad &\lambda<0.
 \end{cases}
\end{equation}

\begin{definition}\label{aprioriassumptions4}
Given $W, U, z_j(t)$, we say that $(U, W, \{z_j(t)\})$ satisfies AS, if the following hold:
\begin{itemize}
    \item [(AS1)] $W\in C([0,T]; X_{\phi(t)})$, $U\in C([0,T];Y_{\phi(t)})$, $z_j(t)\in C^1([0,T];\Omega(t))$.

    \vspace*{1ex}
    
    \item [(AS2)] $\sup_{0\leq t\leq T}\Big(\|U(\cdot,t)\|_{\dot{Y}_{\phi(t)}}^2+\|W_{\alpha}\|_{X_{\phi(t)}}^2\Big)\leq M_{\lambda, h}^2\leq 1$, and $$\sup_{0\leq t\leq T}\|U(\cdot,t)\|_{L^2}^2\leq 1, \quad \quad \sup_{0\leq t\leq T}\|U(\cdot,t)\|_{L^{\infty}}\leq M_{\infty}.$$

    \vspace*{1ex}
    
    \item [(AS3)] $\frac{1}{2}C_0|\alpha-\beta|\leq |Z(\alpha,t)-Z(\beta,t)|$, \quad $\forall~t\in [0,T]$.
    
    \vspace*{1ex}
    
    \item [(AS4)]
$d_{I}(t)\geq \frac{1}{2}d_{I,0}^{9/10}\geq 1,  \quad x(t)\geq \frac{1}{2}x(0),  \quad t\in [0,T]$.

\vspace*{1ex}

\item [(AS5)] $\sup_{t\in [0,T]}\phi(t)\geq \frac{1}{2}L_0$.
\end{itemize}
\end{definition}
\noindent We derive a priori estimates under the a priori assumptions (AS1)-(AS5). Note that (AS2) implies
\begin{equation}
    \sup_{0\leq t\leq T}(\|F\|_{Y_{\phi(t)}}^2+\|Z_{\alpha}-1\|_{X_{\phi(t)}}^2)\leq 4M_{\lambda, h}^2.
\end{equation}

\vspace*{2ex}

\noindent \textbf{Convention:} If not specified, in this section, a constant $C$ will depend on $C_0$, $C_1$,  $M_{\lambda, h}$, $d_I(t)^{-1}$, and $\frac{1}{L_0}$.

\subsection{Estimates}

\begin{lemma}\label{integral}
Assume that $(W, U, \{z_j(t)\})$ satisfies AS. Then any $k\geq 2$, there holds
\begin{equation}
\int_{-\infty}^{\infty}\frac{1}{|Z(\beta,t)-z_j(t)|^k}d\beta\leq \frac{8}{C_0}d_I(t)^{-k+1}+\frac{8}{(k-1)C_0}\Big(\frac{4C_1}{d_I(t)C_0})^{k-1}.
\end{equation}
In particular, we have 
\begin{equation}
    \int_{-\infty}^{\infty}\frac{1}{|Z(\beta,t)-z_j(t)|^k}d\beta\leq \frac{16}{C_0}\Big(\frac{4C_1}{d_I(t)C_0})^{k-1}.
\end{equation}
\end{lemma}
\begin{proof}
We may assume that $d_I(t)=d(z_j(t), z(0,t))$. 
\begin{align*}
&\int_{-\infty}^{\infty}\frac{1}{|z_j(t)-Z(\beta,t)|^k}d\beta\\
=& \int_{|Z(0,t)-Z(\beta,t)|\leq 2d_I(t)}\frac{1}{|z_j(t)-Z(\beta,t)|^k}d\beta+\int_{|z(0,t)-Z(\beta,t)|\geq 2d_I(t)}\frac{1}{|z_j(t)-Z(\beta,t)|^k}d\beta\\
:=&  I+\it{II}.
\end{align*}
Denote $$E:=\{\beta: |Z(0,t)-Z(\beta,t)|\leq 2d_I(t)\}.$$
Since 
$$\frac{1}{2}C_0|\beta-0|\leq |Z(\beta,t)-Z(0,t)|,$$
we have for $\beta\in E$,
$$|\beta-0|\leq \frac{4}{C_0}d_I(t).$$
Therefore
\begin{align*}
I\leq & 8C_0^{-1} d_I(t)^{-k+1}.
\end{align*}
For $\beta\in E^c$, we have 
\begin{equation}
|Z(\beta,t)-Z(0,t)-d_I(t)|\geq |Z(\beta,t)-Z(0,t)|-d_I(t)\geq \frac{1}{2}|Z(\beta,t)-Z(0,t)|\geq \frac{C_0}{4}|\beta-0|.
\end{equation}
Also, we have 
\begin{equation}
2C_1|\beta-0|\geq |Z(\beta,t)-Z(0,t)|\geq 2d_I(t).
\end{equation}
So 
\begin{equation}
|\beta|\geq \frac{1}{C_1}d_I(t)
\end{equation}
Therefore, for $\it{II}$, we have 
\begin{align*}
\it{II}\leq & \frac{4^k}{C_0^k}\int_{|\beta|\geq \frac{1}{C_1}d_I(t) }|\beta|^{-k} d\beta=2\frac{4^k}{(k-1)C_0^k}C_1^{k-1}d_I(t)^{-k+1}=\frac{8}{(k-1)C_0}\Big(\frac{4C_1}{d_I(t)C_0}\Big)^{k-1}.
\end{align*}
\end{proof}

\begin{lemma}\label{lemmaQj}
Denote $\tilde{Q}_j:=\frac{1}{Z(\alpha,t)-z_j(t)}$. Assume that $(W, U, \{z_j(t)\})$ satisfies AS. Then
\begin{equation}
    \|\partial_{\alpha}\tilde{Q}_j\|_{X_{\phi(t)}}\leq C d_I(t)^{-3/2},
\end{equation}
for some constant $C>0$ depending on $M_{\lambda, h}, C_0$,  $C_1$, and $\frac{1}{L_0}$. 
\end{lemma}
\begin{proof}
We have $\partial_{\alpha}\tilde{Q}_j=-\frac{Z_{\alpha}-1}{(Z(\alpha,t)-z_j(t))^2}-\frac{1}{(Z(\alpha,t)-z_j(t))^2}$. Since $$\norm{\frac{Z_{\alpha}-1}{(Z(\alpha,t)-z_j(t))^2}}_{X_{\phi(t)}}\leq C\|Z_{\alpha}-1\|_{X_{\phi(t)}}\norm{\frac{1}{(Z(\alpha,t)-z_j(t))^2}}_{X_{\phi(t)}},$$ it suffices to prove $\norm{\frac{1}{(Z(\alpha,t)-z_j(t))^2}}_{X_{\phi(t)}}\leq Cd_I(t)^{-3/2}$.

Denote $Q_j:=\frac{1}{(Z(\alpha,t)-z_j(t))^2}$, and $g(\alpha):=\frac{1}{(\alpha-z_j(t))^2}$. The Faa di Bruno formua implies
\begin{align*}
    \partial_{\alpha}^n Q_j(\alpha,t)=\sum_{k=1}^n\sum_{\Lambda_{n,k}} \frac{n!}{k_1!k_2!\cdots k_n!}g^{(k)}(Z(\alpha,t),t) \prod_{j=1}^n \Big(\frac{\partial_{\alpha}^j Z(\alpha,t)}{j!}\Big)^{k_j},
\end{align*}
where 
$$\Lambda_{n,k}:=\{(k_1,...,k_n)\in (\mathbb{N}\cup\{0\})^n: \sum_{j=1}^n k_j=k, ~\sum_{j=1}^n jk_j=n\}.$$
Using (AS2) and (AS5), by Sobolev embedding, we estimate $Z_{\alpha}$ by
\begin{align*}
    |Z_{\alpha}(\alpha,t)|\leq & 1+|Z_{\alpha}-1|\leq 1+\|Z_{\alpha}-1\|_{H^1}\\
    \leq &1+(1+\frac{2}{\phi(t)})2M_{\lambda, h}\\
    \leq & 1+6 M_{\lambda, h}\\
    :=& K_0.
\end{align*}
For $2\leq j\leq n-3$, we estimate $\partial_{\alpha}^jZ(\alpha,t)$ by
\begin{equation}
\begin{split}
    |\partial_{\alpha}^j Z(\alpha,t)|\leq & \|\partial_{\alpha}^{j}(Z-\alpha)\|_{H^1}\\
    \leq & \frac{(j!)^2}{\phi^{j}(t)}(1+\frac{(j+1)^2}{\phi(t)})\|Z-\alpha\|_{\dot{Y}_{\phi(t)}}\\
    \leq & 4\frac{((j+1)!)^2}{\phi(t)^{j}}M_{\lambda, h}\leq \frac{K_0((j+1)!)^2}{\phi(t)^{j}}.
\end{split}
\end{equation}
So we obtain 
\begin{align*}
   \prod_{j=1}^{n-1} \Big(\frac{\partial_{\alpha}^j Z(\alpha,t)}{j!}\Big)^{k_j}\leq & \prod_{j=1}^n \Big(\frac{K_0 j!}{\phi(t)^{j-1}}\Big)^{k_j}\\
   =& K_0^{\sum_{j=1}^n k_j} \phi(t)^{-(\sum_{j=1}^n jk_j-\sum_{j=1}^n k_j)} \prod_{j=1}^n (j!)^{k_j}\\
   =& K_0^k \phi(t)^{-n+k}\prod_{j=1}^n (j!)^{k_j}.
\end{align*}

For $n-2\leq j\leq n$, we use 
\begin{equation}
\begin{split}
    \|\partial_{\alpha}^j Z(\alpha,t)\|_{L^2}=&\|\partial_{\alpha}^{j-1}(Z_{\alpha}-1)\|_{L^2}\leq 2\frac{((j-1)!)^2}{\phi(t)^{j-1}}M_{\lambda, h}.
\end{split}
\end{equation}
For $n\leq 5$,  there holds
\begin{equation}
    \sum_{n=0}^5 \frac{\phi(t)^{2n}}{(n!)^4}\|\partial_{\alpha}^n Q_j\|_{L^2}^2\leq  Cd_I(t)^{-3/2},
\end{equation}
for some constant $C$ depending on $C_0, C_1, M_{\lambda, h}$, and $\frac{1}{L_0}$.

Note that for $n\geq 6$ and $k\geq [n/2]$, we must have $j\leq n-3$. So for $n\geq 6$, 
\begin{align*}
    \|\partial_{\alpha}^n Q_j(\alpha,t)\|_{L^2}\leq & \sum_{k=3}^{[n/2]} \sum_{\Lambda_{n,k}}\frac{n! k!}{k_1!\cdots k_n!}K_0^k \phi(t)^{-n+k}\prod_{j=1}^n (j!)^{k_j}\|\frac{1}{|Z(\alpha,t)-z_j(t)|^{k+1}}\|_{L^2}\\
    &+\sum_{k=[n/2]+1}^{n} \sum_{\Lambda_{n,k}}\frac{n! k!}{k_1!\cdots k_n!}K_0^k \phi(t)^{-n+k}\prod_{j=1}^n (j!)^{k_j}\|\frac{1}{|Z(\alpha,t)-z_j(t)|^{k+1}}\|_{L^2}\\
    &+\Big\Vert\sum_{k=1}^2\sum_{\Lambda_{n,k}} \frac{n!}{k_1!k_2!\cdots k_n!}g^{(k)}(Z(\alpha,t),t) \prod_{j=1}^n \Big(\frac{\partial_{\alpha}^j Z(\alpha,t)}{j!}\Big)^{k_j}\Big\Vert_{L^2}\\
    :=& II_1+II_2+II_3.
\end{align*}
Assume that $n\geq 6$ and $k\geq 3$. Let $(k_1,...,k_n)\in \Lambda_{n,k}$. Note that 
\begin{equation}
    \frac{\prod_{j=2}^n (j!)^{k_j}}{n!}= (n!)^{-1}\prod_{\substack{2\leq j\leq n\\ k_j\neq 0}} (j!)^{k_j}\leq \frac{C}{n^2\prod_{j=0}^{k-1} (n-j)},
\end{equation}
for some absolute constant $C>0$. For example, we can take $C=100$.
Note that $n\geq 6$ and $1\leq j\leq [n/2]$ imply $n-j\geq \frac{1}{2}n$.
Therefore, for $3\leq k\leq [n/2]$, 
\begin{equation}\label{conconcon1}
    \prod_{j=0}^{k-1}(n-j)\geq 2^{-k}n^k.
\end{equation}
For $k\geq [n/2]+1$, we have 
\begin{equation}\label{concon2}
    \prod_{j=0}^{k-1}(n-j)\geq 2^{-[n/2]}n^{[n/2]}.
\end{equation}
Lemma \ref{integral} and (\ref{conconcon1}) imply
\begin{align*}
    II_1\leq &  C\frac{(n!)^2}{n^2\phi(t)^n}\sum_{k=3}^{[n/2]} \sum_{\Lambda_{n,k}}\frac{ k!}{k_1!\cdots k_n!}(\frac{4K_0C_1 \phi(t)}{n C_0d_I(t)})^{k}\\
    \leq & C\frac{(n!)^2}{n^2\phi(t)^n}\sum_{k=1}^{n} \sum_{\Lambda_{n,k}}\frac{ k!}{k_1!\cdots k_n!}(\frac{4K_0C_1\phi(t)}{nC_0d_I(t)})^{k}\\
       =& C\frac{(n!)^2}{n^2\phi(t)^n} \frac{4K_0C_1 \phi(t)}{nC_0d_I(t)}\Big(1+\frac{4K_0C_1\phi(t)}{nC_0d_I(t)}\Big)^n\\
   \leq & C d_I(t)^{-1}\frac{(n!)^2}{n^2\phi(t)^n},
\end{align*}
for some constant $C>0$ depending on $C_0, C_1, M_{\lambda, h}, \frac{1}{L_0}$.
Here we've used the identity
\begin{equation}
    \sum_{k=1}^n\sum_{\Lambda_{n,k}}\frac{k!}{(k_1)!\cdots (k_n)!}R^k=R(1+R)^{n-1},
\end{equation}
and the estimate
\begin{equation}
    \sup_{n\geq 1}\Big(1+\frac{4K_0C_1\phi(t)}{nC_0d_I(t)}\Big)^n\leq e^{\frac{4K_0C_1\phi(t)}{C_0d_I(t)}}\leq e^{\frac{4K_0C_1 L_0}{C_0d_I(t)}}.
\end{equation}
So we have 
\begin{equation}
    \sum_{n=0}^{\infty}\frac{(1+n^2) (\phi(t))^{2n}}{(n!)^4}II_1^2\leq Cd_I(t)^{-1},
\end{equation}
for some constant $C>0$ depending on $K_0, C_0, C_1, M_{\lambda, h},  \frac{1}{L_0}$.

For $II_2$,  using Lemma \ref{integral} and (\ref{concon2}), we obtain
\begin{align*}
    II_2\leq &  C\frac{(n!)^2}{\phi(t)^n}\sum_{k=[n/2]+1}^{n} \sum_{\Lambda_{n,k}}\frac{ k!}{k_1!\cdots k_n!}\dfrac{(\frac{2K_0C_1\phi(t)}{C_0d_I(t)})^{k}}{\prod_{j=0}^{k-1} (n-j)}\\
    \leq & C\frac{(n!)^2}{\phi(t)^n}\sum_{k=[n/2]+1}^{n} \sum_{\Lambda_{n,k}}\frac{ k!}{k_1!\cdots k_n!}\dfrac{(\frac{2K_0C_1\phi(t)}{C_0d_I(t)})^{k}}{(n/2)^{n/2}}\\
    \leq & Cd_I(t)^{-1}\frac{(n!)^2}{\phi(t)^n}\dfrac{(1+\frac{2K_0C_1\phi(t)}{C_0d_I(t)})^n}{(n/2)^{n/2}}\\
    \leq & Cd_I(t)^{-1}\frac{(n!)^2}{n^2\phi(t)^n},
\end{align*}
for some constant $C>0$ depending on $K_0, C_0, C_1, M_{\lambda, h}, \frac{1}{L_0}$. Here,  we've used the identity
\begin{equation}
    \sum_{k=1}^n\sum_{\Lambda_{n,k}}\frac{k!}{(k_1)!\cdots (k_n)!}R^k=R(1+R)^{n-1},
\end{equation}
and the estimate
\begin{equation}
    \sup_{n\geq 1}\frac{1}{n^{[n/2]-2}}\Big(1+\frac{2K_0C_1\phi(t)}{C_0d_I(t)}\Big)^n\leq C,
\end{equation}
for some constant $C>0$ depending on $K_0, C_0, C_1, M_{\lambda, h}$ and $\frac{1}{L_0}$.

So we have 
\begin{equation}
\begin{split}
 \sum_{n=1}^{\infty}\frac{(\phi(t))^{2n}}{(n!)^4}II_2^2\leq & Cd_I(t)^{-1}\sum_{n=0}^{\infty}\frac{ (\phi(t))^{2n}}{(n!)^4}\frac{(n!)^4}{\phi(t)^{2n}}\frac{1}{n^4}\leq Cd_I(t)^{-3/2}.
 \end{split}
\end{equation}

\vspace*{2ex}

For $II_3$, when $k=1$, we must have $\Lambda_{n,1}=(0,...,0,1)$. In this case,
\begin{align*}
   &\sum_{n=0}^{\infty}\frac{\phi(t)^{2n}}{(n!)^4}\Big\Vert\sum_{\Lambda_{n,1}} \frac{n!}{k_1!k_2!\cdots k_n!}g^{(k)}(Z(\alpha,t),t) \prod_{j=1}^n \Big(\frac{\partial_{\alpha}^j Z(\alpha,t)}{j!}\Big)^{k_j}\Big\Vert_{L^2}^2\\
   =& \sum_{n=0}^{\infty}\frac{\phi(t)^{2n}}{(n!)^4}\Big\Vert \frac{\partial_{\alpha}^{n-1}(Z_{\alpha}-1)}{|Z(\alpha,t)-z_j(t)|^2}\Big\Vert_{L^2}^2\\
   \leq & \sum_{n=0}^{\infty}\frac{\phi(t)^{2n}}{(n!)^4}\frac{1}{d_I(t)^2}\frac{((n-1)!)^4}{\phi(t)^{2(n-1)}}\Big(\frac{\phi(t)^{2(n-1)}}{((n-1)!)^4}\|\partial_{\alpha}^{n-1}(Z_{\alpha}-1)\|_{L^2}^2\Big)\\
   \leq & Cd_I(t)^{-2}.
\end{align*}
for some constant $C>0$ depending on $C_0, C_1, M_{\lambda, h}$, and $\frac{1}{L_0}$.

The same argument applies to $k=2$. We obtain
\begin{equation}
   \sum_{n=0}^{\infty}\frac{\phi(t)^{2n}}{(n!)^4}II_3^2\leq Cd_I(t)^{-2}.
\end{equation}
This concludes the proof of the lemma.
\end{proof}

Similarly, we can prove 
\begin{lemma}\label{estimateQ_jY}
Assume that $(W, U, \{z_j(t)\})$ satisfies AS, then 
\begin{equation}
    \Big\Vert\frac{1}{(Z(\alpha,t)-z_j(t))^m}\Big\Vert_{Y_{\phi(t)}}\leq Cd_I(t)^{-m+1/2},
\end{equation}
\begin{equation}
    \Big\Vert\frac{1}{Z(\alpha,t)-z_1(t)}-\frac{1}{Z(\alpha,t)-z_2(t)}\Big\Vert_{Y_{\phi(t)}}\leq C\frac{|z_1(t)-z_2(t)|}{d_I(t)^{3/2}},
\end{equation}
\begin{equation}
    \Big\Vert\frac{1}{Z(\alpha,t)-z_1(t)}-\frac{1}{Z(\alpha,t)-z_2(t)}\Big\Vert_{\dot{Y}_{\phi(t)}}\leq C\frac{|z_1(t)-z_2(t)|}{d_I(t)^{5/2}},
\end{equation}
for some constant $C>0$ depending on $C_0, C_1, M_{\lambda, h}$, and $\frac{1}{L_0}$. 
\end{lemma}

As a consequence,
\begin{corollary}\label{sharp}
Assume that $(W, U, \{z_j(t)\})$ satisfies AS.  Let $g\in X_{\phi(t)}$, then \footnote{We will construct solutions with $d_I(t)$ large.  Hence larger $k$ gives smaller quantity $d_I(t)^{-k}$.  } 
\begin{itemize}
\item [1.] For arbitrary $N$, let $\lambda_0:=N\max_{1\leq j\leq N}|\lambda_i|$, we have 
\begin{equation}\label{productdecay}
    \norm{\sum_{j=1}^N\frac{\lambda_j g}{(Z(\alpha,t)-z_j(t))^m}}_{X_{\phi(t)}}\leq C\lambda_0 d_I(t)^{-m}\|g\|_{X_{\phi(t)}},
\end{equation}
\item [2.] Assume $N=2$ and $\lambda_1=-\lambda_2=\lambda$, then $Q:=\frac{\lambda i}{2\pi}\frac{1}{Z(\alpha,t)-z_1(t)}-\frac{\lambda i}{2\pi}\frac{1}{Z(\alpha,t)-z_2(t)}$ satisfies
\begin{equation}\label{productdecay2}
    \|Q\|_{Y_{\phi(t)}}\leq C|\lambda||z_1(t)-z_2(t)|d_I(t)^{-3/2},
\end{equation}
\begin{equation}\label{productdecay3}
    \|Q\|_{\dot{Y}_{\phi(t)}}\leq C|\lambda| |z_1(t)-z_2(t)|d_I(t)^{-5/2},
\end{equation}
\begin{equation}\label{productdecay4}
    \|gQ\|_{X_{\phi(t)}}\leq C|\lambda| |z_1(t)-z_2(t)| d_I(t)^{-2}\|g\|_{X_{\phi(t)}},
\end{equation}
\begin{equation}\label{productdecay5}
    \|(I+\mathbb{H})Q\|_{Y_{\phi(t)}}\leq C|\lambda||z_1(t)-z_2(t)|d_I(t)^{-5/2},
\end{equation}
\end{itemize}
for some constant $C>0$ depending on $C_0, C_1, M_{\lambda, h}$, and $\frac{1}{L_0}$. 

%\begin{equation}
%    \|\frac{g}{(Z(\alpha,t)-z_j(t))^m}\|_{Y_{2,\sigma}}\leq Cd_I(t)^{-m-1/2}\|g\|_{Y_{2,\sigma}}.
%\end{equation}
\end{corollary}
\begin{proof}
 For (\ref{productdecay}), we estimate $\norm{\sum_{j=1}^N\frac{\lambda_j g}{(Z(\alpha,t)-z_j(t))^m}}_{L^2}$ by 
\begin{align*}
    \norm{\sum_{j=1}^N\frac{\lambda_j g}{(Z(\alpha,t)-z_j(t))^m}}_{L^2}\leq C\|g\|_{L^2}\norm{\sum_{j=1}^N\frac{\lambda_j g}{(Z(\alpha,t)-z_j(t))^m}}_{L^{\infty}}\leq Cd_I(t)^{-m}\|g\|_{X_{\phi(t)}}.
\end{align*}
By Lemma \ref{lemmaproduct} and Lemma \ref{estimateQ_jY},
\begin{align*}
   \norm{\sum_{j=1}^N\frac{\lambda_j g}{(Z(\alpha,t)-z_j(t))^m}}_{\dot{X}_{\phi(t)}}\leq C\lambda_0 d_I(t)^{-m-1/2}\|g\|_{X_{\phi(t)}}.
\end{align*}
So 
\begin{align*}
     \norm{\sum_{j=1}^N\frac{\lambda_j g}{(Z(\alpha,t)-z_j(t))^m}}_{X_{\phi(t)}}\leq C\lambda_0 d_I(t)^{-m-1/2}\|g\|_{X_{\phi(t)}}+Cd_I(t)^{-m}\|g\|_{X_{\phi(t)}}\leq Cd_I(t)^{-m}\|g\|_{X_{\phi(t)}}.
\end{align*}
Here, we use the assumption that $d_I(t)\geq 1$. So we obtain (\ref{productdecay}). (\ref{productdecay2})-(\ref{productdecay4}) follow from (\ref{productdecay}) and Lemma \ref{estimateQ_jY}.
\end{proof}

Using the same proof as for Lemma \ref{lemmaQj}, we obtain the following.
\begin{lemma}\label{lemmareciprocalZ}
Assume that $(W, U, \{z_j(t)\})$ satisfies AS. Then 
\begin{equation}
    \norm{\frac{1}{Z_{\alpha}}-1}_{X_{\phi(t)}}\leq CM_{\lambda, h},
\end{equation}
for some constant $C>0$ depending on $C_0$ and $M_{\lambda, h}$.
\end{lemma}

\begin{lemma}\label{holomorphic_decay}
Assume that $(W, U, \{z_j(t)\})$ satisfies AS, then
\begin{equation}\label{holodecay1}
    |\mathcal{U}(z,t)|\leq  C\min\{M_{\infty},      (1+d(z, \Sigma(t)))^{-1/2}\},
\end{equation}
\begin{equation}\label{holodecay2}
    |\mathcal{U}_z(z,t)|\leq C\min\{M_{\lambda, h}, (1+d(z, \Sigma(t)))^{-3/2}\},
\end{equation}
for some constant $C>0$ depending on $C_0, C_1$. Here, $d(z, \Sigma(t)):=\inf_{\beta\in \mathbb{R}}|z-Z(\beta,t)|$.

\end{lemma}
\begin{proof}
Recall that $\mathcal{U}$ is holomorphic in $\Omega(t)$ with boundary value $F$. So we have for $z\in \Omega(t)$,
\begin{align*}
    \mathcal{U}(z,t)=\frac{1}{2\pi i}\int_{-\infty}^{\infty}\frac{Z_{\beta}(\beta,t)}{z-Z(\beta,t)}F(\beta,t)d\beta.
\end{align*}
If $d(z,\Sigma(t))\leq 1$, then  by (AS2)
$$ |\mathcal{U}(z,t)|\leq \|\mathcal{U}(\cdot,t)\|_{L^{\infty}(\Omega(t))}\leq M_{\infty}.$$
If $d(z,\Sigma(t))\geq 1$, we have  
\begin{align*}
    |\mathcal{U}(z,t)|\leq \frac{1}{2\pi}\Big(\int_{-\infty}^{\infty}\frac{1}{|z-Z(\beta,t)|^2}d\beta\Big)^{1/2}\|Z_{\alpha}\|_{L^{\infty}}\|F\|_{L^2}\leq CM_{\lambda, h} z^{-1/2}.
\end{align*}
So we obtain (\ref{holodecay1}). Notice that $\mathcal{U}_z$  has boundary values $\frac{F_{\alpha}}{Z_{\alpha}}$. Using the same proof as for (\ref{holodecay1}), we obtain (\ref{holodecay2}).
\end{proof}

\begin{lemma}\label{velocitypointvortices}
Assume that $(W, U, \{z_j(t)\})$ satisfies AS, then 
\begin{equation}
|\dot{z}_j(t)|\leq C(d_I(t)^{-1/2}+\frac{|\lambda|}{x(t)}), \quad j=1,2.
\end{equation}

\begin{equation}\label{realvelocity}
    |\Re\{\dot{z}_j(t)\}|\leq Cd_I(t)^{-3/2}.
\end{equation}

\begin{equation}\label{velocitydifference}
    |\dot{z}_1(t)-\dot{z}_2(t)|\leq C d_I(t)^{-3/2} x(t),
\end{equation}

\end{lemma}

\begin{proof}
The main tool is the maximum principle for holomorphic functions.  Recall that for $j=1, 2$, 
$$\dot{z}_j(t)=\sum_{1\leq k\leq 2, k\neq j}\frac{\lambda_k i}{2\pi(\overline{z_j(t)-z_k(t)})}+\bar{\mathcal{U}}(z_j(t), t),$$
where $\mathcal{U}(Z,t)=\frac{1}{2\pi i}\int_{-\infty}^{\infty}\frac{Z_{\beta}(\beta,t)}{Z-Z(\beta,t)}F(\beta,t)d\beta$ is holomorphic in $\Omega(t)$ with boundary value $F$ on $\Sigma(t)$.

\vspace*{2ex}

\noindent \underline{\textbf{Estimate $\dot{z}_j$:  }} 
Clearly,
\begin{align*}
    \frac{\lambda_2}{2\pi(\overline{z_1(t)-z_2(t)})}\leq &  \frac{|\lambda|}{4\pi x(t)}.
\end{align*}
By (\ref{holodecay1}) and the assumption that $d_I(t)\geq 1$, we have 
\begin{align*}
|\bar{\mathcal{U}}(z_j(t), t)|\leq C(1+d(z_j(t), \Sigma(t)))^{-1/2}=Cd_I(t)^{-1/2}.
\end{align*}
So we obtain
\begin{equation}
    |\dot{z}_j(t)|\leq C(d_I(t)^{-1/2}+\frac{|\lambda|}{x(t)}).
\end{equation}

\vspace*{2ex}
\noindent \textbf{Estimate $|\dot{z}_1(t)-\dot{z}_2(t)|$:} By (\ref{holodecay2}) of Lemma \ref{holomorphic_decay} we have
\begin{align*}
    |\dot{z}_1(t)-\dot{z}_2(t)|=|\mathcal{U}(z_1(t),t)-\mathcal{U}(z_2(t),t)|\leq \|\mathcal{U}_z\|_{L^{\infty}(\Omega(t))}|z_1(t)-z_2(t)|\leq Cd_I(t)^{-3/2} x(t).
\end{align*}
(\ref{realvelocity}) follows immediately from 
$$\Re\{\dot{z}_j(t)\}=\Re\{\bar{\mathcal{U}}(z_j(t),t)\}\leq |\mathcal{U}_z(\tilde{x}(t)+iy(t),t)|x(t)\leq Cd_I(t)^{-3/2}.$$
\end{proof}
The following lemma follows by direct calculation.
\begin{lemma}\label{vortexvelocityxydirection}
Assume that $(W, U, \{z_j(t)\})$ satisfies AS, then for $t\in [0,T]$,
\begin{equation}
   |\frac{d}{dt}x(t)|\leq C d_I(t)^{-3/2}x(t),
\end{equation}
\begin{itemize}
\item [1.] If $\lambda<0$, then
\begin{equation}
    \frac{d}{dt}y(t)\leq -\frac{|\lambda|}{8\pi x(0)},  
\end{equation}
\item [2.] If $\lambda>0$, then 
\begin{equation}
    \frac{d}{dt}y(t)\geq \frac{|\lambda|}{8\pi x(0)}, 
\end{equation}
\end{itemize}
\end{lemma}

\begin{lemma}\label{corollaryformulaA1}
Assume that $(W, U, \{z_j(t)\})$ satisfies AS, then 
\begin{equation}\label{Aminus1inhomo}
    \|A_1-1\|_{X_{\phi(t)}}\leq C(1+|\lambda|^2 d_I(t)^{-5/2}).
\end{equation}
\begin{equation}\label{Aminusonehomo}
    \|A_1-1\|_{\dot{X}_{\phi(t)}}\leq C(M_{\lambda, h}^2+|\lambda|^2 d_I(t)^{-7/2}),
\end{equation}
\begin{equation}\label{Aminusoneinfty}
    \|A_1-1\|_{L^{\infty}}\leq C(1+\lambda^2d_I(t)^{-3})).
\end{equation}
\end{lemma}
\begin{proof}
We prove (\ref{Aminus1inhomo}) only. 
We have 
$$
         A_1=1+\frac{1}{2\pi}\int \frac{|D_tZ(\alpha,t)-D_tZ(\beta,t)|^2}{(\alpha-\beta)^2}d\beta-\sum_{j=1}^2 \frac{\lambda_j}{2\pi} Re\Big\{\Big((I-\mathbb{H})\frac{Z_{\alpha}}{(Z(\alpha,t)-z_j(t))^2}\Big)(D_tZ-\dot{z}_j(t))\Big\}.
$$
Splitting $D_tZ=\bar{F}+\bar{Q}$. Using (AS2), (\ref{holodecay1}), Lemma \ref{lemmaproduct}, Lemma \ref{commutator1} and (\ref{productdecay2}), we have 
\begin{align*}
    \norm{\int \frac{|D_tZ(\alpha,t)-D_tZ(\beta,t)|^2}{(\alpha-\beta)^2}d\beta}_{X_{\phi(t)}}\leq &\norm{\int \frac{|F(\alpha,t)-F(\beta,t)|^2}{(\alpha-\beta)^2}d\beta}_{X_{\phi(t)}}+\norm{\int \frac{|Q(\alpha,t)-Q(\beta,t)|^2}{(\alpha-\beta)^2}d\beta}_{X_{\phi(t)}}\\
    \leq &C(\|F\|_{X_{\phi(t)}}^2+\|Q\|_{X_{\phi(t)}}^2)\\
    \leq & C(1+M_{\lambda, h}^2+|\lambda|^2 x(t)^2 d_I(t)^{-3}).
\end{align*}
Similarly, we obtain
\begin{align*}
    \norm{\sum_{j=1}^2 \frac{\lambda_j i}{2\pi}\Re \Big\{\Big((I-\mathbb{H})\frac{Z_{\alpha}}{(Z(\alpha,t)-z_j(t))^2}\Big)D_tZ\Big)}_{X_{\phi(t)}}\leq C(1+M_{\lambda, h}^2+|\lambda|^2 x(t)^2 d_I(t)^{-3}).
\end{align*}
For $\norm{\sum_{j=1}^2 \frac{\lambda_j i}{2\pi}\Re \Big\{\Big((I-\mathbb{H})\frac{Z_{\alpha}}{(Z(\alpha,t)-z_j(t))^2}\Big)\dot{z}_j(t)\Big)}_{X_{\phi(t)}}$, we use $\dot{z}_j(t)=\frac{\lambda i}{4\pi x(t)}+\bar{\mathcal{U}}(z_j(t), t)$. We have 
\begin{align*}
    &\sum_{j=1}^2 \frac{\lambda_j i}{2\pi}\Re \Big\{\Big((I-\mathbb{H})\frac{Z_{\alpha}}{(Z(\alpha,t)-z_j(t))^2}\Big)\frac{\lambda i}{4\pi x(t)}\\
    =&-\frac{\lambda^2}{4\pi^2}\Re\Big\{(I-\mathbb{H})\Big(\frac{Z_{\alpha}}{(Z(\alpha,t)-z_1(t))^2(Z(\alpha,t)-z_2(t))}+\frac{Z_{\alpha}}{(Z(\alpha,t)-z_1(t))^2(Z(\alpha,t)-z_2(t))}\Big)   \Big\}
\end{align*}
Using Lemma \ref{estimateQ_jY} (and similar proofs, if necessary), we obtain

\begin{align*}
    \norm{\sum_{j=1}^2 \frac{\lambda_j i}{2\pi}\Re \Big\{\Big((I-\mathbb{H})\frac{Z_{\alpha}}{(Z(\alpha,t)-z_j(t))^2}\Big)\frac{\lambda i}{4\pi x(t)}}_{X_{\phi(t)}}\leq C|\lambda|^2 d_I(t)^{-5/2}, 
\end{align*}

\begin{align*}
    \norm{\sum_{j=1}^2 \frac{\lambda_j i}{2\pi}\Re \Big\{\Big((I-\mathbb{H})\frac{Z_{\alpha}}{(Z(\alpha,t)-z_j(t))^2}\Big)\frac{\lambda i}{4\pi x(t)}}_{L^{\infty}}\leq C|\lambda|^2 d_I(t)^{-3}, 
\end{align*}
and 
\begin{align*}
    \norm{\sum_{j=1}^2 \frac{\lambda_j i}{2\pi}\Re \Big\{\Big((I-\mathbb{H})\frac{Z_{\alpha}}{(Z(\alpha,t)-z_j(t))^2}\Big)\frac{\lambda i}{4\pi x(t)}}_{\dot{X}_{\phi(t)}}\leq C|\lambda|^2 d_I(t)^{-7/2}, 
\end{align*}
Using (\ref{holodecay1}), we obtain
\begin{align*}
    \norm{\sum_{j=1}^2 \frac{\lambda_j i}{2\pi}\Re \Big\{\Big((I-\mathbb{H})\frac{Z_{\alpha}}{(Z(\alpha,t)-z_j(t))^2}\Big)\bar{\mathcal{U}}(z_j(t),t)}_{X_{\phi(t)}}\leq C|\lambda|d_I(t)^{-2}. 
\end{align*}
So we conclude the proof of the lemma.
\end{proof}

Since $A-1=\frac{A_1}{|Z_{\alpha}|^2}-1=\frac{A_1-1}{|Z_{\alpha}|^2}-\frac{|Z_{\alpha}|^2-1}{|Z_{\alpha}|^2}$, a direct consequence of Lemma \ref{corollaryformulaA1} and Lemma \ref{reciprocal} yields the following estimates for $A-1$.
\begin{corollary}\label{estimateAminus1}
Assume that $(W, U, \{z_j(t)\})$ satisfies AS, then 
\begin{equation}
    \|A-1\|_{L^{\infty}}\leq C(1+M_{\lambda, h}^2+|\lambda|^2 d_I(t)^{-3}).
\end{equation}
\begin{equation}\label{AMINUSONEINHOMO222}
    \|A-1\|_{X_{\phi(t)}}\leq C(1+M_{\lambda, h}^2+|\lambda|^2 d_I(t)^{-5/2}).
\end{equation}
\begin{equation}
    \|A-1\|_{\dot{X}_{\phi(t)}}\leq C(M_{\lambda, h}^2+|\lambda|^2 d_I(t)^{-4}).
\end{equation}
In particular, if $m_0\leq C|y_0|^{-1/2}$, then 
\begin{equation}
    \sup_{0\leq t\leq T}\|A-1\|_{L^{\infty}}\leq  \begin{cases} C(1+|y_0|^{\frac{3}{10}}), \quad \quad &\lambda>0,\\
    C, \quad \quad &\lambda<0.
    \end{cases}
\end{equation}
and 
\begin{equation}
    \sup_{0\leq t\leq T}\|A-1\|_{\dot{X}_{\phi(t)}}\leq  \begin{cases} C|y_0|^{-3/8}, \quad \quad &\lambda>0,\\
    C|y_0|^{-1/2}, \quad \quad &\lambda<0.
    \end{cases}
\end{equation}
\end{corollary}
Note that $\|A-1\|_{L^2}\leq C|d_I(t)^{3/4}$, which is significantly larger than $\|A-1\|_{L^{\infty}}$.

\begin{lemma}\label{corollaryQ}
Assume that $(W, U, \{z_j(t)\})$ satisfies AS, then 
\begin{equation}\label{DTQY}
   \|D_tQ\|_{Y_{\phi(t)}}\leq C \frac{|\lambda|}{d_I(t)^2}+C\frac{\lambda^2}{d_I(t)^{5/2}},
\end{equation}
\begin{equation}\label{DTQY2}
   \|D_tQ\|_{\dot{Y}_{\phi(t)}}\leq CM_{\lambda, h} \frac{|\lambda|}{d_I(t)^2}+C\frac{\lambda^2}{d_I(t)^{7/2}},
\end{equation}
\begin{equation}\label{DTQY3}
    \|D_tQ\|_{L^{\infty}}\leq CM_{\lambda, \infty}\frac{|\lambda|}{d_I(t)^2}+C\frac{\lambda^2}{d_I(t)^3},
\end{equation}
for some constant $C>0$ depending on $C_0, C_1, M_{\lambda, h}, \frac{1}{L_0}$.
\end{lemma}
\begin{proof}
To prove (\ref{DTQY}), note that
\begin{equation}
    D_tQ=\sum_{j=1}^2 \frac{\lambda_j i}{2\pi}\frac{D_tZ-\dot{z}_j(t)}{(Z(\alpha,t)-z_j(t))^2}.
\end{equation}
Decomposing $D_tZ=\bar{F}+\bar{Q}$,  we estimate $D_tQ$ by 
\begin{align*}
    \|D_tQ\|_{Y_{\phi(t)}}\leq & \sum_{j=1}^2 \frac{|\lambda_j|}{2\pi}\Big\{\Big\Vert \frac{\bar{F}}{(Z(\alpha,t)-z_j(t))^2}\Big\Vert_{Y_{\phi(t)}}+\Big\Vert\frac{\bar{Q}}{(Z(\alpha,t)-z_j(t))^2}\Big\Vert_{Y_{\phi(t)}}\Big\}+\norm{\sum_{j=1}^2 \frac{\lambda_j}{2\pi}\frac{\dot{z}_j(t)}{(Z(\alpha,t)-z_j(t))^2}}_{Y_{\phi(t)}}\\
   :=& I_1+I_2+I_3.
\end{align*}
Note that Lemma \ref{estimateQ_jY},  Lemma \ref{velocitypointvortices} and (AS2) imply that
\begin{align*}
    I_1+I_2\leq C \frac{|\lambda|}{d_I(t)^2}+C\frac{\lambda^2}{d_I(t)^{7/2}}.
\end{align*}
For $I_3$, using $\dot{z}_j(t)=\frac{\lambda i}{4\pi x(t)}+\bar{\mathcal{U}}(z_j(t), t)$, we obtain 
\begin{align*}
    I_3\leq & \norm{\sum_{j=1}^2 \frac{\lambda_j}{2\pi}\frac{\frac{\lambda i}{4\pi x(t)}}{(Z(\alpha,t)-z_j(t))^2}}_{Y_{\phi(t)}}+\norm{\sum_{j=1}^2 \frac{\lambda_j}{2\pi}\frac{\bar{\mathcal{U}}(z_j(t), t)}{(Z(\alpha,t)-z_j(t))^2}}_{Y_{\phi(t)}}\\
    :=& I_{31}+I_{32}.
\end{align*}
By exploring the cancellations, we have 
\begin{align*}
    I_{31}=& \frac{\lambda^2}{4\pi^2}\norm{\frac{1}{(Z(\alpha,t)-z_1(t))^2(Z(\alpha,t)-z_2(t))}+\frac{1}{(Z(\alpha,t)-z_1(t))^2(Z(\alpha,t)-z_2(t))}}_{Y_{\phi(t)}}\\
    \leq & C\lambda^2 d_I(t)^{-5/2},
\end{align*}
and
\begin{equation}
    I_{32}\leq C |\lambda|^2 d_I(t)^{-3}.
\end{equation}
We obtain (\ref{DTQY}) by combining the estimates for $I_1, I_2, I_{31}, I_{32}$. Similarly,
\begin{align*}
    \|D_tQ\|_{L^{\infty}}\leq & \sum_{j=1}^2 \frac{|\lambda_j|}{2\pi}\Big\{\Big\Vert \frac{\bar{F}}{(Z(\alpha,t)-z_j(t))^2}\Big\Vert_{L^{\infty}}+\Big\Vert\frac{\bar{Q}}{(Z(\alpha,t)-z_j(t))^2}\Big\Vert_{L^{\infty}}\Big\}+\norm{\sum_{j=1}^2 \frac{\lambda_j}{2\pi}\frac{\dot{z}_j(t)}{(Z(\alpha,t)-z_j(t))^2}}_{L^{\infty}}\\
    \leq & CM_{\lambda, \infty}\frac{|\lambda|}{d_I(t)^2}+C\frac{\lambda^2}{d_I(t)^3}.
\end{align*}

To prove (\ref{DTQY2}), it suffices to notice that 
$$\norm{\sum_{j=1}^2 \frac{\lambda_j}{2\pi}\frac{\dot{z}_j(t)}{(Z(\alpha,t)-z_j(t))^2}}_{\dot{Y}_{\phi(t)}}\leq C\frac{\lambda^2}{d_I(t)^{7/2}}.$$
So we obtain (\ref{DTQY2}). 
\end{proof}

\begin{lemma}\label{estimateforb}
Assume that $(W, U, \{z_j(t)\})$ satisfies AS, then 
\begin{equation}\label{bbb111}
    \|b_1\|_{Y_{\phi(t)}}\leq C(1+M_{\lambda, h})|\lambda|d_{I}(t)^{-5/2}),
\end{equation}
\begin{equation}\label{bbb222}
    \|b_1\|_{\dot{Y}_{\phi(t)}}\leq CM_{\lambda, h}|\lambda|d_I(t)^{-5/2},
\end{equation}
\begin{equation}
    \|b_1\|_{L^{\infty}}\leq C(1+M_{\lambda, \infty})|\lambda| d_I(t)^{-3},
\end{equation}
for some constant $C>0$ depending on $C_0, C_1, M_{\lambda,h}$, and $\frac{1}{L_0}$.
\end{lemma}
\begin{proof}
Recall that 
\begin{equation}
    b_1=\Re\{[\bar{Q},\mathbb{H}](\frac{1}{Z_{\alpha}}-1)\}+2\Re \{(I-\mathbb{H})\bar{Q}\}.
\end{equation}
We prove (\ref{bbb111}) only. The proof of (\ref{bbb222}) is similar.
By Lemma \ref{commutator1}, Corollary \ref{sharp}, Lemma \ref{lemmareciprocalZ}, we have 
\begin{align*}
    \|b_1\|_{Y_{\phi(t)}}=&\Big\Vert \Re\{[\bar{Q},\mathbb{H}](\frac{1}{Z_{\alpha}}-1)\}+2\Re \{(I-\mathbb{H})\bar{Q}\}\Big\Vert_{Y_{\phi(t)}}\\
     \leq & \Big\Vert [\bar{Q},\mathbb{H}](\frac{1}{Z_{\alpha}}-1)\}\Big\Vert_{Y_{\phi(t)}}+2\|(I+\mathbb{H})Q\|_{Y_{\phi(t)}}\\
    \leq & C\|Q\|_{\dot{Y}_{\phi(t)}}\norm{\frac{1}{Z_{\alpha}}-1}_{X_{\phi(t)}}+2\|(I+\mathbb{H})Q\|_{Y_{\phi(t)}}\\\\
    \leq &C(1+M_{\lambda, h})|\lambda|d_{I}(t)^{-5/2},
\end{align*}
for some constant $C>0$ depending on $C_0, C_1, M_{\lambda, h}, \frac{1}{L_0}$. 
\end{proof}

The following two lemmas are consequence of the previous estimates. 
\begin{lemma}\label{estimate_Riemann_G}
Assume that $(W, U, \{z_j(t)\})$ satisfies AS, then 
\begin{itemize}
    \item [(a)] $G=G(W, U, \{z_j\})\in C([0,T]; Y_{\phi(t)})$,  and 
    \begin{equation}
        \|G(W, U, \{z_j\})\|_{\dot{Y}_{\phi(t)}}\leq CM_{\lambda, h} \frac{|\lambda|}{d_I(t)^2}+C\frac{\lambda^2}{d_I(t)^{7/2}},
    \end{equation}
    \begin{equation}
        \|G(W, U, \{z_j\})\|_{Y_{\phi(t)}}\leq C\frac{|\lambda|}{d_{I,t}^2}+C\frac{\lambda^2}{d_I(t)^{5/2}},
    \end{equation}
    \item [(b)] $R=R(W, U, \{z_j\})\in C([0,T]; X_{\phi(t)})$, and 
        \begin{equation}\label{estimateforR}
        \|\partial_{\alpha}R(W, U, \{z_j\})\|_{X_{\phi(t)}}\leq C(1+M_{\lambda, h})|\lambda| d_I(t)^{-5/2}.
    \end{equation}

\end{itemize}
\end{lemma}
\begin{proof}
Recall that $G=-\Re\{D_tQ\}$. Invoking Lemma \ref{corollaryQ},  we obtain
\begin{align*}
    \|G\|_{Y_{\phi(t)}}\leq  C \frac{|\lambda|}{d_I(t)^2}+C\frac{\lambda^2}{d_I(t)^{7/2}}.
\end{align*}
Recall that $R=\Re\{Q\}-2\Re \{(I-\mathbb{H})\bar{Q}\}-\Re\{[\bar{Q},\mathbb{H}](\frac{1}{Z_{\alpha}}-1)\}$. Then (\ref{estimateforR}) follows immediately from Lemma \ref{commutator1}, Lemma \ref{estimateforb} and Corollary \ref{estimateQ_jY}.

\end{proof}

\begin{lemma}\label{refinedcontrol}
Assume that $(W, U, \{z_j(t)\})$ satisfies AS.
\begin{itemize}
\item [1.] If $\lambda<0$, then
\begin{equation}
    d_I(t)\geq d_{I,0}+\frac{|\lambda|}{8\pi }t.
\end{equation}
\item [2.] If $\lambda>0$, then 
\begin{equation}
    d_I(t)\geq d_{I,0}-|\lambda| t.
\end{equation}
\end{itemize}
\end{lemma}
\begin{proof}
If $\lambda<0$, then we have 
\begin{align*}
    \Im\{Z(\alpha,t)-z_j(t)\}=&-|\Im\{Z(\alpha,0)-z_j(0)\}|-\int_0^t |\Im\{\frac{d}{d\tau}(Z(\alpha,\tau)-z_j(\tau))\}|d\tau\\
    \leq & -|\Im\{Z(\alpha,0)-z_j(0)\}|-\frac{|\lambda|}{8\pi x(0)}t.
\end{align*}
Therefore \footnote{Recall that we assume $d_{I,0}=\inf_{\alpha\in \mathbb{R}}\min_{j=1,2}|\Im\{Z(\alpha,0)-z_j(0)\}|$.}
\begin{align*}
   d_I(t)\geq |\inf_{\alpha\in \mathbb{R}}\Im\{Z(\alpha,t)-z_1(t)\}|\geq d_{I,0}+\frac{|\lambda|}{2\pi x(0)}
\end{align*}

If $\lambda>0$, then we have 
\begin{align*}
    \Im\{Z(\alpha,t)-z_j(t)\}=&|\Im\{Z(\alpha,0)-z_j(0)\}|-\int_0^t |\Im\{\frac{d}{d\tau}(Z(\alpha,\tau)-z_j(\tau))\}|d\tau\\
    \geq & d_{I,0}-\frac{|\lambda|}{x(0)}t.
\end{align*}
Then we have 
$$d_{I}(t)\geq d_{I,0}-\frac{|\lambda|}{x(0)}t.$$
Recall that we assume $x(0)=1$. So we conclude the proof of the lemma.
\end{proof}

\section{Proof of Theorem \ref{theorem1}}\label{sectionapproximate}

\subsection{The wellposedness of the quasilinear system (\ref{quasi2})}\label{subsectionquasi}

In this subsection we prove Theroem \ref{theorem1}.
Without loss of generality, we assume $d_{I,0}\geq 4$.
\begin{proof}[The proof of Theorem \ref{theorem1}]
Let $T>0$ to be determined, define
\begin{equation*}
S_T=\Bigg\{ \Big(W_{\alpha}, U, \{z_j\}\Big)\Bigg| \text{(AS1)-(AS5)  hold for } (W, U, \{z_j\})\Bigg\}.
\end{equation*}
So 
$$S_T\subset C([0,T];X_{\phi(t)})\times C([0,T];Y_{\phi(t)})\times \{C^1([0,T];\mathbb{C})\}.$$
We denote $D_t^{(n)}$ by $\partial_t+b^n\partial_{\alpha}$, where $b^n$ is the $n$-th approximation of $b$, which will be constructed shortly.

\vspace*{2ex}

\noindent 
\textbf{The zero-th approximation.}
We take $U^0=\Re\{F_0\}$, $\{z_j^0\}=\{z_{j,0}\}$, $Z^0=Z_0$, $W^0:=\Re\{Z_0-\alpha\}$,  $G^0:=G(W^0, U^0, \{z_j^0\})$, $R^0:=R(W^0, U^0, \{z_j^0\})$, $b_1^0:=b_1(W^0, U^0, \{z_j^0\})$, $A^0:=A(W^0, U^0, \{z_j^0\})$. For arbitrary $T>0$, $$(W^0, U^0,  \{z_j^0\})\in S_T.$$

\vspace*{2ex}

\noindent 
\textbf{The $n$-th approximation. }
Assume we have constructed $(W^n, U^n, \{z_j^n\})$  such that $$(W^{n}, U^n, \{z_j^n\})\in S_T.$$
Define $Z^{n}$ and $F^n$ by
\begin{equation}
    Z^n(\alpha,t)=\alpha+(I+\mathbb{H})W^n, \quad \quad F^n=(I+\mathbb{H})U^n.
\end{equation}

\vspace*{2ex}

\noindent \textbf{The $(n+1)$-th approximation.}
Let's construct $(W^{n+1}, U^{n+1},  \{z_j^{n+1}\})$ as follows.

\vspace*{2ex}

\noindent \textbf{Step 1.} Define 
\begin{equation}
    \begin{cases}
    G^{n}:=G(W^n, U^n, \{z_j^n(t)\}),\\
    R^{n}:=R(W^n, U^n, \{z_j^n\}),\\
    b_1^{n}:=b_1(W^n, U^n,  \{z_j^n(t)\}),  \\
    b_0^n:=\Re\{[\bar{F}^{n+1},\mathbb{H}](\frac{1}{Z^{n+1}_{\alpha}}-1)\}+2F^{n+1},\\
    b^n=b_1^n+b_0^n,\\
    A^{n}=A(W^n, U^n,  \{z_j^n(t)\}),\\
    D_t^{(n)}:=\partial_t+b^{n}\partial_{\alpha},\\
    Q^n:=-\sum_{j=1}^2 \frac{\lambda_j i}{2\pi }\dfrac{1}{Z^n(\alpha,t)-z_j^n(t)},\\
    d_{I,n}(t):=\inf_{\alpha\in \mathbb{R}}\min_{j=1,2}|Z^n(\alpha,t)-z_j(t)|,\\
    \end{cases}
\end{equation}
and define $\mathcal{U}^n$ by
$$\mathcal{U}^n(Z,t)=\frac{1}{2\pi i}\int_{-\infty}^{\infty}\frac{\partial_{\beta}Z^n(\beta,t)}{Z-Z^n(\beta,t)}F^n(\beta,t)d\beta.$$
Let $\Sigma^n(t)$ be the curve parametrized by $Z^n(\alpha,t)$, and $\Omega(t)^n$ the region bounded above by $Z^n$. Note that $b_0^n$ depends on the unknowns $F^{n+1}$ and $Z^{n+1}$.

\vspace*{1ex}

\noindent \textbf{Step 2.} $(U^{n+1}, W^{n+1})$ is defined as the solution of 
\begin{equation}\label{iteration_n1}
    \begin{cases}
    D_t^{(n)}U^{n+1}=A\Lambda W^{n+1}+G^n,\\
D^{(n)}_tW^{n+1}=-U^{n+1}-\Re [\overline{F^{n+1}},\mathbb{H}](\frac{1}{\partial_{\alpha}Z^n}-1))+R^n,\\
F^{n+1}=(I+\mathbb{H})U^{n+1},\\
    W^{n+1}(\cdot,0)=W_0,\quad \quad  U^{n+1}(\cdot,0)=U_0.
    \end{cases}
\end{equation}
Define $Z^{n+1}$ and $\{z_j^{n+1}\}$ by 
\begin{equation}\label{partialtZnplus1}
    Z^{n+1}(\alpha,t):=\alpha+(I+\mathbb{H})W^{n+1},
\end{equation}
\begin{equation}\label{pointvortexevolve}
    \begin{cases}
     \frac{d}{dt}z_j^{n+1}(t)=\overline{\mathcal{U}^{n}(z_j^{n}(t),t)}+\sum_{\substack{1\leq k\leq 2\\ k\neq j}}\frac{\lambda_k i}{2\pi} \dfrac{1}{\overline{z_k^n(t)-z_j^n(t)}},\\
     z_j^{n+1}(0)=z_j(0).
    \end{cases}
\end{equation}
We show that $(W^{n+1}, U^{n+1}, \{z_j^{n+1}\})$ satisfies (AS1)-(AS5).

\vspace*{2ex}
We choose $T$ such that
$$T=\begin{cases}\min\{T_1(C_0, \|(U_0, \partial_{\alpha}W_0)\|_{\dot{Y}_{L_0}\times X_{L_0}}), \frac{L_0}{2\delta_0},  \tau_0, \frac{4\pi x(0)(|y_0|-|y_0|^{9/10})}{|\lambda|}\}, \quad \quad &\lambda>0\\
\min\{T_2(C_0, \|(U_0, \partial_{\alpha}W_0)\|_{\dot{Y}_{L_0}\times X_{L_0}}), \tau_0, \frac{L_0}{2\delta_0}\}\quad \quad \quad &\lambda<0\end{cases}.$$
Here, $T_1, T_2$ depend continuously on its parameters.
\vspace*{2ex}

\noindent By Lemma \ref{estimateAminus1},  Lemma \ref{estimateforb},  Lemma \ref{estimate_Riemann_G}, we have $A^n-1\in C([0,T]; X_{\phi(t)})$, $b^n\in C([0,T]; Y_{\phi(t)})$, $G^n\in C([0,T]; Y_{\phi(t)})$ and $\partial_{\alpha}R^n\in C([0,T]; X_{\phi(t)})$. Moreover, for $t\in [0,T]$,
\begin{equation}\label{uniformbound}
\begin{split}
    \|b_1^n\|_{Y_{\phi(t)}}\leq & C(1+M_{\lambda, h})|\lambda|d_{I,n}(t)^{-5/2}\leq C|\lambda| d_{I,n}^{-5/2}\leq \begin{cases} C|y_0|^{-3/4}\quad \quad &\lambda>0, \\ C|y_0|^{-1}, \quad \quad &\lambda<0. \end{cases},\\
     \|A^n-1\|_{X_{\phi(t)}}\leq & C(1+|\lambda|^2 d_{I,n}(t)^{-5/2})\leq \begin{cases} C|y_0|^{3/4}, \quad \quad &\lambda>0\\
     C|y_0|^{1/2}, \quad \quad &\lambda<0.
     \end{cases},\\
          \|A^n-1\|_{L^{\infty}}\leq & C(1+|\lambda|^2 d_{I,n}(t)^{-3})\leq  \begin{cases} C(1+|y_0|^{\frac{3}{10}}), \quad \quad &\lambda>0,\\
    C, \quad \quad &\lambda<0.
    \end{cases},\\
      \|A^n-1\|_{\dot{X}_{\phi(t)}}\leq & C(M_{\lambda,h}^2+|\lambda|^2 d_{I}(t)^{-4})\leq \begin{cases} C|y_0|^{-3/4}, \quad \quad &\lambda>0\\
     C|y_0|^{-1}, \quad \quad &\lambda<0.
     \end{cases},\\
\end{split}
\end{equation}
for some constant $C$ depending only on $C_0, C_1, L_0^{-1}$ and $M_{\lambda, h}$. Since $M_{\lambda,h}\leq 1$, and $C_1\leq \|Z^n_{\alpha}\|_{L^{\infty}}\leq 1+\|Z^{n}_{\alpha}\|_{X_{\phi(t)}}\leq 1$, $L_0^{-1}\leq 1$,  we can chose $C$ depending only on $C_0$. Since we require 
$$
    \sup_{0\leq t\leq T}\Big(\frac{\delta_0}{L_0}-4-2d_0(\|A\|_{L^{\infty}}+\|A\|_{\dot{Y}_{\phi(t)}})-4d_1\|(U_0, \partial_{\alpha}W_0\|_{Y_{\phi(0)}\times X_{\phi(0)}}\Big)\geq 0,
$$
we choose $\delta_0=\begin{cases} |y_0|^{7/16}, \quad &\lambda>0\\ M_2, \quad &\lambda<0\end{cases}$. Here, $M_2$ is a large but absolute constant. Basing on these arguments, we can take
$$T=\begin{cases}  \tau,  \quad\quad\quad\quad \quad\quad\quad\quad \quad &\lambda<0\\ T=\frac{4\pi x(0)(|y_0|-|y_0|^{9/10})}{|\lambda|},  \quad &\lambda>0,\end{cases}$$
where $\tau>0$ is a small but absolute constant.

\vspace*{1ex}

By analyzing (\ref{pointvortexevolve}), using Lemma \ref{velocitypointvortices}, we obtain
\begin{equation}\label{xnplusoneone}
    |x^{n+1}(t)-x(0)|\leq  C\int_0^t |\Re\{\dot{x}^{n+1}(\tau)\}|d\tau\leq CT d_I(t)^{-3/2}\leq C|y_0|^{-1},
\end{equation}
and 
\begin{equation}\label{ynplusoneone}
    |y^{n+1}(t)-y(0)-\frac{\lambda}{4\pi x(0)}|\leq CTd_I(t)^{-3/2}\leq C|y_0|^{-1}.
\end{equation}

\vspace*{1ex}

By Theorem \ref{theoremlinear}, there is a unique solution $(U^{n+1}, \partial_{\alpha}W^{n+1})\in C([0,T]; Y_{\phi(t)})\times C([0,T]; X_{\phi(t)})$, to the system (\ref{iteration_n1}), such that 
\begin{equation}\label{evolvvelocity}
    \sup_{t\in [0,T]}(\|(U^{(n+1)}, \partial_{\alpha}W^{n+1})\|_{\dot{Y}_{\phi(t)}\times X_{\phi(t)}}^2\leq \frac{1}{4}M_{\lambda, h}^2e^{B^n T}+\int_0^T e^{B^n(T-\tau)}\mathcal{N}^n(\tau)d\tau,
\end{equation}
where by (\ref{uniformbound}),
$$B^n:=C(1+\|b_1^n\|_{C(;0,T];Y_{\phi(t))}})\leq C+C\sup_{0\leq t\leq T}|\lambda|d_{I}^{(n)}(t)^{-5/2}:=\gamma^n(T),$$
and 
\begin{align*}
    \mathcal{N}^n(t):=&\|R^n_{\alpha}\|_{X_{\phi(t)}}^2+\|G^n\|_{\dot{Y}_{\phi(t)}}^2+C\|b_1^n\|_{\dot{Y}_{\phi(t)}}^2\\
    \leq & CM_{\lambda, h}^2+C\lambda^2d_{I,n}(t))^{-3}\leq Cm_0^2+C\lambda^2 d_{I,n}(t)^{-7/2}:=\beta^n(t).
\end{align*}

\vspace*{2ex}

\noindent \textbf{Case 1: $\lambda<0$.} By Lemma \ref{refinedcontrol}, we have $d_{I,n}(t)\geq \frac{1}{2}(d_{I,0}+\frac{|\lambda|}{4\pi x(0)}t)$. 
So 
$$\gamma^n(t)\leq Cm_0+C\lambda^2(d_{I,0}+|\lambda|t)^{-3}\leq Cm_0+C\lambda^2 d_{I,0}^{-3},\quad \quad \beta^n(t)\leq Cm_0^2+C\lambda^2(d_{I,0}+|\lambda|t)^{-7/2}.$$
Denote 
$$\gamma_0:=Cm_0+C\lambda_2 d_{I,0}^{-3}.$$
So we have 
\begin{align*}
     \sup_{t\in [0,T]}(\|(U^{(n+1)}, \partial_{\alpha}W^{n+1})\|_{Y_{\phi(t)}\times X_{\phi(t)}}^2\leq & \frac{1}{4}m_0^2e^{\gamma_0 T}+C\int_0^T e^{\gamma_0 (T-\tau)}\frac{\lambda^2}{(d_{I,0}+|\lambda| t)^{7/2}} d\tau\\
     \leq &\frac{1}{4}M_{\lambda, h}^2e^{\gamma_0 T}+Ce^{\gamma_0T}\frac{|\lambda|}{d_{I,0}^{5/2}}.
\end{align*}
Choosing $T\leq \frac{1}{\gamma_0}$ and $|\lambda|^2\approx |d_{I,0}|^{3/2}\gg 1$. Then $Ce^{\gamma_0T}\frac{|\lambda|}{d_{I,0}^{5/2}}\leq \frac{C}{d_{I,0}}$. So we have 
\begin{equation}
    \sup_{t\in [0,T]}(\|(U^{(n+1)}, \partial_{\alpha}W^{n+1})\|_{\dot{Y}_{\phi(t)}\times X_{\phi(t)}}^2\leq m_0^2+C\frac{1}{d_{I,0}}.
\end{equation}
Similarly, we obtain
$$\sup_{0\leq t\leq T}\|U^{n+1}(\cdot,t)\|_{L^2}^2\leq 1, \quad \quad \sup_{0\leq t\leq T}\|U^{n+1}(\cdot,t)\|_{L^{\infty}}\leq M_{\infty}.$$

Therefore, $(W^{n+1}, U^{n+1}, \{z_j^{n+1}\})$ satisfies (AS2) and (AS5).

\vspace*{2ex}

\noindent \textbf{Case 2: $\lambda>0$.} Let's again consider the situation when $|y_0|\gg 1$. So we take $T=\frac{4\pi x(0)(|y_0|-|y_0|^{9/10})}{|\lambda|}$.

Using the similar argument, we have
\begin{equation}
    \sup_{t\in [0,T]}d_{I}(t)\geq \frac{1}{2}d_{I,0}^{9/10}, 
\end{equation}
and 
\begin{equation}
    \sup_{t\in [0,T]}(\|(U^{(n+1)}, \partial_{\alpha}W^{n+1})\|_{Y_{\phi(t)}\times X_{\phi(t)}}^2\leq m_0^2+C\frac{1}{d_{I,0}^{3/4}}.
\end{equation}
So $(W^{n+1}, U^{n+1}, \{z_j^{n+1}\})$ satisfies (AS2) and (AS5).

\vspace*{2ex}
 Let $\Sigma^{n+1}(t)$ be the curve parametrized by $Z^{n+1}(\alpha,t)$, and let $\Omega^{n+1}(t)$ be the region bounded above by $\Sigma^{n+1}(t)$. Define $\mathcal{U}^{n+1}$ by
$$\mathcal{U}^{n+1}(Z,t)=\frac{1}{2\pi i}\int_{-\infty}^{\infty}\frac{Z_{\beta}^{n+1}(\beta,t)}{Z-Z^{n+1}(\beta,t)}F^{n+1}(\beta,t)d\beta.$$
 
\noindent \underline{\textbf{Estimate $|Z^{n+1}(\alpha,t)-Z^{(n+1)}(\beta,t)|:$}}  Since $Z^{n+1}(\alpha,t)=\alpha+(I+\mathbb{H})U^{n+1}$, we have 
\begin{align*}
    \partial_t(Z^{n+1}(\alpha,t)-Z^{n+1}(\beta,t))=&[(\bar{F}^n(\alpha,t)-\bar{F}^n(\beta,t))+(\bar{Q}^n(\alpha,t)-\bar{Q}^n(\beta,t))]\\
    &-(b^n(\alpha,t)Z^{n+1}_{\alpha}(\alpha,t)-b^n(\beta,t)Z^{n+1}_{\beta}(\beta,t)).
\end{align*}
Note that 
\begin{equation}
    \begin{split}
        &\Big|[(\bar{F}^n(\alpha,t)-\bar{F}(\beta,t))+(\bar{Q}^n(\alpha,t)-\bar{Q}(\beta,t))]-(b^n(\alpha,t)Z_{\alpha}(\alpha,t)-b^n(\beta,t)Z_{\beta}(\beta,t))\Big|\\
        \leq & CM_{\lambda, h}|\alpha-\beta|.
    \end{split}
\end{equation}
So we obtain
\begin{equation}
    |Z_0(\alpha)-Z_0(\beta)|-CM_{\lambda, h}|\alpha-\beta|t\leq |Z^{n+1}(\alpha,t)-Z^{n+1}(\beta,t)|\leq |Z_0(\alpha)-Z_0(\beta)|t+CM_{\lambda, h}|\alpha-\beta|.
\end{equation}
So we obtain
\begin{equation}\label{chordardnplusone}
    \frac{1}{2}C_0|\alpha-\beta|\leq |Z^{n+1}(\alpha,t)-Z^{n+1}(\beta,t)|\quad \quad \forall~\alpha,\beta \in \mathbb{R}, \quad t\in [0, T].
\end{equation}
So $(Z^{n+1}, F^{n+1}, \{z_j^{n+1}\})$ satisfies (AS3).

\vspace*{2ex}

By (\ref{xnplusoneone}), (\ref{ynplusoneone}), and (\ref{chordardnplusone}), $(W^{n+1}, U^{n+1}, \{z_j^{n+1}\})$ satisfies (AS4). So $(W^{n+1}, U^{n+1}, \{z_j^{n+1}\})\in S_T$.

\vspace*{2ex}

\noindent \textbf{Error estimates and convergence of the approximate solutions.} We show that $\{(W^k, F^k, \{z_j^k\})\}$ is a Cauchy sequence in some Banach space. Let
\begin{equation}
    V^{k+1}:=-\Re\{[\bar{F}^{k+1},\mathbb{H}](\frac{1}{\partial_{\alpha}Z^k}-1)\}+R^k.
\end{equation}
Denote
\begin{equation}
    \begin{cases}
    \hat{U}=&U^{k+1}-U^k,\\
\hat{G}=&=-(b^k-b^{k-1})\partial_{\alpha}U^k-(A^k-A^{k-1})\Lambda W^k+G^k-G^{k-1},\\
\hat{V}=&-(b^k-b^{k-1})\partial_{\alpha}W^k+V^{k+1}-V^{k},\\
\hat{Z}=&Z^{k+1}-Z^{k},\\
\hat{W}=&W^{k+1}-W^k,\\
\hat{z}_j=& z_j^{k+1}-z_j^{k}.
    \end{cases}
\end{equation}
Then $\hat{U}$ and $\hat{W}$ satisfy
\begin{equation}
    \begin{cases}
   D_t^{(k)}\hat{U}=-A^k\Lambda \hat{W}+\hat{G},\\
   D_t^{(k)}\hat{W}=-\hat{U}+\hat{V},\\
    \hat{U}(\cdot,0)=0,\quad \hat{W}(\cdot,0)=0.
    \end{cases}
\end{equation}
And $\hat{z}_j$ satisfies
\begin{equation}
    \begin{cases}
    \frac{d}{dt}\hat{z}_j=\mathcal{U}_{error}^k+\mathcal{SP}_{error}^k,\\
    \hat{z}_j(0)=0.
    \end{cases}
\end{equation}
Here,
$$\mathcal{U}_{error}^k=\overline{\mathcal{U}^k(z_j^k(t),t)}-\overline{\mathcal{U}^{k-1}(z_j^{k-1}(t),t)},$$
and
$$\mathcal{SP}_{error}^k=\sum_{\substack{1\leq l\leq 2\\ l\neq j}}\frac{\lambda_l i}{2\pi}\Big(\dfrac{1}{\overline{z_l^{k}(t)-z_j^k(t)}}-\dfrac{1}{\overline{z_l^{k-1}(t)-z_j^{k-1}(t)}}).$$
Denote 
\begin{equation}
    \mathfrak{E}^k(t):= \|(\hat{F}, \hat{W})\|_{H^4\times H^4}^2+|\{\hat{z}_j\}|^2.
\end{equation}
It's elementary to check that 
\begin{equation}
    |\hat{z}_j(t)|^2\leq CM_{\lambda, h}^2 t\mathfrak{E}^{k-1}(t).
\end{equation}
Using Lemma \ref{estimate_Riemann_G}, we have 
\begin{equation}
    \|\hat{V}\|_{H^4}+\|\hat{G}\|_{H^{4}}\leq CM_{\lambda, h}\mathfrak{E}^{k-1}.
\end{equation}
Using energy estimates,
we obtain 
\begin{equation}\label{converge2nd}
\|(\hat{F}, \hat{H})\|_{H^4\times H^4}^2\leq CM_{\lambda, h}^2t\mathfrak{E}^{k-1}.
\end{equation}

To estimate $\hat{Z}$, using (\ref{partialtZnplus1}),  we obtain
\begin{equation}
    \|\hat{Z}\|_{H^4}\leq CM_{\lambda, h}^2t\mathfrak{E}^{k-1}.
\end{equation}
So we obtain
\begin{equation}
    \mathfrak{E}^k(t)\leq CM_{\lambda, h}^2t\mathfrak{E}^{k-1}.
\end{equation}
By (\ref{boundMlambdah}) and (\ref{boundsmallm}) and the assumption that $|y_0|$ large, we have
\begin{equation}
    \sup_{0\leq t\leq T}\mathfrak{E}^k(t)\leq c\sup_{0\leq t\leq T}\mathfrak{E}^{k-1}(t),
\end{equation}
for some constant $0<c<1$.
So $(W^n, U^n, \{z_j^n\})$ and therefore $(Z^n-\alpha, F^n, \{z_j^n\}))$ is a Cauchy sequence in $C([0,T]; H^4\times H^4\times \mathbb{C}^2)$. So 
\begin{equation}
    (W^n, U^n, \{z_j^n\})\rightarrow (W, U ,  \{z_j\})\quad 
\end{equation}
in $C([0,T]; H^4\times H^4\times \mathbb{C}^2)$. Since $(\partial_{\alpha}W^n, U^n)$ is bounded in $C([0,T]; X_{\phi(t)}\times Y_{\phi(t)})$, we have $$(W_{\alpha}, U)\in C([0,T]; X_{\phi(t)}\times Y_{\phi(t)}).$$
Also, \begin{equation}
    Q^n\rightarrow -\sum_{j=1}^2\frac{\lambda_j i}{2\pi}\frac{1}{Z(\alpha,t)-z_j(t)} \quad in ~~C([0,T]; Y_{\phi(t)}).
\end{equation}
\begin{equation}
    b^n\rightarrow b=\Re\{[D_tZ,\mathbb{H}](\frac{1}{Z_{\alpha}}-1)\}+2\Re \{Q\}+2\Re\{F\} \quad \quad \text{ in  } C([0,T]; Y_{\phi(t)}),
\end{equation}

\vspace*{2ex}

Let $\Sigma(t)$ be the curve parametrized by $Z(\alpha,t)$ andn $\Omega(t)$ the region bounded above by $\Sigma(t)$.
Then we have 
\begin{equation}
    \mathcal{U}^n(z,t)\rightarrow \frac{1}{2\pi i}\int \frac{Z_{\beta}(\beta,t)}{z-Z(\beta,t)}F(\beta,t)d\beta \quad \text{ uniformly in } \Omega(t).
\end{equation}
So we can verify that 
\begin{equation}
    \dot{z}_j(t)=\overline{\mathcal{U}(z_j(t),t)}+\sum_{\substack{1\leq k\leq 2\\k\neq j}}\frac{\lambda_k i}{2\pi}\dfrac{1}{\overline{z_k(t)-z_j(t)}},
\end{equation}
and $\{z_j(t)\}\in C^2([0,T];\Omega(t))$. So $(W ,U, \{z_j\})$ is the unique solution to (\ref{quasi2}) on $[0,T]$. So we complete the proof of the theorem.
\end{proof}

%\subsection{The proof of the theorem}\label{equivalenttwosystem} To prove Theorem \ref{theorem1}, we need to change of variables back to lagrangian coordinates. Solve 
%\begin{equation}\label{solveforh}
%    \begin{cases}
%    \frac{dh}{dt}=b(h,t),\\
%    h(\alpha,0)=\alpha.
%    \end{cases}
%\end{equation}
%By standard ODE existence and unique theorem, (\ref{solveforh}) admits a unique classical solution on $[0,T]$, and since $b\in C([0,T]; Y_{\phi(t)})$, we have $h\in C([0,T]; \phi(t))$.
%Moreover, there exists a $0<T_0\leq T$, depending on $\|b(t)\|_{H^2}$, such that 
%$$h(\alpha,t)-h(\alpha',t)\geq \frac{1}{2}(\alpha-\alpha'),$$
%for $0\leq t\leq T_0$, $\alpha<\alpha'$. and $h(\alpha,t)-\alpha\in C([0,T_0];Y_{\phi(t)})$. Let $z(\alpha,t)=Z(h(\alpha,t),t)$, $a(\alpha,t)=A(h(\alpha,t),t)h_{\alpha}(\alpha,t)$, $f(\alpha,t)=F(h(\alpha,t),t)$. Note that 
%\begin{equation}
%    \begin{cases}
%    z_{tt}-iaz_{\alpha}=-i,\\
%    \dot{z}_j=\overline{\mathcal{U}(z_j(t),t)}+\sum_{\substack{1\leq k\leq N\\k\neq j}}\frac{\lambda_j i}{2\pi}\dfrac{1}{\overline{z_k(t)-z_j(t)}},\\
%    (I-\mathfrak{H})f=0.
%    \end{cases}
%\end{equation}
%Then the proof of Theorem \ref{theorem1} is concluded.

\section{Applications: sign changing of the Taylor sign coefficient}\label{sectionfail}

In this section, we apply \emph{Theorem} \ref{theorem1} to prove \emph{Theorem} \ref{taylorsignfailtheorem2}.

Given appropriate initial data $(W_0, U_0, \{z_{1,0}, z_{2,0}\})$ and $\lambda>0$, we solve the water waves backward and forward in time. Note that the backward evolution is equivalent to solving the forward in time water waves with the same initial data $(W_0, U_, \{z_{1,0}, \{z_{2,0}\})$ but with $\lambda$ replaced by $-\lambda$. We choose $|\lambda|=O(|y_0|^{3/2})$.
\begin{itemize}
    \item [(1)]  If $\lambda>0$, the point vortices travel toward the free interface. At time $T_0=\frac{4\pi x(0)(|y_0|-|y_0|^{9/10})}{|\lambda|}$, the distance between the point vortices and the free interface is $\approx |y_0|^{9/10}$. This distance is significantly smaller than the initial one, so we can show that $\inf_{\alpha\in \mathbb{R}}A_1(\alpha, T_0)\leq -\eta_1$, for arbitrary large $\eta_1>0$, provided that we choose $|y_0|$ sufficiently large.

    \item [(2)] If $\lambda<0$, the point vortices travel away from the free interface. At time $T_0=O(1)$, the distance between the point vortices and the free interface is $\approx |y_0|^{3/2}$. This distance is significantly larger than the initial one, so we can show that $\inf_{\alpha\in \mathbb{R}}A_1(\alpha, T_0)\geq 1-\eta_0$, for arbitrary small $\eta_0>0$, provided that we choose $|y_0|$ sufficiently large.
\end{itemize}

\subsection{The Taylor sign coefficient}\label{basiccal}
Let $A_1:=A|Z_{\alpha}|^2$. we have 
\begin{proposition}\label{Taylorsign} (Corollary 4.2 in \cite{su2018long})
Let $(Z, F, \{z_j\})$ be a solution to the water waves system such that $(Z_{\alpha}-1, D_tZ)\in C([0,T_0]; H^2\times H^2)$, then we have
\begin{equation}\label{AONE}
\begin{split}
A_1=&1+\frac{1}{2\pi}\int \frac{|D_tZ(\alpha,t)-D_tZ(\beta,t)|^2}{(\alpha-\beta)^2}d\beta-Im\Big\{\sum_{j=1}^N \frac{\lambda_j i}{2\pi} \Big((I-\mathbb{H})\frac{Z_{\alpha}}{(Z(\alpha,t)-z_j(t))^2}\Big)(D_tZ-\dot{z}_j(t)) \Big\}.
\end{split}
\end{equation}
\end{proposition}

\begin{corollary}\label{goodformula}
Let $(Z, F, \{z_j\})$ be a solution to the water waves system such that $(Z_{\alpha}-1, D_tZ)\in C([0,T_0]; H^2\times H^2)$. Fix a time $t\in [0, T_0]$. Assume that $Z(\alpha,t)=\alpha$, $D_tZ=\sum_{j=1}^2\dfrac{\lambda_j i}{2\pi}\dfrac{1}{\overline{\alpha-z_j(t)}}$. Assume $z_1(t)=-x(t)+iy(t)$, $z_2(t)=x(t)+iy(t)$, with $x(t)>0$, $y(t)<0$, and $\lambda_1=-\lambda_2:=\lambda$. We simply write $x(t)$, $y(t)$ as $x$, $y$.  Then 
\begin{equation}
\begin{split}\label{complicated}
A_1(\alpha,t)=&1+\frac{\lambda^2}{\pi^2}\frac{3y\alpha^4+(x^2+y^2)y(3x^2-y^2+2\alpha^2)}{(\alpha^4+(x^2+y^2)^2+2\alpha^2(y^2-x^2))^2}\\
    & + \frac{\lambda^2}{4\pi^2}\frac{\alpha^2x^2+x^4+5x^2y^2}{((\alpha+x)^2+y^2)((\alpha-x)^2+y^2)(x^2+y^2)|y|}.\end{split}
\end{equation}
\end{corollary}
The detail calculation for (\ref{complicated}) is in \S \ref{appendixgoodformula}.

 To simplify (\ref{complicated}), we take $x\approx 1$ and $|y|\gg 1$. We will take $|\lambda|=O(|y|^{3/2})$. Direct calculation yields the following lemma.
\begin{lemma}
Let $x\leq 1$  and $|y|\gg 1$ and $\lambda^2\leq C|y|^3$, where $C>0$ is a given constant. Then 
\begin{equation}
    \sup_{\alpha\in \mathbb{R}}\frac{\lambda^2}{4\pi^2}\frac{\alpha^2x^2+x^4+5x^2y^2}{((\alpha+x)^2+y^2)((\alpha-x)^2+y^2)(x^2+y^2)|y|}\leq \frac{C'}{y^2},
\end{equation}
\begin{equation}\label{refineG1}
    \sup_{\alpha\in \mathbb{R}}\frac{\lambda^2}{\pi^2}\Big|\frac{3y\alpha^4+(x^2+y^2)y(3x^2-y^2+2\alpha^2)}{(\alpha^4+(x^2+y^2)^2+2\alpha^2(y^2-x^2))^2}-\frac{3y\alpha^4+y^3(-y^2+2\alpha^2)}{(\alpha^4+y^4+2\alpha^2y^2)^2}\Big|\leq \frac{C'}{y^2},
\end{equation}
for some constant $C'$ depending on $C$ only. In particular, we have 
\begin{equation}
    |A_1(\alpha,0)-G(\alpha; y,\lambda)|\leq \frac{C'}{y^2}, 
\end{equation}
where 
\begin{equation}
    G(\alpha; y,\lambda):=1+\frac{\lambda^2}{\pi^2}\frac{3y\alpha^4+y^3(-y^2+2\alpha^2)}{(\alpha^4+y^4+2\alpha^2y^2)^2}.
\end{equation}
\end{lemma}

\vspace*{2ex}

\noindent Denote $\gamma:=\frac{\lambda^2}{\pi^2|y|^3}$. Assume $\alpha=k|y|$, we have \footnote{Recall that $y<0$.}
\begin{equation}
    f(\gamma,k):=G(\alpha;y,\lambda)=1-\gamma g(k),
\end{equation}
where 
$$g(k)=\frac{3k^4+2k^2-1}{(k^2+1)^4}.$$
By routine calculus, we have
\begin{itemize}
\item [1.]  $g(k)$ obtains its absolute maximum $\frac{1}{4}$ at $k=\pm 1$.
\item [2.] $g(k)$ obtains its absolute minimum $-1$ at $k=0$.
\end{itemize}
So we conclude that
\begin{itemize}
    \item [i.] For $\gamma<4$, we have $f(\gamma,k)>0$ for all $k\in \mathbb{R}$.
    
    \item [ii.] For $\gamma>4$, we have $f(\gamma,\pm 1)<0$. 
    
    \item [iii.] $f(4,\pm 1)=0$, which is equivalent to $G(\pm |y|; -|y|, 2\pi|y|^{3/2})=0$.
    
\end{itemize}

\subsection{The case $\lambda_1=-\lambda_2>0$ (travel upward)}\label{caseupward}

\subsubsection{Initial data}\label{initialsection7}

Fix $|y_0|\gg 1$, $x_0=1$, and $\gamma>0$ (independent of $|y_0|$). Assume that $\lambda>0$ and let $\lambda=\gamma^{1/2}\pi |y_0|^{3/2}$. 
Let 
\begin{equation}
    z_{1}(0)=-x_0+i(y_0-1), \quad \quad z_{2}(0)=x_0+i(y_0-1).
\end{equation}
Choosing the initial $W_0$ and $U_0$ such that 
\begin{equation}
    \|(\partial_{\alpha}W_0, U_0)\|_{X_{10}\times Y_{10}}+\|W_0\|_{L^2}\leq \frac{1}{|y_0|}.
\end{equation}
We assume that $U_0$ and $W_0$ are odd functions. Denote $Q(\alpha,t)=\frac{\lambda}{2\pi i}\frac{1}{Z(\alpha,t)-z_1(t)}-\frac{\lambda}{2\pi i}\frac{1}{Z(\alpha,t)-z_2(t)}$.  We use $L_0=10$ and $\delta_0=1000$. In this case, $d_{I,0}>|y_0|$.

\subsubsection{Some basic estimates}
By Theorem \ref{theorem1}, there exists $T_0>0$ such that  (\ref{vortex_model_Riemann}) admits  a unique solution $(W, U, \{z_{1}(t), z_2(t)\})$ on $[0,T_0]$ such that
\begin{itemize}
    \item [(1)] $T_0=\frac{|y_0|-|y_0|^{9/10}}{\frac{|\lambda|}{4\pi x_0}}$. 
    
    \item [(2)] Define $Z=\alpha+(I+\mathbb{H})W$ and $F=(I+\mathbb{H})U$.  For all $t\in [0,T_0]$, 
    \begin{equation}\label{wavesmall}
        \|(Z_{\alpha}(\cdot,t)-1, F(\cdot,t))\|_{C([0,T_0];H^4\times \dot{H}^4)}\leq \frac{C}{|y_0|^{3/8}}, \quad \quad \|F\|_{C([0,T_0];L^2)}\leq 1.
    \end{equation}
    Here, $\dot{H}^4$ represents the homogeneous Sobolev space $\dot{H}^4$.
    
    \item [(3)] For each fixed $t\in [0, T_0]$, $W(\cdot, t)$ and $U(\cdot, t)$ are odd functions, and $\Re\{z_1(t)\}=-\Re\{z_2(t)\}$, $\Im\{z_1(t)\}=\Im\{z_2(t)\}$.
    
    \item [(4)] $\sup_{0\leq t\leq T_0}d_I(t)\geq \frac{1}{2}|y_0|^{9/10}$.
    
    \item [(5)] For $0\leq t\leq T_0$,
\begin{align*}
    \|Q\|_{L^{\infty}}=&\norm{\frac{\lambda}{2\pi}\Big(\frac{1}{Z(\alpha,t)-z_1(t)}-\frac{1}{Z(\alpha,t)-z_1(t)}\Big)}_{L^{\infty}}\leq C|\lambda| d_I(t)^{-2}\leq C|y_0|^{-3/10}.
\end{align*}
\end{itemize}
We simply bound $\|D_tZ(\cdot, t)\|_{L^{\infty}}$ by
\begin{equation}
   \sup_{t\in [0,T_0]} \|D_tZ(\cdot, t)\|_{L^{\infty}}\leq \sup_{0\leq t\leq T_0}(\|F(\cdot, t)\|_{L^{\infty}}+\|Q(\cdot,t)\|_{L^{\infty}})\leq 1,
\end{equation}
and simply bound $b$ by
\begin{equation}
    \sup_{0\leq t\leq T_0}\|b(\cdot,t)\|_{L^{\infty}}\leq 1.
\end{equation}
Using
\begin{align*}
    Z(\alpha,t)-\alpha=& \int_0^t \partial_{\tau}(Z(\alpha,\tau)-\alpha)d\tau\\
    =& \int_0^t D_{\tau}Z(\alpha,\tau)-b(\alpha,\tau)Z_{\alpha}(\alpha,\tau)d\tau,
    \end{align*}
we obtain
\begin{equation}\label{controlelevation}
    \|Z(\alpha,t)-\alpha\|_{H^4}\leq C\sup_{t\in [0, T_0]}(\|D_tZ(\cdot,t)\|_{L^{\infty}}+\|b(\cdot,t)\|_{L^{\infty}})t\leq C|y_0|^{-1/2}\leq C|y_0|^{-1/2}.
\end{equation}

\subsubsection{The velocity of the point vortices}
We have 
\begin{equation}\label{z1tz2t}
    \begin{split}
        &\frac{d}{dt}z_1(t)=-\frac{\lambda i}{2\pi}\dfrac{1}{\overline{z_1(t)-z_2(t)}}+\mathcal{U}(z_1(t),t)=\frac{\lambda i}{4\pi x(t)}+\mathcal{U}(z_1(t),t),\\
        &\frac{d}{dt}z_2(t)=\frac{\lambda i}{2\pi}\dfrac{1}{\overline{z_2(t)-z_1(t)}}+\mathcal{U}(z_2(t),t)=\frac{\lambda i}{4\pi x(t)}+\mathcal{U}(z_2(t),t).
    \end{split}
\end{equation}
 (\ref{holodecay1}) implies
 \begin{equation}\label{mathcalUupperbound}
     |\mathcal{U}(z_j(t),t)|\leq C d_I(t)^{-1/2}\leq C|y_0|^{-9/20}.
 \end{equation}
An application of  (\ref{realvelocity}) yields
\begin{equation}
    |\frac{d}{dt}x(t)|\leq C|d_I(t)|^{-3/2}\leq C|y_0|^{-27/20},  \quad\quad \text{for}\quad  0\leq t\leq T_0.
\end{equation}
So 
\begin{equation}\label{goodrealcontrol}
    1-C|y_0|^{-27/20}t\leq x(t)\leq 1+C|y_0|^{-27/20}t, \quad \quad 0\leq t\leq T_0.
\end{equation}
Integrating (\ref{z1tz2t}) in time, using (\ref{mathcalUupperbound}) and (\ref{goodrealcontrol}), we obtain
\begin{equation}\label{goodvortexcontrol}
\begin{cases}
    z_{1}(t)=-x_0+i(y_0+\frac{\lambda}{4\pi x_0}t)+O(t|y_0|^{-9/20}), \\ z_{2}(t)=x_0+i(y_0+\frac{\lambda}{4\pi x_0}t)+O(t|y_0|^{-9/20})
    \end{cases}
\end{equation}
In particular, at $T_0=\frac{4\pi x_0(|y_0|-|y_0|^{9/10})}{|y_0|^{3/2}}$, 
\begin{equation}
    z_1(T_0)=-|x_0|-i|y_0|^{9/10}+O(\frac{1}{|y_0|^{19/20}}), \quad \quad z_2(t_0)=|x_0|-i|y_0|^{9/10}+O(\frac{1}{|y_0|^{19/20}})
\end{equation}

\subsubsection{The Taylor sign at $T_0$}
By (\ref{AONE}), at $T_0$, 
\begin{align*}
A_1(\alpha,T_0)=&1+\frac{1}{2\pi}\int \frac{|D_tZ(\alpha,T_0)-D_tZ(\beta,T_0)|^2}{(\alpha-\beta)^2}d\beta\\
&-\Im\Big\{\sum_{j=1}^2 \frac{\lambda_j i}{2\pi} \Big((I-\mathbb{H})\frac{Z_{\alpha}}{(Z(\alpha,T_0)-z_j(t_0))^2}\Big)(D_tZ-\dot{z}_j(T_0)) \Big\}.\end{align*}
Denote
\begin{equation}
    G_1(\alpha; x, y, \lambda):=\frac{\lambda^2}{\pi^2}\frac{3y\alpha^4+(x^2+y^2)y(3x^2-y^2+2\alpha^2)}{(\alpha^4+(x^2+y^2)^2+2\alpha^2(y^2-x^2))^2}.
\end{equation}
\begin{equation}\label{formulaforG2}
    G_2(\alpha; x, y, \lambda):=\frac{\lambda^2}{4\pi^2}\frac{\alpha^2x^2+x^4+5x^2y^2}{((\alpha+x)^2+y^2)((\alpha-x)^2+y^2)(x^2+y^2)|y|}.
\end{equation}

\begin{proposition}\label{propositiontaylorsign}
Choosing the initial data as in \S \ref{initialsection7}, then
\begin{equation}
    A_1(\alpha, T_0)=1+G_1(\alpha;0, -|y_0|^{9/10},  \lambda)+O(\frac{1}{|y_0|^{1/2}}).
\end{equation}
\end{proposition}

Decomposing the integral term in (\ref{AONE}) as follows:
\begin{align*}
    \frac{1}{2\pi}\int \frac{|D_tZ(\alpha,t)-D_tZ(\beta,t)|^2}{(\alpha-\beta)^2}d\beta
=&\frac{1}{2\pi}\int \frac{|Q(\alpha,t)-Q(\beta,t)|^2}{(\alpha-\beta)^2}d\beta+\frac{1}{2\pi}\int \frac{|F(\alpha,t)-F(\beta,t)|^2}{|\alpha-\beta|^2}d\beta\\
&+\frac{1}{\pi}\int \Re \frac{(F(\alpha,t)-F(\beta,t))\overline{Q(\alpha,t)-Q(\beta,t)}}{(\alpha-\beta)^2}d\beta\\
:=& I_1(\alpha,t)+I_2(\alpha,t)+I_3(\alpha,t).
\end{align*}
First, direct calculation yields the following estimates for $I_1, I_2$, and $I_3$. 
\begin{lemma}\label{lemmapartone}
For $0\leq t\leq T_0$, we have 
\begin{equation}
   |I_1(\alpha, t)|+ |I_2(\alpha,t)|+|I_3(\alpha,t)|=O(\frac{1}{|y_0|^{1/2}}).
\end{equation}
\end{lemma}
\begin{proof}
For $I_1$, we have 
\begin{align*}
    I_1=& \int_{|\alpha|\leq |y_0|^{1/2}}\frac{|Q(\alpha,t)-Q(\beta,t)|^2}{(\alpha-\beta)^2}d\beta+\int_{|\alpha|\geq |y_0|^{1/2}}\frac{|Q(\alpha,t)-Q(\beta,t)|^2}{(\alpha-\beta)^2}d\beta:= I_{11}+I_{12}.
\end{align*}
We use (\ref{productdecay2}) to bound  $\frac{Q(\alpha,t)-Q(\beta,t)}{\alpha-\beta}$ by $$\Big|\frac{Q(\alpha,t)-Q(\beta,t)}{\alpha-\beta}\Big|\leq \|\partial_{\alpha}Q(\cdot,t)\|_{L^{\infty}}\leq C|\lambda|d_I(t)^{-5/2}\leq C|y_0|^{-3/4}.$$
Therefore,
\begin{align*}
    I_{11}\leq & \int_{|\alpha|\leq |y_0|^{1/2}}\|\partial_{\alpha}Q(\cdot,t)\|_{L^{\infty}}^2 d\alpha\\
    \leq & C|y_0|^{1/2}|y_0|^{-3/2}=C|y_0|^{-1}.
\end{align*}
For $I_{12}$,
\begin{align*}
    I_{12}\leq & 2\|Q\|_{L^{\infty}}^2\int_{|\alpha-\beta|\geq |y_0|^{1/2}}\frac{1}{(\alpha-\beta)^2}d\beta\leq C|y_0|^{-1}.
\end{align*}
So we obtain $I_1(\alpha,t)|\leq C|y_0|^{-1}$. Using the same argument, we obtain
$$I_2\leq C|y_0|^{-1/2}.$$
For $I_3$, using Cauchy-Schwarz inequality, we have $I_3\leq 2(I_1+I_2)$. So we conclude the proof of the lemma.
\end{proof}
Second, we manipulate $-Im\Big\{\sum_{j=1}^2 \frac{\lambda_j i}{2\pi} \Big((I-\mathbb{H})\frac{Z_{\alpha}}{(Z(\alpha,t)-z_j(t))^2}\Big)(D_tZ-\dot{z}_j(t)) \Big\}$. 
\begin{lemma}\label{lemmaparttwo}
For $0\leq t\leq T_0$, we have 
$$-Im\Big\{\sum_{j=1}^2 \frac{\lambda_j i}{2\pi} \Big((I-\mathbb{H})\frac{Z_{\alpha}}{(Z(\alpha,t)-z_j(t))^2}\Big)(D_tZ-\dot{z}_j(t)) \Big\}=G_1(\alpha;0, y(t),\lambda)+O(\frac{1}{|y_0|^{1/2}}).$$
\end{lemma}
\begin{proof}
Indeed, by Corollary \ref{corollaryA2},
\begin{equation}
-Im\Big\{\sum_{j=1}^2 \frac{\lambda_j i}{2\pi} \Big((I-\mathbb{H})\frac{Z_{\alpha}}{(Z(\alpha,t)-z_j(t))^2}\Big)(D_tZ-\dot{z}_j(t)) \Big\}=-\sum_{j=1}^2 \frac{\lambda_j}{\pi} Re\Big\{\frac{D_tZ-\dot{z}_j}{c_0^j(\alpha-w_0^j)^2}\Big\},
\end{equation}
where 
\begin{equation}\label{goodapproximation}
c_0^j=(\Phi^{-1})_z(\omega_0^j),\quad \quad \omega_0^j=\Phi(z_j).
\end{equation}
We claim that \footnote{(\ref{riemannalmost}) is certainly not optimal. Indeed, most of the estimates in this section are quite rough. It is not our goal to obtain optimal estimates.}
\begin{equation}\label{riemannalmost}
    c_0^j=1+O(\frac{1}{|y_0|}), \quad \quad \omega_0^j=z_j(t)+O(\frac{1}{|y_0|^{19/20}}), \quad \quad t\in [0,T_0].
\end{equation}
Indeed, since $\Phi^{-1}$ has boundary value $Z(\alpha,t)$, we have 
\begin{equation}\label{rising}
    \Phi^{-1}(z,t)-z=\frac{1}{2\pi i}\int_{-\infty}^{\infty}\frac{1}{z-\beta}(Z(\beta,t)-\beta)d\beta
\end{equation}
Expanding $\Phi(z,t)$ about $z_j(t)$, we obtain
\begin{align*}
    \omega_0^j-z_j(t)=&\Phi(z_j(t),t)-z_j(t)=\Phi(z_j(t),t)-\Phi(\Phi^{-1}(z_j(t),t),t)\\
    \leq &\|\Phi_z\|_{L^{\infty}}|\Phi^{-1}(z_j(t),t)-z_j(t)|\\
    \leq &\frac{1}{\pi}\int_{-\infty}^{\infty}\frac{1}{|z_j(t)-\beta|}|Z(\beta,t)-\beta|d\beta\\
    \leq & Cd_I(t)^{-1/2}\Big(\int_{\mathbb{R}}|Z(\beta,t)-\beta|^2d\beta\Big)^{1/2}.
\end{align*}

Using (\ref{controlelevation}) and (\ref{goodvortexcontrol}), we obtain
\begin{equation}
    |\omega_0^j-z_j(t)|\leq C(|y_0|^{9/10})^{-1/2}|y_0|^{-1/2}\leq  C|y_0|^{-19/20}.
\end{equation}
Taking $\partial_z$ on both sides of (\ref{rising}) yields
\begin{align*}
    (\Phi^{-1})_z(z,t)=1-\frac{1}{2\pi i}\int_{-\infty}^{\infty}\frac{1}{(z-\beta)^2}(Z(\beta,t)-\beta)d\beta.
\end{align*}
By (\ref{controlelevation}), Cauchy-Schwarz and $\omega_0^j=z_j(t)+O(\frac{1}{|y_0|^{19/20}})$, we obtain

\begin{align*}
   |\Phi^{-1}_z(\omega_0^j,t)-1|\leq \Big|\int_{-\infty}^{\infty}\frac{1}{(\omega_0^j(t)-\beta)^2}(Z(\beta,t)-\beta)d\beta\Big|\leq C|y_0|^{-1}.
\end{align*}
So we obtain $|c_0^j-1|\leq C|y_0|^{-1}$ and therefore verify (\ref{riemannalmost}).
\vspace*{2ex}

Decomposing $D_tZ=\bar{F}+\bar{Q}$. We use the following rough estimate
\begin{align*}
    \Big|\sum_{j=1}^2 \frac{\lambda_j }{\pi}\frac{\bar{Q}}{c_0^j(\alpha-\omega_0^j)^2}\Big|\leq C|y|^{-1}.
\end{align*}
Using the symmetry $\lambda_1=-\lambda_2$,  the estimate $\|F\|_{L^{\infty}}\leq C|y_0|^{-1/3}$ (from Theorem \ref{theorem1}), (\ref{goodapproximation}), we obtain
\begin{align*}
    \Big|\sum_{j=1}^2 \frac{\lambda_j }{\pi}\frac{\bar{F}}{c_0^j(\alpha-\omega_0^j)^2}\Big|\leq C|\lambda| d_I(t)^{-3} \|F\|_{L^{\infty}}\leq  C|y_0|^{-1}.
\end{align*}
The same argument gives
\begin{align*}
    \Big|\sum_{j=1}^2 \frac{\lambda_j }{\pi}\frac{\mathcal{U}(z_j(t),t)}{c_0^j(\alpha-\omega_0^j)^2}\Big|\leq C|y_0|^{-1}.
\end{align*}
Using (\ref{riemannalmost}) and the symmetry $\lambda_1=-\lambda_2$, it is straightforward to verify that
$$\frac{1}{c_0^j(\alpha-\omega_0^j)^2}=\frac{1}{(\alpha-z_j(t))^2}+O(\frac{1}{d_I(t)^3}|y_0|^{-19/20}).$$
using $\lambda=O(|y_0|^{3/2})$ and $d_I(t)\geq C|y_0|^{9/10}$, we have
$$\sum_{j=1}^2 \frac{\lambda_j}{\pi} Re\Big\{\frac{\frac{\lambda i}{4\pi x(t)}}{c_0^j(\alpha-w_0^j)^2}\Big\}=\sum_{j=1}^2 \frac{\lambda_j}{\pi} Re\Big\{\frac{\frac{\lambda i}{4\pi x(t)}}{(\alpha-z_j(t))^2}\Big\}+O(|y_0|^{-1/2}).$$
Since $\dot{z_j}(t)=\frac{\lambda i}{4\pi x(t)}+\bar{\mathcal{U}}(z_j(t),t)$,  we obtain
\begin{align*}
   & -Im\Big\{\sum_{j=1}^2 \frac{\lambda_j i}{2\pi} \Big((I-\mathbb{H})\frac{Z_{\alpha}}{(Z(\alpha,t)-z_j(t))^2}\Big)(D_tZ-\dot{z}_j(t)) \Big\}\\
   =&-\sum_{j=1}^2 \frac{\lambda_j}{\pi} \Re\Big\{\frac{-\dot{z}_j}{c_0^j(\alpha-w_0^j)^2}\Big\}+O(|y_0|^{-1})\\
   =& -\sum_{j=1}^2 \frac{\lambda_j}{\pi}\Re\Big\{\frac{-\frac{\lambda i}{4\pi x(t)}}{(\alpha-z_j(t))^2}\Big\}+O(|y_0|^{-1})\\
   =& \frac{\lambda^2}{\pi^2}\frac{3y(t)\alpha^4+(x(t)^2+y(t)^2)y(t)(3x(t)^2-y(t)^2+2\alpha^2)}{(\alpha^4+(x(t)^2+y(t)^2)^2+2\alpha^2(y(t)^2-x(t)^2))^2}+O(|y_0|^{-1}),
\end{align*}
where the last equality follows from 
\begin{equation}\label{calculatevortices}   -\sum_{j=1}^2 \frac{\lambda_j}{\pi}\Re\Big\{\frac{-\frac{\lambda i}{4\pi x(t)}}{(\alpha-z_j(t))^2}\Big\}=\frac{\lambda^2}{\pi^2}\frac{3y(t)\alpha^4+(x(t)^2+y(t)^2)y(t)(3x(t)^2-y(t)^2+2\alpha^2)}{(\alpha^4+(x(t)^2+y(t)^2)^2+2\alpha^2(y(t)^2-x(t)^2))^2}.
\end{equation}
For the calculation of (\ref{calculatevortices}), see
\S \ref{appendixgoodformula}. So we obtain
\begin{align*}
    &-Im\Big\{\sum_{j=1}^2 \frac{\lambda_j i}{2\pi} \Big((I-\mathbb{H})\frac{Z_{\alpha}}{(Z(\alpha,t)-z_j(t))^2}\Big)(D_tZ-\dot{z}_j(t)) \Big\}\\
    =&G_1(\alpha; x(t), y(t),\lambda)+O(|y_0|^{-1}).
\end{align*}
Finally, using (\ref{refineG1}) and $|y(t)|\geq C |y_0|^{9/10}$, we obtain
\begin{equation}\label{errorerror}
    G_1(\alpha;x(t), y(t), \lambda)=G_1(\alpha;0, y(t), \lambda)+O(\frac{1}{|y_0|^{1/2}}).
\end{equation}
So we conclude the proof of the lemma.

\end{proof}

\begin{proof}[Proof of Proposition \ref{propositiontaylorsign}]
Proposition \ref{propositiontaylorsign} follows from  Lemma \ref{lemmapartone}, $y(T_0)=-|y_0|^{9/10}$, and Lemma \ref{lemmaparttwo}.
\end{proof}

Now we are able to calculate $A_1(\cdot, T_0)$.
Note that $$1+G_1(|y_0|^{9/10}; 0, -|y_0|^{9/10}, \lambda)=G(|y_0|^{9/10}; -|y_0|^{9/10},\lambda)=1-\frac{\lambda^2}{\pi^2 (|y_0|^{9/10})^3}\times \frac{1}{4}.$$
Therefore, (\ref{errorerror}) and Proposition \ref{propositiontaylorsign}, we have 
\begin{align*}
    A_1(|y_0|^{9/10}, T_0)=&1+G_1(|y_0|^{9/10};0, -|y_0|^{9/10}, \lambda)+O(\frac{1}{|y_0|^{1/2}})\\
    =& 1-\frac{\lambda^2}{(\pi^2|y_0|^{9/10})^3}\times \frac{1}{4}+O(\frac{1}{|y_0|^{1/2}})\\
    =& 1-\frac{\gamma |y_0|^3}{4\pi^2 |y_0|^{27/10}}+O(\frac{1}{|y_0|^{1/2}})\\
    =& 1-\frac{\gamma |y_0|^{3/10}}{4\pi^2}+O(\frac{1}{|y_0|^{1/2}}).
\end{align*}
So for any given $\eta_1$, choosing $|y_0|\geq \Big(\frac{4\pi^2(\eta_1+1)}{\gamma}\Big)^{10/3}$, we have 
$$A_1\leq -\eta_1.$$

\subsection{The case $\lambda_1=-\lambda_2<0$ (Travel downward).}\label{traveldown}

To distinguish from the case that $\lambda>0$, we write the solution as $(W_-,  U_-,  \{z_{j,-}\})$, and denote the corresponding Taylor sign coefficient by $A_{1,-}$. Denote the strength of the $j$-th point vortex by $\lambda_{j,-}, j=1,2$. We take 
$$\lambda_{1,-}=-|\lambda|, \quad \quad \lambda_{2,-}=|\lambda|,$$
so the point vortices travel downward. $Z_-$ and $F_-$ are defined as
$$Z_-(\alpha, t)=\alpha+(I+\mathbb{H})W_-, \quad \quad F_-=(I+\mathbb{H})U_-.$$
\subsubsection{The initial data}
Choosing the same initial data as in \S \ref{initialsection7}: Let $x_0$, $y_0$, $U_0$, $W_0$ and $\gamma$ be the same as in \S \ref{initialsection7}. 
Let 
\begin{equation}
    z_{1,-}(0)=-x_0+i(y_0-1), \quad \quad z_{2,-}(0)=x_0+i(y_0-1).
\end{equation}
We use $L_0=10$ and $\delta_0=1000$. 

\vspace*{1ex}

By Theorem \ref{theorem1}, there exists $T_{0,-}>0$ such that  (\ref{vortex_model_Riemann}) admits  a unique solution $$(\partial_{\alpha}W_-, U_-, \{z_{1,-}(0), z_{2,-}(0)\})$$
satisfying
\begin{itemize}
    \item [(1)] $T_{0,-}=O(1)$. In particular, $T_{0,-}$ does not depend on $|y_0|$. 
    
    \item [(2)] At $T_{0,-}$,  $$d_I(T_{0,-})=O(|\lambda|)=O(|y_0|^{3/2}).$$
    
     \item [(3)] $W$ and $U$ are odd functions, and $\Re\{z_1(t)\}=-\Re\{z_2(t)\}$, $\Im\{z_1(t)\}=\Im\{z_2(t)\}$.
    
    \item [(4)] For $0\leq t\leq T_{0,-}$, we have $\frac{1}{2}\leq x(t)\leq 2$, and $y(T_{0,-})\geq |y_0|+\frac{|\lambda|}{16\pi}T_{0,-}$. Then we have 
    $$Q(T_{0,-})\leq C\frac{|\lambda|}{d_I(T_{0,-})^{2}}\leq C\frac{|\lambda|}{|\lambda|^{2}}\leq C|y_0|^{-3/2}.$$
\end{itemize}
We simply bound $D_tZ_-$ by the rough estimate
$$\|D_tZ_{-}(\cdot, T_{0,-})\|_{L^{\infty}}\leq \|F_-(\cdot,T_{0,-})\|_{L^{\infty}}+\|Q(\cdot, T_{0,-})\|_{L^{\infty}}\leq 1.$$
Using 
$$|\dot{z}_{j,-}(t)|=\frac{|\lambda|}{4\pi x_0}+O(1),$$
we obtain
    \begin{align*}
       & \Big|Im\Big\{\sum_{j=1}^2 \frac{\lambda_{j,-} i}{2\pi} \Big((I-\mathbb{H})\frac{\partial_{\alpha}Z_-}{(Z_-(\alpha,t)-z_{j,-}(t))^2}\Big)(D_tZ_- -\dot{z}_{j,-}(t))\Big\}\Big|\\
       \leq & \Big|Im\Big\{\sum_{j=1}^2 \frac{\lambda_{j,-} i}{2\pi} \Big((I-\mathbb{H})\frac{\partial_{\alpha}Z_-}{(Z_-(\alpha,t)-z_{j,-}(t))^2}\Big)\Big\}\Big| \sum_{j=1}^2 |D_tZ_- -\dot{z}_{j,-}(t)|\\
        \leq & C\frac{|\lambda|}{d_I(T_{0,-})^3}(\|F_-\|_{L^{\infty}}+\|Q\|_{L^{\infty}}+|\lambda|)\\
        \leq & C\frac{|\lambda|^2}{(|\lambda|T_{0,-})^3}\\
        \leq & C\frac{1}{|\lambda| T_{0,-}^3}.
        \end{align*}
Since $T_{0,-}$ does not depend on $|\lambda|$, for given $\eta_0\in (0,1)$, we can take $|\lambda|$ sufficiently large such that $C\frac{1}{|\lambda|T_{0,-}^3}\leq \eta_0$. So we have 
\begin{equation}
    A_{1,-}(\alpha, T_{0,-})\geq 1-\Big|Im\Big\{\sum_{j=1}^2 \frac{\lambda_j i}{2\pi} \Big((I-\mathbb{H})\frac{\partial_{\alpha}Z_-}{(Z_-(\alpha,t)-z_{j,-}(t))^2}\Big)(D_tZ_--\dot{z}_{j,-}(t))\Big|\geq 1-\eta_0.
\end{equation}

\subsection{Conclude the proof of Theorem \ref{taylorsignfailtheorem2}}
Let $\lambda>0$. Let $(W, U, \{z_j\})$ and $(W_-, U_-, \{z_{j,-}\})$ be the solution constructed in \S \ref{caseupward} and \S \ref{traveldown}, respectively.
Let $(\tilde{W},  \tilde{U},  \{\tilde{z}_{j}\})$ be the solution to (\ref{quasi3}) with initial data 
\begin{equation}
    \begin{cases}
      \tilde{W}(\cdot, -T_{0,-}):=W_-(\cdot, T_{0,-}),\\
      \tilde{U}(\cdot, -T_{0,-}):=U_-(\cdot, T_{0,-}),\\
      \tilde{z}_j(-T_{0,-}):=z_{j,-}(T_{0,-}).
    \end{cases}
\end{equation}
Since (\ref{quasi3}) is time reversible and translation invariant in time, by the uniqueness of solutions, we have 
\begin{equation}  \tilde{W}(\cdot, t)=W_-(\cdot, T_{0,-}-t), \quad \tilde{U}(\cdot, t)=U_-(\cdot, T_{0,-}-t), \quad \tilde{z}_j(t)=z_{j,-}(T_{0,-}-t).\end{equation}
In particular, 
$$\tilde{W}(\cdot, T_{0,-})=W_-(\cdot,0),\quad \tilde{U}(\cdot, T_{0,-})=U_-(\cdot, 0), \quad \tilde{z}_j(T_{0,-})=z_{j,-}(0).$$
By the uniqueness of solution again, we have 
\begin{equation}  
\tilde{W}(\cdot, t)=W(\cdot, t), \quad \tilde{U}(\cdot, t)=U(\cdot, t), \quad \tilde{z}_j(t)=z_{j}(t), \quad \quad t\in [0,T_0].\end{equation}
Denote $\tilde{A}_1$ the Taylor sign coefficient corresponding to $(\tilde{W}, \tilde{U}, \{\tilde{z}_j\})$, then\begin{itemize}
    \item [(1)] $\inf_{\alpha\in \mathbb{R}}\tilde{A}_1(\alpha,-T_{0,-})\geq 1-\eta_0$.
    
    \item [(2)] $\inf_{\alpha\in \mathbb{R}}\tilde{A}_1(\alpha, T_0)\leq -\eta_1$. 
\end{itemize}
Up to a time translation $t\mapsto t+T_{0,-}$, we conclude the proof of Theorem \ref{taylorsignfailtheorem2}.

\section*{Acknowledgement}
The author would like to thank Prof. Sijue Wu for suggesting this problem and thank Gong Chen for invaluable comments.

\appendix

\section{The Taylor sign}\label{appendixA}
\begin{lemma}\label{rational}
Let $w_1, w_2\in \mathbb{P}_-$. Then 
\begin{equation}
\int_{-\infty}^{\infty}\frac{1}{(\beta-w_1)(\beta-\overline{w_2})}d\beta=\frac{2\pi i}{\overline{w_2}-w_1}
\end{equation}
\end{lemma}
\begin{proof}
$\overline{w_2}$ is the only residue of $\frac{1}{(\beta-w_1)(\beta-\overline{w_2})}$ in $\mathbb{P}_+$. By residue Theorem, 
$$\int_{-\infty}^{\infty}\dfrac{1}{(\beta-w_1)(\beta-\overline{w_2})}d\beta=\frac{2\pi i}{\overline{w_2}-w_1}.$$
\end{proof}

\begin{corollary}\label{residuetheorem}
 Assume further that $Z(\alpha,t)=\alpha$, $\bar{Q}=\sum_{j=1}^N\dfrac{\lambda_j i}{2\pi}\dfrac{1}{\overline{\alpha-z_j(t)}}$. Then we have 
\begin{equation}
\frac{1}{2\pi}\int \frac{|\bar{Q}(\alpha,t)-\bar{Q}(\beta,t)|^2}{(\alpha-\beta)^2}d\beta=\sum_{1\leq j,k\leq N}\frac{\lambda_j\lambda_k}{(2\pi)^2}\frac{1}{(\alpha-z_j)\overline{(\alpha-z_k)}} \frac{i}{\overline{z_k}-z_j}.
\end{equation}
\end{corollary}
\begin{proof}
We have
\begin{equation}
\begin{split}
\bar{Q}(\alpha,t)-\bar{Q}(\beta,t)=\sum_{j=1}^N \frac{\lambda_j i}{2\pi}\frac{\beta-\alpha}{\overline{(\alpha-z_j)(\beta-z_j)}}.
\end{split}
\end{equation}
So we have 
\begin{equation}
\begin{split}
&\Big| \frac{\bar{Q}(\alpha,t)-\bar{Q}(\beta,t)}{\alpha-\beta}\Big|^2=\Big|\sum_{j=1}^N\frac{\lambda_j}{2\pi}\frac{1}{(\alpha-z_j)(\beta-z_j)}\Big|^2\\
=&\sum_{j=1}^N\sum_{k=1}^N\frac{\lambda_j\lambda_k}{(2\pi)^2}\frac{1}{(\alpha-z_j)(\beta-z_j)\overline{(\alpha-z_k)(\beta-z_k)}}\\
%=& \Big(2\sum_{1\leq j<k\leq N}+\sum_{1\leq j=k\leq N}\Big)\frac{\lambda_j\lambda_k}{(2\pi)^2}\frac{1}{(\alpha-z_j)(\beta-z_j)\overline{(\alpha-z_k)(\beta-z_k)}}
\end{split}
\end{equation}
Apply lemma \ref{rational}, we have 
\begin{align*}
\int_{-\infty}^{\infty}\frac{1}{(\alpha-z_j)(\beta-z_j)\overline{(\alpha-z_k)(\beta-z_k)}}d\beta=\frac{1}{(\alpha-z_j)\overline{(\alpha-z_k)}}\frac{2\pi i}{\overline{z_k}-z_j}.
\end{align*}
So we have 
\begin{align*}
\frac{1}{2\pi}\int \frac{|\bar{Q}(\alpha,t)-\bar{Q}(\beta,t)|^2}{(\alpha-\beta)^2}d\beta= & \sum_{1\leq j,k\leq N}\frac{\lambda_j\lambda_k}{(2\pi)^3}\frac{1}{(\alpha-z_j)\overline{(\alpha-z_k)}} \frac{2\pi i}{\overline{z_k}-z_j}\\
=& \sum_{1\leq j,k\leq N}\frac{\lambda_j\lambda_k}{(2\pi)^2}\frac{1}{(\alpha-z_j)\overline{(\alpha-z_k)}} \frac{i}{\overline{z_k}-z_j}.
\end{align*}
\end{proof}

\begin{lemma}
Let $z_0\in \Omega(t)$. Assume that $Z(\alpha,t)=\Phi^{-1}(\alpha,t)$, where $\Phi: \Omega(t)\rightarrow\mathbb{P}_-$ is the Riemann mapping,  we have 
\begin{equation}
(I-\mathbb{H})\frac{1}{Z(\alpha,t)-z_0}=\frac{2}{c_1(\alpha-w_0)},\quad \quad c_1=(\Phi^{-1})_z(w_0),\quad \quad w_0=\Phi(z_0,t).
\end{equation}
\end{lemma}
\begin{proof}
Note that $Z(\alpha,t)=\Phi^{-1}(\alpha,t)$. So $Z(\alpha,t)-z_0$ is the boundary value of $\Phi^{-1}(z,t)-z_0$ in the lower half plane. Since $\Phi^{-1}$ is 1-1 and onto, $\Phi^{-1}(z,t)-z_0$ has a unique  zero $w_0:=\Phi(z_0)$, so $\frac{1}{Z(\alpha,t)-z_0}$ has a exactly one pole of multiplicity one. For $z$ near $w_0$, we have 
\begin{equation}
\Phi^{-1}(z,t)-z_0=c_1(z-w_0)+\sum_{n=2}^{\infty}c_n(z-w_0)^n,\quad \quad where\quad c_1=(\Phi^{-1})_z(w_0)\neq 0.
\end{equation}
Therefore, we have $\frac{1}{Z(\alpha,t)-z_0}-\frac{1}{c_1(\alpha-w_0)}$ is holomorphic in $\mathbb{P}_-$, and hence 
\begin{equation}
(I-\mathbb{H})(\frac{1}{Z(\alpha,t)-z_0}-\frac{1}{c_1(\alpha-w_0)})=0.
\end{equation}
Since $\frac{1}{c_1(\alpha-w_0)}$ is holomorphic in $\mathbb{P}_+$, we have 
\begin{equation}
(I-\mathbb{H})\frac{1}{Z(\alpha,t)-z_0}=(I-\mathbb{H})\frac{1}{c_1(\alpha-w_0)}=\frac{2}{c_1(\alpha-w_0)}.
\end{equation}
\end{proof}

\begin{corollary}\label{corollaryA2}
Let $z_j(t)\in \Omega(t)$. Assume that $Z(\alpha,t)=\Phi^{-1}(\alpha,t)$, where $\Phi: \Omega(t)\rightarrow\mathbb{P}_-$ is the Riemann mapping, then 
\begin{equation}
(I-\mathbb{H})\frac{Z_{\alpha}}{(Z(\alpha,t)-z_j(t))^2}=\frac{2}{(\Phi^{-1})_z(\Phi(z_j(t)))(\alpha-\Phi(z_j(t)))^2}
\end{equation}
\end{corollary}
\begin{proof}
We have 
\begin{align*}
(I-\mathbb{H})\frac{Z_{\alpha}}{(Z(\alpha,t)-z_j(t))^2}=-&\partial_{\alpha}(I-\mathbb{H})\frac{1}{Z(\alpha,t)-z_j(t)}\\
=&-\partial_{\alpha}\frac{2}{(\Phi^{-1})_z(\Phi(z_j(t)))(\alpha-\Phi(z_j(t)))}\\
=&\frac{2}{(\Phi^{-1})_z(\Phi(z_j(t)))(\alpha-\Phi(z_j(t)))^2}.
\end{align*}
\end{proof}
\begin{corollary}\label{corollaryformulaA1proof}
Assume that $Z(\alpha,t)=\Phi^{-1}(\alpha,t)$, where $\Phi: \Omega(t)\rightarrow\mathbb{P}_-$ is the Riemann mapping
\begin{equation}\label{formulaA1}
A_1=1+\frac{1}{2\pi}\int \frac{|\mathcal{D}_tZ(\alpha,t)-\mathcal{D}_tZ(\beta,t)|^2}{(\alpha-\beta)^2}d\beta-\sum_{j=1}^N \frac{\lambda_j}{\pi} Re\Big\{\frac{\mathcal{D}_tZ-\dot{z}_j}{c_0^j(\alpha-w_0^j)^2}\Big\},
\end{equation}
where 
\begin{equation}
c_0^j=(\Phi^{-1})_z(\omega_0^j),\quad \quad \omega_0^j=\Phi(z_j).
\end{equation}
\end{corollary}

\subsection{The proof of Corollary \ref{goodformula}}\label{appendixgoodformula}

\begin{proof}

we calculate $-Im\{\sum_{j=1}^2\frac{\lambda_j i}{\pi}\frac{1}{(\alpha-z_j)^2}(D_tZ(\alpha)-\dot{z}_j)\}$. We have 
$$\dot{z}_j=\frac{\lambda i}{4\pi x}.$$
We have 
\begin{align*}
\sum_{j=1}^2\frac{\lambda_j i}{\pi}\frac{1}{(\alpha-z_j)^2}=& \frac{\lambda i}{\pi}(\frac{1}{(\alpha-z_1)^2}-\frac{1}{(\alpha-z_2)^2})=\frac{\lambda i}{\pi}\frac{(\alpha-z_2)^2-(\alpha-z_1)^2}{(\alpha-z_1)^2(\alpha-z_2)^2}\\
=& \frac{\lambda i}{\pi} \frac{(z_1-z_2)(2\alpha-z_1-z_2)}{(\alpha-z_1)^2(\alpha-z_2)^2}
\end{align*}
We have 
\begin{align*}
z_t(\alpha)=& \frac{\lambda i}{2\pi}(\frac{1}{\overline{\alpha-z_1}}-\frac{1}{\overline{\alpha-z_2}})=\frac{\lambda i}{2\pi}\frac{\overline{z_1-z_2}}{\overline{(\alpha-z_1)(\alpha-z_2)}}
\end{align*}
So we have 
\begin{align*}
\sum_{j=1}^2\frac{\lambda_j i}{\pi}\frac{1}{(\alpha-z_j)^2}z_t(\alpha)=&-\frac{\lambda^2}{2\pi^2}\frac{|z_1-z_2|^2(2\alpha-2yi)}{|\alpha-z_1|^2|\alpha-z_2|^2(\alpha-z_1)(\alpha-z_2)}\\
=& -\frac{\lambda^2}{2\pi^2} \frac{8x^2(\alpha-yi)\overline{(\alpha-z_1)(\alpha-z_2)}}{|\alpha-z_1|^4|\alpha-z_2|^4}\\
=&-\frac{4\lambda^2 x^2}{\pi^2}\frac{(\alpha-yi)(\alpha^2-x^2-y^2+2y\alpha i)}{((\alpha+x)^2+y^2)^2((\alpha-x)^2+y^2)^2},
\end{align*}
here, we've used 
$$(\alpha-z_1)(\alpha-z_2)=\alpha^2-\alpha(z_1+z_2)+z_1z_2=\alpha^2-x^2-y^2-2y\alpha i.$$
Use also that 
$$((\alpha+x)^2+y^2)^2((\alpha-x)^2+y^2)^2=(\alpha^4+(x^2+y^2)^2+2\alpha^2(y^2-x^2))^2,$$
we have 
\begin{align*}
-Im\{\sum_{j=1}^2\frac{\lambda_j i}{\pi}\frac{1}{(\alpha-z_j)^2}z_t(\alpha)\}=& Im\Big\{\frac{4\lambda^2 x^2}{\pi^2}\frac{(\alpha-yi)(\alpha^2-x^2-y^2+2y\alpha i)}{(\alpha^4+(x^2+y^2)^2+2\alpha^2(y^2-x^2))^2}\Big\}\\
=& \frac{4\lambda^2 x^2}{\pi^2}\frac{2y\alpha^2-y(\alpha^2-x^2-y^2)}{(\alpha^4+(x^2+y^2)^2+2\alpha^2(y^2-x^2))^2}\\
=& \frac{4\lambda^2 x^2}{\pi^2}\frac{y\alpha^2+y(x^2+y^2)}{(\alpha^4+(x^2+y^2)^2+2\alpha^2(y^2-x^2))^2}
\end{align*}
We have 
\begin{align*}
\sum_{j=1}^2\frac{\lambda_j i}{\pi}\frac{1}{(\alpha-z_j)^2}(-\dot{z}_j)=&-\frac{\lambda i}{4\pi x}\sum_{j=1}^2\frac{\lambda_j i}{\pi}\frac{1}{(\alpha-z_j)^2}\\
=& -\frac{\lambda i}{\pi} \frac{(z_1-z_2)(2\alpha-z_1-z_2)}{(\alpha-z_1)^2(\alpha-z_2)^2}\frac{\lambda i}{4\pi x}\\
=& -\frac{\lambda^2}{\pi^2}\frac{\alpha-yi}{(\alpha-z_1)^2(\alpha-z_2)^2}\\
=& -\frac{\lambda^2}{\pi^2} \frac{(\alpha-yi)\overline{(\alpha-z_1)^2(\alpha-z_2)^2}}{|\alpha-z_1|^4|\alpha-z_2|^4}\\
=&-\frac{\lambda^2}{\pi^2}\frac{(\alpha-yi)\overline{(\alpha^2-x^2-y^2-2y\alpha i)^2}}{|\alpha-z_1|^4|\alpha-z_2|^4}\\
=&-\frac{\lambda^2}{\pi^2}\frac{(\alpha-yi)((\alpha^2-x^2-y^2)^2-4y^2\alpha^2+4y\alpha(\alpha^2-x^2-y^2)i)}{(\alpha^4+(x^2+y^2)^2+2\alpha^2(y^2-x^2))^2}
\end{align*}
So 
\begin{align*}
-Im \Big\{\sum_{j=1}^2\frac{\lambda_j i}{\pi}\frac{1}{(\alpha-z_j)^2}(-\dot{z}_j) \Big\}=& \frac{\lambda^2}{\pi^2}\frac{4y\alpha^2(\alpha^2-x^2-y^2)-y((\alpha^2-x^2-y^2)^2-4y^2\alpha^2)}{(\alpha^4+(x^2+y^2)^2+2\alpha^2(y^2-x^2))^2}.
\end{align*}

So we obtain
\begin{align*}
    &-Im\{\sum_{j=1}^2\frac{\lambda_j i}{\pi}\frac{1}{(\alpha-z_j)^2}(D_tZ(\alpha)-\dot{z}_j)\}\\
    =&\frac{4\lambda^2 x^2}{\pi^2}\frac{y\alpha^2+y(x^2+y^2)}{(\alpha^4+(x^2+y^2)^2+2\alpha^2(y^2-x^2))^2}\\
    &+\frac{\lambda^2}{\pi^2}\frac{4y\alpha^2(\alpha^2-x^2-y^2)-y((\alpha^2-x^2-y^2)^2-4y^2\alpha^2)}{(\alpha^4+(x^2+y^2)^2+2\alpha^2(y^2-x^2))^2}\\
    =&\frac{\lambda^2}{\pi^2}\frac{3y\alpha^4+(x^2+y^2)y(3x^2-y^2+2\alpha^2)}{(\alpha^4+(x^2+y^2)^2+2\alpha^2(y^2-x^2))^2}
\end{align*}

\vspace*{2ex}

On the other hand, we have 
\begin{align*}
&\sum_{1\leq j,k\leq 2}\frac{\lambda_j\lambda_k}{(2\pi)^2}\frac{1}{(\alpha-z_j)\overline{(\alpha-z_k)}} \frac{i}{\overline{z_k}-z_j}\\
=&\frac{\lambda^2}{4\pi^2}\frac{1}{|\alpha-z_1|^2}\frac{i}{\bar{z}_1-z_1}+\frac{\lambda^2}{4\pi^2}\frac{1}{|\alpha-z_2|^2}\frac{i}{\bar{z}_2-z_2}-\frac{\lambda^2}{4\pi^2}\frac{1}{(\alpha-z_1)(\alpha-\bar{z}_2)}\frac{i}{\bar{z}_2-z_1}\\
&-\frac{\lambda^2}{4\pi^2}\frac{1}{(\alpha-z_2)(\alpha-\bar{z}_1)}\frac{i}{\bar{z}_1-z_2}\\
=&\frac{\lambda^2}{4\pi^2}\frac{1}{(\alpha+x)^2+y^2}\frac{1}{-2y}+\frac{\lambda^2}{4\pi^2}\frac{1}{(\alpha-x)^2+y^2}\frac{1}{-2y}\\
&+\frac{\lambda^2}{4\pi^2} \frac{(\alpha^2-x^2+y^2)y-2x^2y}{((\alpha+x)^2+y^2)((\alpha-x)^2+y^2)(x^2+y^2)}\\
=& \frac{\lambda^2}{4\pi^2}\frac{\alpha^2x^2+x^4+5x^2y^2}{((\alpha+x)^2+y^2)((\alpha-x)^2+y^2)(x^2+y^2)|y|}
\end{align*}
So we obtain
\begin{align*}
    A_1(\alpha)=&1+\frac{\lambda^2}{\pi^2}\frac{3y\alpha^4+(x^2+y^2)y(3x^2-y^2+2\alpha^2)}{(\alpha^4+(x^2+y^2)^2+2\alpha^2(y^2-x^2))^2}\\
    & + \frac{\lambda^2}{4\pi^2}\frac{\alpha^2x^2+x^4+5x^2y^2}{((\alpha+x)^2+y^2)((\alpha-x)^2+y^2)(x^2+y^2)|y|}.
\end{align*}

\end{proof}
\bibliography{qingtang}{}
\bibliographystyle{plain}
\end{document}